\documentclass{amsart}

\usepackage[OT2, T1]{fontenc}
\usepackage{url}
\usepackage{amsmath}
\usepackage{array}
\usepackage{graphicx}
\usepackage{amsfonts}
\usepackage{amssymb}
\usepackage{amstext}
\usepackage{amsthm}
\usepackage{mathrsfs}
\usepackage{colonequals}
\usepackage{enumitem}
\usepackage[all,cmtip]{xy}
\usepackage{fullpage}
\usepackage{mathtools}
\usepackage{bm}
\usepackage[alphabetic,lite]{amsrefs}
\usepackage{hyperref}
\usepackage{cleveref}
\usepackage{tikz-cd}

\numberwithin{equation}{subsection}

\newtheorem{theorem}{Theorem}
\numberwithin{theorem}{section}
\newtheorem{lemma}[theorem]{Lemma}
\newtheorem{proposition}[theorem]{Proposition}
\newtheorem{corollary}[theorem]{Corollary}

\theoremstyle{definition}
\newtheorem{defn}[theorem]{Definition}

\theoremstyle{remark}
\newtheorem{rem}[theorem]{Remark}

\newtheorem{para}[theorem]{}
\newtheorem{example}[theorem]{Example}

\newcommand{\bA}{\mathbb{A}}
\newcommand{\bC}{\mathbb{C}}

\newcommand{\bF}{\mathbb{F}}
\newcommand{\bG}{\mathbb{G}}

\newcommand{\bI}{\mathbb{I}}

\newcommand{\bP}{\mathbb{P}}
\newcommand{\bQ}{\mathbb{Q}}
\newcommand{\bR}{\mathbb{R}}

\newcommand{\bZ}{\mathbb{Z}}

\newcommand{\cA}{\mathcal{A}}

\newcommand{\cE}{\mathcal{E}}
\newcommand{\cF}{\mathcal{F}}

\newcommand{\cH}{\mathcal{H}}

\newcommand{\cO}{\mathcal{O}}

\newcommand{\cX}{\mathcal{X}}
\newcommand{\cY}{\mathcal{Y}}

\newcommand{\fm}{\mathfrak{m}}
\newcommand{\fp}{\mathfrak{p}}
\newcommand{\fq}{\mathfrak{q}}

\newcommand{\fD}{\mathfrak{D}}

\newcommand{\et}{{\text{\'et}}}

\newcommand{\dR}{_{\mathrm{dR}}}

\newcommand{\cris}{_{\mathrm{cris}}}

\newcommand*{\im}{\text{im}}
\newcommand{\Isom}{\underline{\mathrm{Isom}}}

\newcommand{\ul}[1]{\underline{\smash{#1}}}

\DeclareMathOperator{\GL}{GL}
\DeclareMathOperator{\SL}{SL}
\DeclareMathOperator{\PSL}{PSL}
\DeclareMathOperator{\GSp}{GSp}

\DeclareMathOperator{\SO}{SO}
\DeclareMathOperator{\GSpin}{GSpin}

\DeclareMathOperator{\Tr}{Tr}
\DeclareMathOperator{\Nm}{Nm}
\DeclareMathOperator{\Trd}{Trd}
\DeclareMathOperator{\trd}{trd}
\DeclareMathOperator{\nrd}{nrd}
\DeclareMathOperator{\discrd}{discrd}
\DeclareMathOperator{\disc}{disc}
\DeclareMathOperator{\Gal}{Gal}
\DeclareMathOperator{\End}{End}
\DeclareMathOperator{\Hom}{Hom}
\DeclareMathOperator{\Aut}{Aut}

\DeclareMathOperator{\Res}{Res}
\DeclareMathOperator{\Spec}{Spec}
\DeclareMathOperator{\Spf}{Spf}

\DeclareMathOperator{\Span}{Span}
\DeclareMathOperator{\rank}{rank}
\DeclareMathOperator{\der}{der}
\DeclareMathOperator{\Fil}{Fil}

\DeclareMathOperator{\Cl}{Cl}
\DeclareMathOperator{\Pic}{Pic}
\DeclareMathOperator{\Cor}{Cor}
\DeclareMathOperator{\sign}{sign}
\DeclareMathOperator{\inv}{inv}
\DeclareMathOperator{\Sh}{Sh}
\DeclareMathOperator{\Iso}{Iso}
\DeclareMathOperator{\Art}{Art}
\DeclareMathOperator{\sgn}{sign}

\begin{document}

\title[]{Infinitely many supersingular primes for some Mumford's abelian fourfolds}

\author {Fangu Chen\\
\texttt{\lowercase{fangu@berkeley.edu}}}

\begin{abstract}
    Elkies (\cites{MR903384,MR1030140}) proved the infinitude of supersingular primes for elliptic curves over real number fields. We generalize Elkies' result to some abelian fourfolds in Mumford's families (\cite{MR248146}*{\S4}), and more generally,  to certain families of Kuga-Satake abelian varieties. The proof relies on the study of local deformation spaces at closed points of the integral model of a Hodge-type Shimura variety, based on the work of Madapusi \cite{MR3484114}, and on the analysis of real points of a Shimura curve, based on the work of Shimura \cite{MR572971}. 
\end{abstract}

\maketitle

\section{Introduction}

\subsection{Background}

Serre conjectured that an abelian variety defined over a number field $K$ has ordinary reduction at a density one set of primes up to a finite extension of $K$, and proved the conjecture in the case of elliptic curves (\cite{MR644559}). Katz and Ogus proved Serre’s conjecture in the case of abelian surfaces (\cite{MR654325}*{pp. 370-372}). Recently, Hui \cite{hui2025distributionsupersingularprimesabelian} proved that a non-CM abelian variety over a number field has supersingular reduction at a density zero of primes.

A natural question is whether the density zero set of supersingular primes is finite or infinite. 
In \cites{MR903384,MR1030140}, Elkies proved that an elliptic curve defined over a number field with at least one real embedding has infinitely many primes of supersingular reduction. 
This result has since been generalized to various families of abelian varieties (\cite{MR2704678}, \cite{MR2705896}, \cite{MR2414789}, \cite{LMPT}). In all previously known cases, the moduli variety is a Shimura curve of PEL type.
In this paper, we extend Elkies’ theorem to certain families of abelian varieties parametrized by Shimura curves of Hodge type, including families of abelian fourfolds of Mumford’s type in \cite{MR248146}, which are the first examples of Shimura varieties of Hodge type but not of PEL type. 

\subsection{The main result}

The main result of this paper concerns certain families of Kuga-Satake abelian varieties constructed from K3-type Hodge structures with real multiplication by $F$ on the trace zero part of a quaternion algebra over $F$. 
\begin{theorem}\label{thm:main}
    Let $F$ be a totally real number field with narrow class number $1$ and $B$ be a quaternion algebra over $F$ unramified at all finite places and exactly one of the real places of $F$. Let $\cO$ be a maximal order of $B$ and $\cO^1$ be the group of units of $\cO$ of reduced norm $1$. Let $\cH$ be the upper half plane, and suppose the canonical model of the Shimura curve $\cO^1\backslash \cH$ is isomorphic to $\bP^1_F$. Assume \begin{enumerate}
        \item $n := [F:\bQ]$ is odd; \footnote{This follows immediately from our assumption on $B$. In the future, we may consider cases where $B$ is ramified at some finite places.}
        \item $2$ is inert in $F$;
        \item $F(\sqrt{-\epsilon_i})$ has class number $1$, where $\epsilon_i$ is a unit of $F$ that is negative at exactly one of the real place $\rho_i:F\xhookrightarrow{}\bR$;
        \item $F(\sqrt{-\epsilon_1}, \dotsc, \sqrt{-\epsilon_n})$ has class number $1$.
    \end{enumerate}
    Let $A$ be an abelian variety parametrized by the Hodge type Shimura datum $(G,X)$ (see \Cref{sec:ShimuraDatum}). Suppose the field of moduli of $A$ is in an odd-degree extension of the field of moduli of the elliptic point of order $2$,\footnote{The condition that $B$ is unramified at all finite places implies the existence of an elliptic point of order $2$. When $B$ is ramified at some finite place, analogous points may still exist based on the geometry of the corresponding Shimura curve.} then $A$ has supersingular reduction at infinitely many primes. 
\end{theorem}
\begin{rem}
    The Shimura curve has good reduction everywhere. Further work is required to remove the assumption that the quaternion algebra is unramified at all finite places and $[F:\bQ]$ is odd. The remaining assumptions are technical conditions and will be explained in \cref{sec:strategy}. 
\end{rem}

When $[F:\bQ] = 3$, Galluzzi \cite{MR1909819} showed that the abelian variety obtained via Kuga-Satake construction is isogenous to powers of abelian fourfold of Mumford's type. 

\begin{theorem}\label{thm:Mumfordfourfold}
    Let $F$ be a totally real cubic number field with $\disc(F) \in \{49, 81, 169, 321, 361, \allowbreak 473, 785, 993\}$ and $B$ be a quaternion algebra over $F$ unramified at all finite places and exactly one of the real places of $F$. Let $X$ be an abelian fourfold in the one-dimensional family defined in \cite{MR248146}*{\S4} by $B$. Suppose $X$ has field of moduli $F$, then $X$ has supersingular reduction at infinitely many primes. 
\end{theorem}

More examples satisfying all the assumptions of \Cref{thm:main} can be found in \Cref{examples}. 

\subsection{Related works and heuristics}
Elkies' result has been generalized to some families of abelian surfaces with quaternionic multiplication (see  \cite{MR2414789},\cite{MR2704678} for the case of discriminant $6$, and \cite{MR2705896} for the case of discriminant $21, 33$), and some abelian fourfolds with an action of $\mu_5$(\cite{LMPT}). In all known cases, the coarse moduli variety is isomorphic to $\bP^1$. Heuristically, following the philosophy of \cite{MR3841493}, for an abelian variety $A$ on a Shimura variety $S$, 
the probability that $A \bmod\fp$ lies in the Hecke orbit of a codimension $d$ subvariety $V \subset S_{\bF_{\fp}}$ is roughly $(N\fp)^{-d/2}$. Since \[\sum_{p \leq x} p^{-d/2} \sim \begin{cases}
    \sqrt{x}/\log x & d = 1, \\ \log\log x & d = 2, \\ 1 & d\geq 3,
\end{cases}\] we expect infinitely many supersingular primes when the Shimura variety has dimension $1$ and its supersingular locus has codimension $1$.

There are similar results showing that certain density zero set of primes related to the reduction of abelian varieties is infinite, as in the case of split reduction of abelian surfaces (\cites{MR4065146,MR4490194, MR4836458}), and geometrically isogenous reductions of non-isogenous elliptic curves (\cite{MR3843371}). The proofs of these theorems rely on the intersection of the given arithmetic $1$-cycle with the reductions of divisors defined in characteristic $0$. By contrast, the supersingular locus varies from prime to prime. We follow Elkies' strategy to use CM cycles, which have supersingular reduction at roughly half of the primes, and we need to detect whether the intersection occurs at the supersingular primes. For this, we work with $\bP^1$, where the intersection theory is especially simple.

\subsection{The strategy of the proof}\label{sec:strategy}

Given an elliptic curve $E$ defined over a number field, Elkies' strategy to construct a new supersingular prime $\fp$ for $E$ is to find a CM cycle whose $\overline{\bQ}$-points are elliptic curves with complex multiplication by $\mathbb{Z}[\frac{1}{2}(D+\sqrt{-D})]$ such that this CM cycle intersects $E$ at $\fp$ and the CM elliptic curves on the cycle have supersingular reduction at $\fp$, which occurs when the residue field has characteristic $p$ that is ramified or inert in $\bQ(\sqrt{-D})$. 
With a coordinate defined by the $j$-invariant, the intersection is captured by the non-archimedean part of  $P_D(j(E))$, where $P_D(X)$ is the monic polynomial whose roots are $j$-invariants of the elliptic curves on the CM cycle, and the goal is to find a prime $\fp$ such that $P_D(j_E)$ has positive $\fp$-valuation and $p$ divides $D$ or $-D$ is a quadratic non-residue modulo $p$. By quadratic reciprocity, the problem reduces to studying the reduction of the CM cycle modulo primes dividing $D$, as well as at the real place. 

Our proof builds on the idea of Elkies. We relate the Hodge-type Shimura curve parametrizing the Kuga-Satake abelian varieties to the quaternionic Shimura curve $\cO^1 \backslash \cH$ that is assumed to be isomorphic to $\mathbb{P}^1$. This allows us to choose a coordinate and construct CM cycles on $\mathbb{P}^1$. For each totally positive odd prime $\lambda\in F$ such that $-\lambda$ is a square modulo $8$, we construct polynomials $P_{\lambda}(x)$ and $P_{4\lambda}(x)$ from CM cycles $\mathcal{P}_{\lambda}, \mathcal{P}_{4\lambda}$ defined over $F$. The cycles correspond to optimal embeddings $\cO_{F(\sqrt{-\lambda})} \xhookrightarrow{} \cO$ and $\cO_{F}[\sqrt{-\lambda}]\xhookrightarrow{} \cO$, respectively.\footnote{If we drop the simplifying assumption that $2$ is inert in $F$, then we need to consider optimal embeddings $R\xhookrightarrow{} \cO$ for all $\cO_F$-orders $R$ satisfying $\cO_{F}[\sqrt{-\lambda}]\subseteq R\subseteq \cO_{F(\sqrt{-\lambda})}$.}
By computing the Newton polygon via the Shimura-Taniyama formula, if a prime $\fp$ of $F$ is ramified or inert in $F(\sqrt{-\lambda})$, then the CM abelian varieties have supersingular reduction modulo primes above $\fp$. 

\subsubsection{Reduction of the CM cycles $\mathcal{P}_{\lambda}, \mathcal{P}_{4\lambda}$ modulo $\lambda$} \label{sec:redCMmodlambda}
As in the classical setting, the CM points in each cycle fall into pairs modulo $\lambda$, except possibly those that reduce to elliptic points with an even order automorphism. 
Instead of using Lubin-Tate deformation of formal groups typical of PEL cases (such as in \cite{MR1030140}, \cite{MR2704678}, \cite{LMPT}), we follow the approach of Madapusi \cite{MR3484114} on integral models for GSpin Shimura varieties, applying Grothendieck–Messing theory to establish a bijection between deformations of the abelian variety $\cA_{x_0}$ at a closed point $x_0$ with a special endomorphism and liftings of certain isotropic lines orthogonal to the special endomorphism. 
This description solves the local intersection problem and provides a clear picture of how the liftings are expected to occur in pairs given the quadratic space defined by the quaternion algebra.  

In particular, our method explains how to pair the liftings in the neighborhood of the exception points. In the case of elliptic curves, the only exception point is $j = 1728$. In the case of \cite{MR2704678} and \cite{MR2414789}, the elliptic points have CM by orders in distinct imaginary quadratic fields, and the presence of unpaired points can be predicted explicitly by checking whether a maximal order of a known quaternion algebra contains two anticommuting CM orders of given discriminants. By contrast, we have multiple elliptic points of order $2$ when $h(F(\sqrt{-1})) > 1$. To address the complication, we equip the space of special quasi-endomorphisms $V(\cA_{x_0})$ with an $F$-linear structure via comparison, and use the automorphism to pair the special endomorphisms. This yields a pairing of liftings in the neighborhood of any exception point when we consider the union of all cycles corresponding to all orders containing $\cO_F[\sqrt{-\lambda}]$.

This method for studying the deformation of mod $\fp$ points on Shimura curves works in general, without assuming that the Shimura curve has genus zero. 

\subsubsection{Real CM points on $\mathcal{P}_{\lambda}, \mathcal{P}_{4\lambda}$}\label{sec:realCMpointsonPoly}
At the archimedean places, we follow the work of Shimura \cite{MR572971} to study real points of the Shimura curve and apply Hecke's equidistribution of primes (see for instance \cite{MR1282723}*{\MakeUppercase{\romannumeral 15}, \S5}). 
The real points are given by geodesics of the form $Z_{\alpha} := \{z\in \mathcal{H}: \alpha(\overline{z}) = z\}$, where $\alpha \in \cO^\times $ satisfies $\trd(\alpha) = 0$ and $\nrd(\alpha)=\epsilon$, with $\epsilon$ a unit negative at the real place where $B$ is split and positive at other real places. Under the assumptions that $2$ is inert in $F$ and the class number $h(F(\sqrt{-\epsilon})) = 1$, it suffices to consider a single geodesic. The congruence condition on $\lambda$ ensures that each of the CM cycles we consider has a unique real point. These real CM points correspond to solutions of the norm equation $\Nm_{F(\sqrt{-\epsilon})/F} (x) = \lambda$. Similar equations are obtained in \cite{MR2704678} and \cite{MR2414789} by describing the real locus as the hyperbolic lines segments between two elliptic points. By Hecke's equidistribution of primes in $F(\sqrt{-\epsilon})$, we can find $\lambda$ such that $P_{\lambda}(x) P_{4\lambda}(x)$ is negative at the coordinate of the given abelian variety. 

To account for all real embeddings of $F$, we need to consider all conjugates of the Shimura curve, which correspond to quaternion algebra with different local invariants at infinite places. Assuming that $\cF := F(\sqrt{-\epsilon_1}, \dotsc, \sqrt{-\epsilon_n})$ has class number $1$, we construct a Hecke character of $\bA_{\cF}^\times$ and apply the equidistribution of primes in $\cF$.

Here we have found a good condition that simplifies the setting to a single geodesic. More generally, each $\cO^1$-conjugacy class of embeddings $\cO_F[\sqrt{-\epsilon}] \hookrightarrow \cO$ gives a geodesic. The associated real CM points correspond to solutions of an equation of the form $\nrd(v) = \lambda$, where $v \in B^0 \cap \cO$ is orthogonal to the image of $\sqrt{-\epsilon}$. Computations are possible with an explicit basis of maximal order, and a similar equidistribution result can be derived on each geodesic. 

\vspace{1em}
Following the proof of \cite{MR903384}, we combine \ref{sec:redCMmodlambda} and \ref{sec:realCMpointsonPoly} to obtain a desired supersingular prime.

\subsection{Organization of the paper}

In \S2, we introduce the Shimura curves considered in this work, review the theory of integral canonical model for Shimura varieties of Hodge type, focusing on the $\GSpin$ case. The section also includes some other preliminaries that will be used in the proof. In \S3, we construct the CM cycles and study their basic properties. In \S4, we investigate the reduction of the CM cycles modulo a finite prime, and pair the liftings in the neighborhood of each closed point. In \S5, we analyze the distribution of the CM points in the real locus of the Shimura curve. In \S6, we combine the results in previous sections to prove \Cref{thm:main}. In \S7, we compute explicit examples in which all the technical assumptions of \Cref{thm:main} are satisfied. 

\subsection{Notation and conventions}

Assume the following unless specified otherwise. 

Suppose $F$ is a totally real number field with narrow class number $1$. Equivalently, $F$ is a totally real number field with class number $1$ and units of independent signs.  

Suppose $B$ is a quaternion algebra over $F$ unramified at all finite places and exactly one of the real places of $F$. Necessarily $[F:\bQ]$ is odd. Denote by $\rho:F\xhookrightarrow{}\bR$ the real place where $B$ is split. 
Let $\epsilon$ be a unit of $F$ that is negative at $\rho$ and positive at the other real places, so that the field $F(\sqrt{-\epsilon})$ splits $B$.

Let $\mathcal{O} = \mathcal{O}_B$ be a maximal order of $B$.\footnote{If $F$ has narrow class number $1$, then any two maximal orders in an indefinite quaternion algebra over $F$ are conjugate to each other.} 
Its normalizer $N_{B^\times}(\mathcal{O}) = F^\times \mathcal{O}^\times$. 

\subsection*{Acknowledgments} I thank my advisor Yunqing Tang for introducing this problem to me and for the enlightening discussions and encouragement. I thank Frank Calegari for his blog post ``Polymath Proposal: 4-folds of Mumford’s type'', which was a source of inspiration and useful references. I thank Robin Huang, Wanlin Li, and Sug Woo Shin for helpful discussions.

\section{Preliminaries}

\subsection{The Shimura curves}

\subsubsection{}
Let $\widetilde{G} = \Res_{F/\bQ}\GL_{1,B}$ be the algebraic group over $\bQ$ with $\widetilde{G}(\bQ) = B^\times$, and $\widetilde{K} = \widehat{\mathcal{O}}^\times = \prod \mathcal{O}_v^\times \subset \widetilde{G}(\bA_f)$. Strong approximation implies that \begin{equation}\label{eq:strongapprox}|B^\times_{>0} \backslash \widehat{B}^\times / \widehat{\mathcal{O}}^\times| = |F^\times_{>0} \backslash \widehat{F}^\times / \nrd(\widehat{\mathcal{O}}^\times)| = h_+(F) = 1,\end{equation} 
and then \[\widetilde{G}(\bQ)\backslash \mathcal{H}^{\pm} \times \widetilde{G}(\bA_f) / \widetilde{K} = \widetilde{G}(\bQ)_+ \backslash \mathcal{H} \times \widetilde{G}(\bA_f) / \widetilde{K} = \widetilde{\Gamma} \backslash \mathcal{H} = \widetilde{\Gamma}^1 \backslash \mathcal{H},\] where \begin{align*}
    &\widetilde{\Gamma} = \{\alpha \in \mathcal{O}^\times : \nrd(\alpha) \in \mathcal{O}_F^\times \text{ is totally positive}\},\\
    &\widetilde{\Gamma}^1 = \cO^1 = \{\alpha \in \mathcal{O}^\times : \nrd(\alpha) = 1\}.
\end{align*}
(Since the totally real number field $F$ has units with independent signs, every totally positive unit in $\mathcal{O}_F^*$ is a square (\cite{MR963648}*{12.2}), which implies $\widetilde{\Gamma} \backslash \mathcal{H} = \widetilde{\Gamma}^1 \backslash \mathcal{H}$. )
The reflex field of this Shimura curve is $\rho(F)$. 

\subsubsection{}\label{sec:ShimuraDatum}
Let $V = B^0= \{\alpha \in B: \trd(\alpha) = 0\}$, then $(V, Q_F:=\nrd|_{B^0})$ is a $3$-dimensional quadratic space over $F$. For any field extension $\sigma:F\xhookrightarrow{}K$, denote by $(V\otimes_{F, \sigma} K, Q_{\sigma}:=\sigma\circ Q_F)$ the $3$-dimensional quadratic space over $K$. Let $(V, Q := \Tr_{F/\bQ}\circ Q_F)$ be the $(3[F:\bQ])$-dimensional quadratic space over $\bQ$. The signature of $(V, Q)$ is $(3[F:\bQ]-2,2)$, since there is an orthogonal direct sum decomposition $V_{\bR}= V\otimes_{\bQ}\bR \simeq \oplus_{\sigma} (V\otimes_{F,\sigma}\mathbb{R})$, where $V\otimes_{F, \sigma}\bR$ has signature $(3,0)$ for $\sigma \neq \rho$ and $(1,2)$ for $\sigma = \rho$. 

Let $\mathcal{D} = \{w\in V\otimes_{\bQ} \bC: [w,w]_Q = 0, [w,\bar{w}]_Q <0 \} / \bC^\times$ be the space of oriented negative definite $2$-planes in $V_{\bR}$ and \begin{equation*} \label{eq:Hermitian_domain}
    X = \{w\in V\otimes_{F, \rho} \bC: [w,w]_Q = 0, [w,\bar{w}]_Q <0 \} / \bC^\times.
\end{equation*} Note that $X\simeq\mathcal{H}^\pm$ is one-dimensional. 
Each $V\otimes_{F, \sigma} \bC$ is an eigenspace for the $F$-action on $V\otimes_\bQ \bC$ induced by the $F$-linear structure on $V$. The weight $0$ Hodge structure on $V$ defined by $[w] \in X$ is $V^{-1,1} = \bC w, V^{1,-1} = \bC \bar{w}, V^{0,0} = (V^{-1,1}\oplus V^{1,-1})^{\perp}$, then for $a\in F^\times$, $aV^{-1,1} = V^{-1,1}, aV^{1,-1} = V^{1,-1}$, and since $[au, v]_Q = [u,av]_Q$ for all $u,v \in V$, $aV^{0,0} = V^{0,0}$. Thus, for a Hodge structure $V$ defined by  $[w] \in X$, we have $F \subseteq \End_{Hdg}(V)$. 

Let $C(V)$ be the Clifford algebra of $V$ over $\bQ$. It has a $\bZ/2\bZ$-grading $C(V) = C^+(V) \oplus C^-(V)$, where $C^+(V)$ is the even Clifford algebra. The reductive group scheme $\GSpin(V,Q)$ over $\bQ$ is defined by  \[\GSpin(V,Q)(R) := \{g\in C^+_R(V_R)^\times: g V_R g^{-1} = V_R\}\] for any $\bQ$-algebra $R$, where $C_R$ (resp. $C_R^+$) denotes the Clifford algebra (resp. even Clifford algebra) of a quadratic space over $R$. There is a canonical involution $*$ on $C(V)$ given by the reversal involution on $\oplus_{d=0}^\infty V^{\otimes d}$, and the spinor norm $\nu: \GSpin(V, Q) \to \bG_m$ is defined by $x\mapsto x^*x$. 
A choice of $\delta \in C(V)^\times$ such that $\delta^* = -\delta$ (for example, $\delta = ef$ for orthogonal vectors $e,f\in V$ with $Q(e)<0, Q(f) < 0$) 
defines a symplectic form $\psi_{\delta}: C(V)\times C(V)\to \bQ, \, (x,y)\mapsto \Trd(x\delta y^*)$. The action of $\GSpin(V, Q)$ on $C(V)$ by left multiplication induces an embedding $\GSpin(V, Q) \xhookrightarrow{} \GSp(C(V), \psi_{\delta})$, under which the similitude character on $\GSp(C(V), \psi_{\delta})$ restricts to the spinor norm $\nu$ on $\GSpin(V, Q)$ (\cite{MR3484114}*{1.6, 1.7}).

The action of $\GSpin(V,Q)$ on $V$ by conjugation induces an exact sequence of group schemes over $\bQ$ \begin{equation} \label{eq:GspinSo}
    1 \to \mathbb{G}_m \to \GSpin(V,Q) \to \SO(V,Q) \to 1.
\end{equation}
Define the algebraic groups $G \subset \GSpin(V, Q)$ and $G_0\subset \SO(V, Q)$ by \begin{align*}
    G(R) &= \{g\in \GSpin(V, Q)(R): \alpha (g\cdot v) = g\cdot (\alpha v) \, \forall \alpha \in F, v\in V_R\} \\
    G_0(R) &= \{g \in \SO(V,Q)(R): \alpha g  = g \alpha \, \forall \alpha \in F\}
\end{align*}
for any $\bQ$-algebra $R$, so that the image of $G$ under \eqref{eq:GspinSo} is $G_0$. Since $(\alpha, \beta)\mapsto \Tr_{F/\bQ}(\alpha\beta)$ is a perfect pairing on the $\bQ$-vector space $F$, we have $G_0(\bQ) = \SO(V, Q_F)(F) = B^\times /F^\times$. 

The embedding \[(G, X) \xhookrightarrow{}(\GSpin(V, Q), \mathcal{D}) \xhookrightarrow{}(\GSp(C(V),\psi_{\delta}), \mathcal{S}^{\pm})\] realizes $(G,X)$ as a Shimura datum of Hodge type. When $F$ is cubic, we have $\Cor_{F/\bQ}(B) \simeq M_8(\bQ)$ since $B$ is split at all finite places, and this Kuga-Satake construction from K3 type Hodge structures with real multplication is consistent with Mumford's original construction (\cite{MR2492400}*{6.4}).

\subsubsection{}
We can use the corestriction of algebras to describe the morphism $\widetilde{G} \to G$ (\cite{MR248146}*{\S4}, \cite{MR2492400}*{6.2, 6.3}).
Let $\tilde{F}$ be the Galois closure of $F$. The direct sum decomposition of quadratic spaces \[(V\otimes_{\bQ} \tilde{F}, Q) = \bigoplus_{\sigma:F\xhookrightarrow{}\tilde{F}} (V\otimes_{F, \sigma}\tilde{F}, Q_{\sigma})\] gives an isomorphism of $\tilde{F}$-algebras \begin{equation} \label{eq:Cl}
    C(V)\otimes_{\bQ}\tilde{F} = C_{\tilde{F}}(V\otimes_\bQ \tilde{F}) = \widehat{\bigotimes}_{\sigma:F\xhookrightarrow{}\tilde{F}}C_{\tilde{F}}(V\otimes_{F, \sigma} \tilde{F}) = \widehat{\bigotimes}_{\sigma:F\xhookrightarrow{}\tilde{F}}\left(C_F(V)\otimes_{F, \sigma} \tilde{F}\right),
\end{equation} where $\widehat{\bigotimes}_{\sigma:F\xhookrightarrow{}\tilde{F}}$ denotes a graded tensor product over $\tilde{F}$ indexed by embeddings $\sigma:F\xhookrightarrow{}\tilde{F}$. On the all even part, the graded tensor product over $\tilde{F}$ is the usual tensor product over $\tilde{F}$, so $\bigotimes_{\sigma:F\xhookrightarrow{}\tilde{F}} \left(C_F^+(V)\otimes_{F, \sigma} \tilde{F}\right)$ lies in the even part of $\widehat{\bigotimes}_{\sigma:F\xhookrightarrow{}\tilde{F}}\left(C_F(V)\otimes_{F, \sigma} \tilde{F} \right)$. The Galois group $\Gal(\tilde{F}/\bQ)$ acts on $V\otimes_\bQ \tilde{F}$ via the second factor of the tensor product, thus under the isomorphism $V\otimes_{\bQ} \tilde{F} = \bigoplus_{\sigma:F\xhookrightarrow{}\tilde{F}} V\otimes_{F, \sigma}\tilde{F}$, this action permutes the eigenspaces $V\otimes_{F, \sigma}\tilde{F}$: for $g\in \Gal(\tilde{F}/\bQ)$, we have $g: V\otimes_{F, \sigma} \tilde{F} \to V\otimes_{F, g\sigma} \tilde{F}$  defined by $r\otimes a \mapsto r\otimes g(a)$. Similarly, $\Gal(\tilde{F}/\bQ)$ acts on $C(V)\otimes_{\bQ} \tilde{F}$ via the second factor of the tensor product. 
Over $\tilde{F}$, we have a group homomorphism \[\left(C_F^+(V) \otimes_{\bQ}\tilde{F}\right)^\times = \prod_{\sigma:F\xhookrightarrow{}\tilde{F}}\left(C_F^+(V) \otimes_{F, \sigma} \tilde{F} \right)^\times \xrightarrow{\Nm}\bigotimes_{\sigma:F\xhookrightarrow{}\tilde{F}} \left(C_F^+(V)\otimes_{F, \sigma} \tilde{F}\right)^\times\xhookrightarrow{} \left(C^+(V)\otimes_{\bQ}\tilde{F}\right)^\times,\] where the norm map is defined by \[\Nm\left((b_{\sigma})_{\sigma: F\xhookrightarrow{}\tilde{F}} \right)) = \bigotimes_{\sigma:F\xhookrightarrow{}\tilde{F}} b_{\sigma}.\]
The image of $ (b_{\sigma})_{\sigma: F\xhookrightarrow{}\tilde{F}}$ acts as $b_\sigma$ on each $V\otimes_{F,\sigma}\tilde{F}$; in particular, it preserves $V\otimes_{\bQ}\tilde{F}$ and respects the $F$-linear structure on $V\otimes_{\bQ}\tilde{F}$. Therefore, we have a group homomorphism $(C_F^+(V)\otimes_\bQ \tilde{F})^\times  \to G(\tilde{F})$. 
Taking $\Gal(\tilde{F}/\bQ)$-invariants gives a group homomorphism $C_F^+(V)^\times \to G(\bQ)$. Under the identification $C_F^+(V)\xrightarrow{\simeq}B$, we obtain the corresponding morphism of algebraic groups $\widetilde{G}\to G$ over $\bQ$, whose kernel is the algebraic torus $T_F^1$ over $\bQ$ defined by $T_F^1(\bQ) = \{x\in F^\times : \Nm_{F/\bQ} x = 1\}$. 
The morphism $\widetilde{G}\to G$ induces $\widetilde{G}^{\der} \to G^{\der}$ and the following diagram \[
\begin{tikzcd}
& & 1 \arrow[d] & 1 \arrow[d] &\\
&   & T_F^1 \arrow[r, "x\mapsto x^2"] \arrow[d] & T_F^1 \arrow[d] & \\
1 \arrow[r] & \widetilde{G}^{\der} \arrow[r] \arrow[d] & \widetilde{G} \arrow[r, "\nrd"] \arrow[d] & \Res_{F/\bQ}\mathbb{G}_m \arrow[r] \arrow[d, "\Nm_{F/\bQ}"] & 1 \\
1 \arrow[r] & G^{\der} \arrow[r] & G \arrow[r, "\nu"] \arrow[d] & \mathbb{G}_m \arrow[r] \arrow[d] & 1 \\
&  & 1 & 1 &
\end{tikzcd}
\] commutes. Since $\widetilde{G}^{\der}(\bQ)\cap T_F^1(\bQ) = \{x\in F^\times : x^2 = 1, \Nm_{F/\bQ}(x) = 1\}=\{1\}$ as $[F:\bQ]$ is odd, we have an isomorphism $\Res_{F/\bQ} \SL_{1,B}=\widetilde{G}^{\der} \xrightarrow{\simeq}G^{\der}$.

\subsubsection{}
Let $L =  V\cap \cO$, then it is a rank $3$ free $\cO_F$-module 
and a rank $3[F:\bQ]$ free $\bZ$-module. By \cite{MR4279905}*{22.4.15}, for any quaternion $\cO_F$ order $\cO'$ we have $\cO' = \cO_F + \discrd(\cO')(\cO'^{\#})^0(\cO'^{\#})^0 = C_{\cO_F}^+((\cO'^{\#})^0, N\nrd)$, where $\discrd(\cO') = (N)$ and $(\cO'^{\#})^0= \{\alpha\in B: \trd(\alpha \cO')\subseteq \cO_F, \, \trd(\alpha) = 0\}$. 
Here $\discrd(\cO) = \disc(B) =(1)$ and $\cO^{\#} = \cO$ since $\cO$ is maximal and $B$ is split at all finite places. 
Define \begin{align*}
    K &= \{g\in G(\bA_f) : g( C(L\otimes_{\bZ}\widehat{\bZ})) = C(L\otimes_{\bZ}\widehat{\bZ})\}=G(\bA_f)\cap C(L\otimes_{\bZ}\widehat{\bZ})^\times,\\
    K_0 &=\{g\in G_0(\bA_f): g(L\otimes_{\bZ}\widehat{\bZ}) = L\otimes_{\bZ}\widehat{\bZ}\}.
\end{align*}  
The image of $\widetilde{K}$ under $\widetilde{G} \to G\to G_0$ is $K_0$ since $N_{B^\times}(\cO) = F^\times \cO^\times$.
The image of $K$ in $G_0(\bA_f)$ is $\{g\in K_0: g \text{ acts trivially on } L^\vee_{\widehat{\bZ}}/L_{\widehat{\bZ}}\}$ (\cite{MR3484114}*{2.6}).
The dual lattice $L^\vee = \fD_{F/\bQ}^{-1}L \subset V$, where $\fD_{F/\bQ}\subset\cO_F$ is the different of $F/\bQ$. 
There are well-defined maps \begin{equation}\label{eq:morSV1}
    \widetilde{G}(\bQ)\backslash\cH^{\pm}\times \widetilde{G}(\bA_f)/\widetilde{K} \to G_0(\bQ)\backslash X \times G_0(\bA_f)/K_0
\end{equation} and \begin{equation}\label{eq:morSV2}
    G(\bQ)\backslash X\times G(\bA_f)/K \to G_0(\bQ)\backslash X \times G_0(\bA_f)/K_0.
\end{equation} Note that $1\to \Res_{F/\bQ} \bG_m \to \widetilde{G} \to G_0 \to 1$ and $H^1(\bQ_l, \Res_{F/\bQ}\bG_m) \simeq H^1(\bQ_l, \prod_{v|l} \Res_{F_v/\bQ_l}\bG_m) \simeq \prod_{v|l}H^1(\bQ_l, \Res_{F_v/\bQ_l}\bG_m)\simeq \prod_{v|l} H^1(F_v, \bG_m)$ is trivial by Hilbert's 90. Therefore, the map \eqref{eq:morSV1} is surjective, and in particular, $G_0(\bQ)\backslash X \times G_0(\bA_f)/K_0$ is connected. Let $\Gamma_0 := G_0(\bQ) \cap K_0 = F^\times \cO^\times/F^\times$, then \eqref{eq:morSV1} is an isomorphism \begin{equation} \label{eq:connmorSV1}\widetilde{\Gamma}\backslash \cH \xrightarrow{\simeq} \Gamma_0\backslash X^+.
\end{equation} 
For each $g\in G(\bA_f)$, let $\Gamma_g := G(\bQ)_+\cap gKg^{-1}$, then $[x]\mapsto [x,g]$ defines a connected component $\Gamma_g\backslash X^+ \xhookrightarrow{} G(\bQ)_+\backslash X^+\times G(\bA_f)/K$ and \eqref{eq:morSV2} gives a finite map \begin{equation}\label{eq:connmorSV2}
    \Gamma_g\backslash X^+ \to \Gamma_0 \backslash X^+ .
\end{equation}
Note that \[G(\bQ)\backslash X\times G(\bA_f)/K\simeq \bigsqcup_{[g] \in G(\bQ)_+\backslash G(\bA_f)/K}\Gamma_g\backslash X^+,\] and
if $x_1, x_2\in G(\bQ)\backslash X\times G(\bA_f)/K$ maps to the same point under \eqref{eq:morSV2}, 
then the corresponding abelian varieties $\cA_{x_1}, \cA_{x_2}$ are isogenous. 

\subsubsection{}
See \cite{MR2476577}*{\S4} for a complete list of genus $0$ Shimura curves constructed from congruence arithmetic Fuchsian group in $\PSL_2(\bR)$.
In particular, we have $[F:\bQ] \leq 7$. 

\subsection{Integral model}

Let $p$ be an odd prime unramified in $F$. Let $L_{(p)} = L\otimes_{\bZ} \bZ_{(p)}$ and $H_{(p)} = C(L_{(p)})$ be the Clifford algebra of $L_{(p)}$ over $\bZ_{(p)}$. Since $B$ is unramified at all finite places, $(L_{(p)}, Q)$ is non-degenerate, i.e., $Q$ induces an isomorphism $L_{(p)} \xrightarrow{\simeq} L_{(p)}^\vee$, and $\GSpin(V,Q)$ extends to a reductive group scheme $\GSpin(L_{(p)}, Q)$ over $\bZ_{(p)}$. 
Via left multiplication, the group $C^+(L_{(p)}) \subset \GL(H_{(p)})$ is the subgroup of automorphisms that preserves the grading and centralizes the right $C(L_{(p)})$-action on $H_{(p)}$.
The pairing \[[\varphi_1,\varphi_2] = \frac{1}{2^{3[F:\bQ]-1}}\Tr(\varphi_1\circ \varphi_2)\] on $\End(H_{(p)})$ restricts to $[\cdot, \cdot]_Q$ on $L_{(p)}$ since $[v,v] = \frac{1}{2^{3[F:\bQ]-1}} \Tr(Q(v)|_{H_{(p)}}) = 2Q(v)$ for all $v\in L_{(p)}$. Let $\bm{\pi}: \End(H_{(p)}) \to \End(H_{(p)})$ be the orthogonal projection onto $L_{(p)}$. Then $\GSpin(L_{(p)}, Q) \subset C^+(L_{(p)})$ is the stablizer of the idempotent operator $\bm{\pi}$ (\cite{MR3484114}*{1.4}).
For $\epsilon_1, \dotsc,\epsilon_m \in \mathcal{O}_F$ such that $\mathcal{O}_F = \bZ[\epsilon_1, \dotsc, \epsilon_m]$, consider the endomorphism $\bm{\pi}_{\epsilon_i}: \End(H_{(p)}) \to \End(H_{(p)})$ given by $\bm{\pi}_{{\epsilon_i}}(\varphi) = {\epsilon_i}(\bm{\pi}(\varphi))$. Let $G_{\bZ_{(p)}}\subset \GSpin(L_{(p)}, Q)$ be the stablizer of all $\bm{\pi}_{\epsilon_i}$. Let $\{s_{\alpha}\} \subset H_{(p)}^{\otimes}$ denote a finite collection of tensors defining $G\subset \GL(H_{(p)})$ with $\bm{\pi}, \bm{\pi}_{\epsilon_i} \in \{s_{\alpha}\}$. 

Let $K_p = G(\bZ_p)$ and $K^p\subset G(\bA_f^p)$ be a small enough open compact subgroup. Denote by $\Sh_{K_pK^p} = \Sh_{K_pK^p}(G,X)$ the Shimura curve attached to $G$ and $\mathscr{S}_{K_pK^p}$ its integral canonical model. The integral model $\mathscr{S}_{K_pK^p}$ is constructed in \cite{MR2669706} as the normalization of $\mathscr{S}_{K_pK^p}^-(G,X)$, the closure of $\Sh_{K_pK^p}$ in the natural integral model of the Siegel modular variety. The normalization step can be removed by \cite{MR4464238}. Let $ \mathcal{A} \to \mathscr{S}_{K_pK^p}$ denote the universal abelian scheme, which exists assuming that $K^p$ sufficiently small. Let $s_{\alpha, ?}$ denote the cohomological realizations of $s_{\alpha}$, where $ ? = B, \mathrm{dR}, \ell, p, \mathrm{cris}$. 

Let $k = \mathbb{F}_q \subset \overline{\mathbb{F}}_p$ and $W=W(k)$ be its ring of Witt vectors. Suppose $x\in \mathscr{S}_{K_pK^p}(k)$ is a closed point. Write $\widehat{U}_x$ for the completion of $\mathscr{S}_{K_pK^p}$ at $x$. Suppose $\tilde{x} \in \mathscr{S}_{K_pK^p}(E)$ is a point specializing to $x$, where $E$ is a finite extension of $W[p^{-1}]$.  

Choose an embedding $\iota: \overline{E}\xhookrightarrow{}\bC$. The  natural isomorphism $H_{(p)} \simeq H^1_B(\mathcal{A}_{\iota(\tilde{x})}(\bC), \bZ_{(p)})$ takes $s_{\alpha}$ to $s_{\alpha, B, \iota(\tilde{x})}$. 
The comparison isomorphism \begin{equation} \label{eq:Betti_deRham}
    H^1\dR(\mathcal{A}_{\tilde{x}}/E) \otimes_{E, \iota} \bC \xrightarrow{\simeq} H^1_B(\mathcal{A}_{\iota(\tilde{x})}(\bC), \bQ)\otimes_{\bQ} \bC
\end{equation}
takes $s_{\alpha, \mathrm{dR}, \tilde{x}}\otimes 1$ to $s_{\alpha, B, \iota(\tilde{x})}\otimes 1$, and the comparison isomorphism \begin{equation}\label{eq:Betti_et}
    H^1_B(\mathcal{A}_{\iota(\tilde{x})}(\bC), \bZ_{(p)})\otimes_{\bZ_{(p)}} \bZ_p \xrightarrow{\simeq}H^1_{\et}(\mathcal{A}_{\tilde{x}_{\bar{E}}}, \bZ_p)
\end{equation} takes $s_{\alpha, B,\iota(\tilde{x})}\otimes 1$ to $s_{\alpha, p, \tilde{x}}$. 
Via the $p$-adic comparison isomorphism \begin{equation}
    H^1_{\et}(\mathcal{A}_{\tilde{x}_{\bar{E}}}, \bZ_p)\otimes_{\bZ_p} B\cris \xrightarrow{\simeq} H^1\cris(\mathcal{A}_x/W)\otimes_W B\cris,
\end{equation}
the $\Gal(\overline{F_v}/E)$-invariant tensors $s_{\alpha, p, \tilde{x}}$ give rise to Frobenius invariant tensors $s_{\alpha, \mathrm{cris}, x} \in H^1\cris(\mathcal{A}_x/W)^\otimes$. From the proof of \cite{MR2669706}*{2.3.5}, it follows that the tensors $s_{\alpha, \mathrm{cris}, x}$ are independent of the choice of $\tilde{x}$. 
There exists a $W$-linear isomorphism \begin{equation} \label{eq:et_cris}
    H^1_{\et}(\mathcal{A}_{\tilde{x}_{\bar{E}}}, \bZ_p)\otimes_{\bZ_p} W \xrightarrow{} H^1\cris(\mathcal{A}_x/W)
\end{equation} taking $s_{\alpha, p, \tilde{x}}$ to $s_{\alpha, \mathrm{cris}, x}$ (\cite{MR2669706}*{1.4.3}). From \eqref{eq:Betti_et} and \eqref{eq:et_cris}, there exists a $W$-linear isomorphism $H_{(p)} \otimes_{\mathbb{Z}_{(p)}} W \to  H^1\cris(\mathcal{A}_x/W)$ taking $s_\alpha$ to $s_{\alpha, \mathrm{cris}, x}$, and the tensors $s_{\alpha, \mathrm{cris}, x}$ define a reductive subgroup $G_W\subset GL(H^1\cris(\mathcal{A}_x/W))$, which is isomorphic to $G_{\bZ_{(p)}}\times_{\Spec \bZ_{(p)}}\Spec W$. Let $\bm{L}_{\mathrm{cris}, x}$ be the image of $\bm{\pi}_{\mathrm{cris}, x}$ on $\End(H^1\cris(\cA_x/W))$, then since all $\bm{\pi}_{{\epsilon_i}}$ and the quadratic form on $L_{(p)}$ are absolute Hodge cycles, there is an $\mathcal{O}_F$-linear isometry \begin{equation}
    L_{(p)}\otimes_{\bZ_{(p)}} W \to \bm{L}_{\mathrm{cris}, x},
\end{equation} and $G_W \subset \GSpin(\bm{L}_{\mathrm{cris},x})$. Since $p$ is unramified in $F$, we have an orthogonal direct sum decomposition \begin{equation} \label{eq:orthogonaldecomp}
    \bm{L}_{\mathrm{cris}, x} \simeq  L_{(p)}\otimes_{\bZ_{(p)}} W=\bigoplus_{\sigma':F\xhookrightarrow{}W[p^{-1}]}L\otimes_{\cO_F, \sigma'} W,
\end{equation} where orthogonality arises from the distinct scalar actions of $\cO_F$ via the embeddings $\sigma'$. 
From the proof of (2.3.5) in \cite{MR2669706}, $\widehat{U}_x$ is isomorphic to $\Spf R_{G_W}$, where $R_{G_W}$ is the complete local ring at the identity section of the opposite unipotent defined by a cocharacter $\mathbb{G}_m \to G_W$ whose reduction modulo $p$ induces the filtration on $H^1\cris(\mathcal{A}_x/W)\otimes k = H^1\dR(\mathcal{A}_x/k)$. For any $\tilde{x}' \in \widehat{U}_x$ defined over a finite extension $E'$ of $W[p^{-1}]$, the filtration on \[H^1\dR(\mathcal{A}_{\tilde{x}'}/E')\xrightarrow{\simeq}H^{1}\cris(\mathcal{A}_x/W)\otimes_W E'\] is induced by a $G_W\otimes_W E'$-valued cocharacter (\cite{MR2669706}*{1.4.5}). 
 
\subsection{Special endomorphisms}
We follow \cite{MR3484114}*{5.1-5.13} to define the space of special endomorphisms $L(\cA_x)$ of a point $x\to \mathscr{S}_{K_pK^p}$. 

\subsubsection{}
Let $x\to \Sh_{K_pK^p}$ be a geometric point with $k(x)\subseteq \bC$ and $\ell$ be a rational prime. 
The comparison isomorphism \begin{equation}\label{eq:Betti_l}
    H^1_B(\mathcal{A}_{x}(\bC), \bQ)\otimes_{\bQ} \bQ_\ell\xrightarrow{\simeq}H^1_{\et}(\mathcal{A}_{x}, \bQ_\ell)
\end{equation} takes $s_{\alpha, B,x}\otimes 1$ to $s_{\alpha, \ell, x}$. Let $\bm{V}_{\ell, x}$ be the image of $\bm{\pi}_{\ell,x}$, then \eqref{eq:Betti_l} induces an $F$-linear isometry \begin{equation}\label{eq:Betti_l_isometry}
    V\otimes_{\bQ}\bQ_l \xrightarrow{\simeq} \bm{V}_{\ell,x}.
\end{equation}  Let $\bm{L}_{\ell,x} = \bm{V}_{\ell,x} \cap \End(T_\ell(\cA_x))$, where $T_\ell(\cA_x) = \varprojlim \cA_x[\ell^n]$ is the $\ell$-adic Tate module of $\cA_x$. Since the comparison isomorphism gives \begin{equation}H^1_B(\mathcal{A}_{x}(\bC), \bZ)\otimes_{\bZ} \bZ_\ell\xrightarrow{\simeq}H^1_{\et}(\mathcal{A}_{x}, \bZ_\ell)\end{equation} and $L = V\cap \End(C(L))$,  under \eqref{eq:Betti_l_isometry} we have  \begin{equation}L\otimes_{\bZ}\bZ_\ell \simeq \bm{L}_{\ell,x}.\end{equation}

\begin{defn}[\cite{MR3484114}*{5.4, 5.5}]
    An endomorphism $f\in \End(\cA_x)$ is special if it satisfies any of the equivalent conditions: \begin{enumerate}
        \item the Betti realization of $f_{\bC}$ gives a section of $\bm{V}_{B,x} \subset \End(H^1_B(\mathcal{A}_{x}(\bC), \bQ))$; 
        \item the $\ell$-adic realization of $f$ is an element of $\bm{V}_{\ell,x} \subset \End(H^1_{\et}(\mathcal{A}_{x}, \bQ_\ell))$ for some prime $\ell$;
        \item the $\ell$-adic realization of $f$ is an element of $\bm{V}_{\ell,x} \subset \End(H^1_{\et}(\mathcal{A}_{x}, \bQ_\ell))$ for all prime $\ell$.
    \end{enumerate}
\end{defn}

\subsubsection{}
Let $x\in \mathscr{S}_{K_pK^p}(\overline{\bF}_p)$ and $\ell \neq p$. 

\begin{defn}
    An endomorphism $f\in \End(\cA_x)$ is special if its crystalline realization lies in $\bm{L}_{\mathrm{cris},x}$. 
\end{defn}
The definition of $\bm{V}_{\ell,x}$ and $ \bm{L}_{\ell,x}$ carries over since $T_\ell(\cA_{\tilde{x}}) \xrightarrow{\simeq }T_\ell(\cA_x)$ for any lift $\cA_{\tilde{x}}$ of $\cA_x$. 
By \cite{MR3484114}*{5.13}, if $f\in \End(\cA_x)$ is special, then its $\ell$-adic realization is an element of $\bm{L}_{\ell,x}$. 

\subsection{Intersection number}

\begin{defn}
    Let $R$ be a Dedekind domain and $\mathcal{X} \to \Spec R$ be an arithmetic surface. 
    Let $D$ and $E$ be two effective divisors on $\cX$ with no common irreducible component. Let $z_0\in \mathcal{X}$ be a closed point. The local intersection number $i_{z_0}(D,E)$ of $D$ and $E$ at $z_0$ is the length of the $\cO_{\cX,z_0}$-module $\cO_{\cX, z_0} / (\cO_{\cX}(-D)_{z_0} + \cO_{\cX}(-E)_{z_0})$. 
\end{defn} 
\begin{example}
    Let $R$ be a discrete valuation ring with field of fraction $K$, maximal ideal $\fp$ and residue field $k$. Let $x,y\in \cX(K)$ be distinct, then $x,y$ extend uniquely to $\ul{x}, \ul{y} \in \cX(R)$ by properness of $\mathcal{X} \to \Spec R$, and $\ul{x}, \ul{y} $ are closed immersions since they are sections to a separated map. Define \[(\ul{x}. \ul{y}) :=  \sum_{z\in \cX_k} i_z(\ul{x}, \ul{y}).\] Let $x_n, y_n \in \cX(R/\fp^n)$ be the reduction of $\ul{x}, \ul{y}$ modulo $\fp^n$ for positive integer $n$. Suppose $x\neq y$, 
    then as in \cite{MR2705896}*{3.13}, \[(\ul{x}. \ul{y}) = \max\{n : x_n = y_n \}. \] In particular, 
    \begin{enumerate}
        \item if $\cX$ is a fine moduli space with a universal object $\cA\to \cX$, then the complete local ring $\widehat{\cO}_{\cX, z_0}$ is the universal deformation ring for $z_0$ and \[(\ul{x}. \ul{y}) = \max\{n : \cA_{\ul{x}} \simeq \cA_{\ul{y}} \mod{\fp^n} \};\]
        \item if $\cX = \bP^1_{R}$ and $x_1 = y_1 = z_0 \in \cX(k)$, then \[(\ul{x}. \ul{y}) = i_{z_0}(\ul{x}, \ul{y}) = \begin{cases}
          v_{\fp}(x-y)   & v_{\fp}(x) \geq 0 , v_{\fp}(y)\geq0,\\
          v_{\fp}(\frac{1}{x} - \frac{1}{y}) & v_{\fp}(x) <0 \text{ or }x = \infty, v_{\fp}(y) <0 \text{ or }y = \infty.
        \end{cases}\]
    \end{enumerate}
\end{example}

\begin{lemma}\label{lem:intersection_auto}
    Let $R$ be a complete discrete valuation ring with field of fraction $K$, maximal ideal $\fp$ and algebraically closed residue field $k = \bar{k}$. Suppose $\cY \to \Spec R$ is a smooth curve over $R$ and $G$ is a finite group acting $R$-linearly on $\cY$.  
    Let $\cX = \cY/G$.
    \begin{enumerate}
        \item The complete local ring of $\cY$ at a closed point $y_0$ is $R$-isomorphic to $R[[T]]$. The isotropy group $G_{y_0} \subset G$ of $y_0$ acts $R$-linearly on $R[[T]]$, and the complete local ring of $\cX$ at the image $x_0$ of $y_0$ is the ring of invariants $\left(R[[T]]\right)^{G_{y_0}} = R[[\Nm_{\overline{G}_{y_0}}(T)]]$, where $\overline{G}_{y_0}$ is the image of $G_{y_0}$ in $\Aut_{R}(R[[T]])$ and $\Nm_{\overline{G}_{y_0}}(T) = \prod_{\gamma\in \overline{G}_{y_0}} \gamma(T)$. 
        \item Suppose $\ul{y}, \ul{y}' \in \cY(R)$ both reduce to $y_0 \in \cY(k)$. Suppose $T, T'$ cut out the sections $\ul{y}, \ul{y}': \widehat{\cO}_{\cY, y_0} \twoheadrightarrow{} R$, respectively. Let $\ul{x}, \ul{x'}$ be the images of $\ul{y}, \ul{y}'$, respectively, then \[(\ul{x}.\ul{x'}) = \sum_{\gamma\in \overline{G}_{y_0}}\mathrm{length} (\widehat{\cO}_{\cY, y_0} / (T, \gamma(T')).\]
        In particular, let $G_{y'} \subset G_{y_0}$ be the isotropy group of $\ul{y}'$, and $\overline{G}_{y'}$ be its image in $\overline{G}_{y_0}$, then \[(\ul{x}.\ul{x'}) = \# \overline{G}_{y'}\sum_{\gamma\in \overline{G}_{y_0}/ \overline{G}_{y'}}\mathrm{length} (\widehat{\cO}_{\cY, y_0} / (T, \gamma(T'))\] is divisible by $\# \overline{G}_{y'}$, and $(\ul{x}.\ul{x'}) = \# \overline{G}_{y_0}(\ul{y}.\ul{y'})$ if $G_{y'} = G_{y_0}$. 
    \end{enumerate}
\end{lemma}
\begin{proof}
    \begin{enumerate}
        \item \cite{MR772569}*{pp. 508-509}
        \item \cite{MR2083211}*{p.12}
    \end{enumerate}
\end{proof}

\subsection{Quadratic reciprocity}

\begin{theorem}[\cite{MR638719}*{Theorem 167}] \label{thm:quadratic_reciprocity}
        Fix a number field $K$ with $r_1$ real embeddings $\rho_1, \dotsc, \rho_{r_1}: K\xhookrightarrow{} \bR$. Let $\alpha, \beta \in \cO_K$, where $\alpha$ is odd, i.e., $\alpha$ is relatively prime to $2$. 
        If $\beta = \mathfrak{m}\mathfrak{n}$, where $\mathfrak{m}$ is an integral ideal without odd prime factors and $\mathfrak{n}$ is an odd integral ideal, then  \[\left(\frac{\beta}{\alpha}\right)\cdot \left(\frac{\alpha}{\mathfrak{n}}\right) = (-1)^{\sum_{i=1}^{r_1}\frac{\sgn\rho_i(\alpha)-1}{2}\frac{\sgn\rho_i(\beta)-1}{2}}\] if the odd number $\alpha$ is a quadratic residue$\mod 4\mathfrak{m}$ and relatively prime to $\beta$.
\end{theorem}

\section{CM points}\label{sec:CMpts}

In this section, we construct CM cycles and describe some properties of the cycles and the abelian varieties they parametrize.
An abelian variety parametrized by a Shimura variety of Hodge type is of CM type if and only if the corresponding Hodge cocharacter factors through a $\bQ$-rational torus.
The CM cycles are constructed in \cref{sec:constructCM} using subtori of $\widetilde{G}$. Each CM cycle on $\widetilde{G}(\bQ)\backslash \mathcal{H}^{\pm} \times \widetilde{G}(\bA_f) / \widetilde{K}  = \widetilde{\Gamma} \backslash \mathcal{H} \simeq \Gamma_0 \backslash  X^+$ corresponds to an $\cO_F$-order in a CM extension of $F$ by \Cref{lem:optemb}, and it is defined over $F$ by \Cref{lem:realCMpt}. Its size equals the class number of the corresponding $\cO_F$-order, as shown in \Cref{lem:orderPicard} and computed in \Cref{lem:classnumber}.  The specific $\cO_F$-orders we will consider are described in \cref{order}.  Finally, using the Shimura-Taniyama formula, we obtain in \Cref{supersingular} a sufficient condition for the primes at which the abelian varieties on the CM cycle have supersingular reduction.

\subsection{CM cycles} \label{sec:constructCM}

Let $L = F(\sqrt{-\lambda})$ be a CM extension of $F$ and $\phi: L\xhookrightarrow{} B$ be an embedding.\footnote{Such an embedding $\phi: L \xhookrightarrow{} B$ exists because $B$ is split at all finite places of $F$, which implies that any CM extension $L$ of $F$ splits $B$, and $L$ is a maximal subfield of $B$.}
Note that $\phi(L)\cap \cO = \phi(\phi^{-1}(\cO))$ and $\phi$ defines an optimal embedding $\phi^{-1}(\cO)\xhookrightarrow{}\cO$. 
Let $\widetilde{T} = \Res_{L/\bQ}\bG_m$ and $\tilde{h}: \Res_{\bC/\bR} \bG_m \to \widetilde{G}_\bR$ be the cocharacter defined by \begin{equation} \label{eq:CMcyclecocharacter}
    \bC^\times \to \prod_{\sigma:F\xhookrightarrow{}\bR} \bC^\times = \prod_{\sigma:F\xhookrightarrow{}\bR} (L\otimes_{F, \sigma}\bR)^\times  = L\otimes_{\bQ}\bR\xhookrightarrow{\phi}B\otimes_{\bQ}\bR, 
\end{equation} where $ \bC^\times \to \prod_{\sigma:F\xhookrightarrow{}\bR} \bC^\times$ is the inclusion on the coordinate corresponding to $\rho: F\xhookrightarrow{}\bR$ at which $B$ is split. Consider the CM cycle \begin{equation}\label{eq:CMcycle}\widetilde{T}(\bQ) \backslash \{\tilde{h}\} \times\widetilde{T}(\bA_f) / K_{\widetilde{T}} \xhookrightarrow{}\widetilde{G}(\bQ) \backslash \mathcal{H}^\pm \times\widetilde{G}(\bA_f) / \widetilde{K},\end{equation} where $K_{\widetilde{T}} = \widetilde{T}(\bA_f) \cap \phi^{-1}(\mathcal{\widehat{O}})$. This is injective because if $[\tilde{h},t_1] = [\tilde{h},t_2]$ with $t_1, t_2\in \widetilde{T}(\bA_f)$, then there exist $a\in \widetilde{G}(\bQ), \, k\in \widetilde{K}$ such that $a\cdot \tilde{h} = \tilde{h}$ and $ at_1k= t_2$. The first equality implies $a\in L^\times$ and then $k = t_1^{-1}a^{-1}t_2 \in \widetilde{T}(\bA_f) \cap \widetilde{K} =K_{\widetilde{T}}$. 

This CM cycle can be described via optimal embeddings $\phi^{-1}(\cO) \xhookrightarrow{}\cO$ by \Cref{lem:optemb}, which is a special case of the trace formula \cite{MR4279905}*{30.4.7}. Let \begin{align*}
    \mathcal{E} := \{\beta\in B^\times: \beta^{-1} \phi(L)\beta \cap \mathcal{O} = \beta^{-1}\phi(\phi^{-1}(\mathcal{O}))\beta\} = \{\beta\in B^\times: \phi(L)\cap \beta\mathcal{O}\beta^{-1} = \phi(\phi^{-1}(\mathcal{O}))\},\\
    \widehat{\mathcal{E}} := \{\widehat{\beta}\in \widehat{B}^\times: \widehat{\beta}^{-1} \widehat{\phi(L)}\widehat{\beta} \cap \widehat{\mathcal{O}} = \widehat{\beta}^{-1} \widehat{\phi(\phi^{-1}(\mathcal{O}))}\widehat{\beta}\}= \{\widehat{\beta}\in \widehat{B}^\times:  \widehat{\phi(L)} \cap \widehat{\beta}\widehat{\mathcal{O}}\widehat{\beta}^{-1} =  \widehat{\phi(\phi^{-1}(\mathcal{O}))}\}.
\end{align*} By Skolem-Noether theorem, the map $\beta \mapsto \beta^{-1}\phi\beta$ gives a bijection from $\phi(L)^\times \backslash \cE$ to the set of optimal embedings $\phi^{-1}(\cO)\xhookrightarrow{}\cO$. 
Given $e_1 \in \widehat{\cE}$, if $[\tilde{h}, e_1] =  [\tilde{h}, e_2]$ in $\widetilde{G}(\bQ)\backslash \mathcal{H}^{\pm} \times \widetilde{G}(\bA_f) / \widetilde{K}$ with $e_2\in \widetilde{G}(\bA_f)$, then there exist $a\in \widetilde{G}(\bQ)$ and $ k\in \widetilde{K}=\widehat{\cO}^\times$ such that $e_2 = a e_1 k \in \widehat{\cE}$. 
Thus, $e \mapsto [\tilde{h},e]$ defines a bijection from $\phi(L)^\times \backslash \widehat{\mathcal{E}}/\widehat{\mathcal{O}}^\times$ to a set of CM points in $\widetilde{G}(\bQ)\backslash \mathcal{H}^{\pm} \times \widetilde{G}(\bA_f) / \widetilde{K}$. Note that $\widetilde{T}(\bA_f)\xhookrightarrow{} \widehat{\mathcal{E}}$.

\begin{lemma} \label{lem:optemb}
    \begin{enumerate}
        \item When $\widehat{B}^\times = B^\times \widehat{\mathcal{O}}^\times$, the set of CM points $\{[\tilde{h}, e]: e\in \widehat{\mathcal{E}}\}$ corresponds to $\cO^\times$-conjugacy classes of optimal embeddings $\phi^{-1}(\mathcal{O})\xhookrightarrow{}\mathcal{O}$.  
        \item When $B$ is split at all finite primes of $F$, the set of CM points $\{[\tilde{h}, e]: e\in \widehat{\mathcal{E}}\}$ is the image of the CM cycle $\widetilde{T}(\bQ) \backslash \{\tilde{h}\} \times\widetilde{T}(\bA_f) / K_{\widetilde{T}}$, which has cardinality $h(\phi^{-1}(\mathcal{O}))$.
    \end{enumerate}
\end{lemma}
\begin{proof}
    \begin{enumerate}
        \item Clearly $\cE \subset \widehat{\cE}$ and there is a natural map $\phi(L)^\times \backslash\cE/\mathcal{O}^\times \to \phi(L)^\times \backslash \widehat{\cE} /\widehat{\cO}^\times$. This is injective since $B\cap \widehat{\cO}^\times = \cO^\times$. For any $\widehat{\beta}\in \widehat{B}^\times$, since $\widehat{B}^\times = B^\times \widehat{\mathcal{O}}^\times$, there is $\beta\in B^\times$ such that $\widehat{\beta}\widehat{\cO}^\times = \beta\widehat{\cO}^\times$. Note that the class $\beta\cO^\times$ is well-defined, and $\widehat{\beta}\in \widehat{\cE}$ if and only if $\beta\in \cE$. Therefore, for $[\tilde{h}, \widehat{\beta}]$ where $\widehat{\beta}\in \widehat{\cE}$, pick $\beta \in \cE$ such that $\widehat{\beta}\mathcal{O}^\times = \beta\cO^\times$, then $[\tilde{h}, \widehat{\beta}] = [\tilde{h}, \beta] = [\beta^{-1}\cdot \tilde{h}, 1]$, where $\beta^{-1}\cdot \tilde{h}$ is defined by an optimal embedding $z\mapsto \beta^{-1}\phi(z) \beta$ of $\phi^{-1}(\cO)\xhookrightarrow{}\cO$. 
        \item As in the proof of \cite{MR4279905}*{30.4.7}, there is a natural surjective map $\phi(L)^\times \backslash \widehat{\mathcal{E}}/\widehat{\mathcal{O}}^\times\to \widehat{\phi(L)}^\times\backslash \widehat{\mathcal{E}}/\widehat{\mathcal{O}}^\times$ whose fiber over the identity element is $\phi(L)^\times \backslash \widehat{\phi(L)}^\times/(\widehat{\phi(L)}^\times\cap\widehat{\mathcal{O}}^\times)\simeq Pic (\phi^{-1}(\mathcal{O}))$. 
        There is a bijection $\widehat{\phi(L)}^\times\backslash \widehat{\mathcal{E}}/\widehat{\mathcal{O}}^\times \xrightarrow{\simeq} \{\widehat{\mathcal{O}}^\times\text{-conjugacy classes of optimal }\widehat{\phi^{-1}(\mathcal{O})}\xhookrightarrow{} \widehat{\mathcal{O}}\}$. In the case $B$ is split at all finite primes of $F$, since $\#\{GL_2(\mathcal{O}_{F_v})\text{-conjugacy classes of optimal }S\xhookrightarrow{} M_2(\mathcal{O}_{F_v})\} = 1$ for any $\mathcal{O}_{F_v}$-order $S$ by  \cite{MR4279905}*{30.5.3}, we have $\# \widehat{\phi(L)}^\times\backslash \widehat{\mathcal{E}}/\widehat{\mathcal{O}}^\times = \prod_v 1 = 1$.
    \end{enumerate}
\end{proof}

\begin{lemma}\label{lem:realCMpt}
    When $B$ is split at all finite primes of $F$, complex conjugation defines an involution on the CM cycle $\widetilde{T}(\bQ) \backslash \{\tilde{h}\} \times\widetilde{T}(\bA_f) / K_{\widetilde{T}}$. In particular, if $h(\phi^{-1}(\mathcal{O}))$ is odd, then there is a real point on this CM cycle. 
\end{lemma}
\begin{proof}
    For a CM point $[\tilde{h},t]$ where $\tilde{h}$ is induced by $\phi: L\xhookrightarrow{}B$ and $t\in \widetilde{T}(\bA_f)$, its complex conjugate is $[\iota \tilde{h}, t]$ (\cite{MR621017}), where $\iota \tilde{h}$ is induced by $\overline{\phi}: L\xhookrightarrow{}B,\, \overline{\phi}(z) = \phi(\bar{z})$.
    By the Skolem-Noether theorem, there exists $\beta\in B^\times$ such that $\overline{\phi}(\alpha) = \beta^{-1} \phi(\alpha) \beta$, then $\iota \tilde{h}(z) = \beta^{-1} \tilde{h}(z) \beta$ and $[\iota \tilde{h}, t] = [\tilde{h}, \beta t] = [\tilde{h}, \bar{t}\beta]$. 
    Since $\beta^{-1} \phi(L)\beta = \overline{\phi(L)} = \phi(L)$ and $\beta^{-1} \phi(\phi^{-1}(\cO))\beta = \overline{\phi(\phi^{-1}(\cO))} =\phi(\phi^{-1}(\cO))= \phi(L)\cap \cO$, we have $\beta\in \mathcal{E}$, then $\bar{t}\beta \in \widehat{\mathcal{E}}$, and the result follows from the second part of \Cref{lem:optemb}.
\end{proof}

\begin{lemma} \label{lem:uniquereal}
    Suppose $F$ is a totally real number field with narrow class number $1$ and $B$ is a quaternion algebra over $F$ unramified at all finite places and exactly one of the real places of $F$. If $\phi: L\xhookrightarrow{}B$ is an embedding with $h(\phi^{-1}(\cO))$ odd, then there is a unique real point on the CM cycle  $\widetilde{T}(\bQ) \backslash \{\tilde{h}\} \times\widetilde{T}(\bA_f) / K_{\widetilde{T}}$ defined by $\phi$.   
\end{lemma}
\begin{proof}
    Existence of a real point follows from \Cref{lem:realCMpt}. Note that $\widehat{B}^\times = B^\times \widehat{\cO}^\times$ by \eqref{eq:strongapprox}. 
    Then by \Cref{lem:optemb}, replacing $\phi$ by $\beta^{-1}\phi\beta$ and consequently $\tilde{h}$ by $\beta^{-1}\cdot \tilde{h}$ for $\beta\in\cE$ does not change the CM cycle. Thus, we may assume $[\tilde{h},1]$ is a real point on this CM cycle. There exists some $\beta\in B^\times$ such that the complex conjugation of any point $[\tilde{h},t]$ with $ t\in \widetilde{T}(\bA_f)$ is given by $[\tilde{h}, \beta t] = [\tilde{h}, \bar{t}\beta]$. Since $[\tilde{h}, 1]$ is real, we have $[\tilde{h}, 1] = [\tilde{h}, \beta]$, which implies $\beta \in \widetilde{T}(\bQ) \widetilde{K}$, and then $[\tilde{h}, \bar{t}\beta] = [\tilde{h}, \bar{t}]$.

    By \Cref{lem:orderPicard}, there is an isomorphism of groups $i:\widetilde{T}(\bQ)\backslash \widetilde{T}(\bA_f)/ K_{\widetilde{T}} \xrightarrow{\simeq} \Pic(\phi^{-1}(\cO))$ compatible with complex conjugation. Then $[\tilde{h}, t]$ is real, i.e.,  $[\tilde{h}, t] = [\tilde{h}, \bar{t}]$ if and only if $i(t) = \overline{i(t)}$.
    The norm map defines a homomorphism $\widetilde{T}(\bQ)\backslash \widetilde{T}(\bA_f)/ K_{\widetilde{T}} \simeq \Pic(\phi^{-1}(\cO))\to \Pic(\cO_F) = \Cl(F)$. Since $h(F) = 1$ and $\mathcal{O}_F\subset \phi^{-1}(\mathcal{O})$, for $I$ an invertible ideal of $\phi^{-1}(\cO)$ we have $I\bar{I} = \Nm_{L/F}(I) \phi^{-1}(\cO)$ a principal ideal in $\phi^{-1}(\cO)$. Then the real CM points in this CM cycle correspond to the $2$-torsion points in $\Pic(\phi^{-1}(\cO))$. Under the assumption that $h(\phi^{-1}(\cO))$ is odd, there is a unique $2$-torsion point in $\Pic(\phi^{-1}(\cO))$; thus there is a unique real CM point in this cycle. 
\end{proof}

\begin{lemma}
    Let $F$ be a number field with class number $1$ and $L$ be a quadratic extension of $F$. Then the maximal order of $L$ is of the form $\cO_L = \cO_F[\alpha]$ for some $\alpha \in \cO_L$, and any order $R$ of $K$ containing $\cO_F$ is of the form $R = \cO_F + f \cO_K = \cO_F [f\alpha]$ for some $f\in \cO_F$; in particular, $R$ has conductor $f\cO_K$ and $R$ is stable under the action of $\Gal(L/F)$.
\end{lemma}
\begin{proof}
    Since $\cO_F$ is a PID, given an order $R\subseteq \cO_K$ containing $\cO_F$, let $\omega_1, \omega_2$ be a basis for $\cO_K$ over $\cO_F$ such that $a_1\omega_1, a_2\omega_2$ form a basis for $R$ over $\cO_F$ for some $a_1, a_2\in \cO_F$.  We have $1 = b_1\omega_1+b_2\omega_2$ for some $b_1, b_2\in \cO_F$ relatively prime, then $b_1c_1 + b_2c_2 = 1$ for some $c_1, c_2\in \cO_F$, and $1, \alpha:= -c_2\omega_1+c_1\omega_2$ form another basis for $\cO_K$ over $\cO_F$. 
    Let $f = a_1a_2$, then $f \cO_K \subset R$ and $\disc_{L/F}(R) = f^2\disc_{L/F}(\cO_L)$.
    Since $\cO_F + f\cO_K \subseteq R$ and $\cO_F+f\cO_K = \cO_F[f\alpha]$ has discriminant $f^2 d$, we have $R = \cO_F + f\cO_K$.
\end{proof}

\begin{lemma} \label{lem:orderPicard}
    Let $L/F$ be an extension of number fields, and $\mathfrak{f}\subset \cO_F$ be an ideal. Let $R = \cO_F+\mathfrak{f}\cO_L$, which is an order of $L$ containing $\cO_F$, and $R_{\fp} = R\otimes_{\cO_F}\cO_{F_{\fp}} $ be the completion of $R$ at a prime $\fp \subset \cO_F$. Note that $\cO_{L, \fp} = \prod_{\fq | \fp} \cO_{L_{\fq}}$, so that $\prod_{\fp} R_{\fp}^\times \subset \prod_{\fp} \prod_{\fq|\fp} \cO_{L_{\fq}}^\times \subset \bA_{L,f}^\times$. There is an isomorphism $\bA_{L,f}^\times /(L^\times \prod_{\fp} R_{\fp}^\times ) \simeq \Pic(R)$. 
\end{lemma}
\begin{proof}
    Write $\mathfrak{f} = \prod_{\mathfrak{p}} \fp^{m(\fp)}$, then $\mathfrak{f}\cO_L = \prod_{\mathfrak{p}}\prod_{\mathfrak{q}|\mathfrak{p}} \fq^{e(\fq/\fp)m(\fp)}$. Let $\bI_{\mathfrak{f},R} := \{(\alpha_\fq)\in \bA_{L,f}^\times: (\alpha_\fq)_{\fq|\fp} \in R_{\fp}^\times \text{ for all } \fp | \mathfrak{f}\}$
and $L_{\mathfrak{f},R}:= \{x\in L^\times: x \in R_{\fp}^\times \text{ for all }\fp |\mathfrak{f}\} = L^\times \cap \bI_{\mathfrak{f},R}$. Since $R_{\fp} \supseteq \cO_{F_{\fp}}+\prod_{\fq|\fp}\mathfrak{f}\cO_{L_{\fq}}$, we have  $\bI_{\mathfrak{f},R} \supset \bI_{\mathfrak{f}, 1} := \{(\alpha_\fq)\in \bA_{L,f}^\times: \alpha_{\fq}-1\in \fq^{e(\fq/\fp)m(\fp)}\cO_{L_{\fq}}\text{ for all }\fq |\mathfrak{f}\cO_L\}$ and $L_{\mathfrak{f},R} \supset L_{\mathfrak{f}, 1} := \{x\in L^\times: v_{\fq}(x-1)\geq e(\fq/\fp)m(\fp)\text{ for all }\fq |\mathfrak{f}\cO_L\}$. The inclusion $\bI_{\mathfrak{f},R} \xhookrightarrow{} \bA_{L,f}^\times$ defines an isomorphism $\bI_{\mathfrak{f},R} /((L^\times \prod_{\fp} R_{\fp}^\times )\cap \bI_{\mathfrak{f},R}))\xrightarrow{\simeq} \bA_{L,f}^\times/(L^\times \prod_{\fp} R_{\fp}^\times )$, 
    where $(L^\times \prod_{\fp} R_{\fp}^\times )\cap \bI_{\mathfrak{f},R} = L_{\mathfrak{f},R}\prod_{\fp}R_{\fp}^\times$. 
    
    Let $J(R, \mathfrak{f})$ (resp. $J(\cO_{L}, \mathfrak{f})$) be the subgroup of the group $J(R)$ (resp. $J(\cO_L)$) of invertible ideals of $R$ (resp. $\cO_{L}$) generated by prime ideals not dividing $\mathfrak{f}\cO_{L}$.
    Then $\mathfrak{a}\mapsto \mathfrak{a}\cO_{L}$ defines an isomorphism $J(R, \mathfrak{f}) \xrightarrow{\simeq}J(\cO_{L}, \mathfrak{f}) $ (\cite{MR1697859}*{\MakeUppercase{\romannumeral 1}.12.6, \MakeUppercase{\romannumeral 1}.12.10}).
    The natural map $\bI_{\mathfrak{f},R} \to J(\cO_L, \mathfrak{f})$ has kernel $\prod_{\fp} R_{\fp}^\times$ and defines an isomorphism \[\bI_{\mathfrak{f},R}/(L_{\mathfrak{f},R}\prod_{\fp}R_{\fp}^\times)\xrightarrow{\simeq} J(\cO_{L}, \mathfrak{f})/\{x\cO_{L}: x\in L_{\mathfrak{f}}\}.\] 
    The inclusion $J(\cO_L, \mathfrak{f})\simeq J(R, \mathfrak{f}) \xhookrightarrow{} J(R)$ defines an injection \begin{equation}\label{eq:orderPicard}J(\cO_L, \mathfrak{f})/\{x\cO_L: x\in L_{\mathfrak{f}}\} \xrightarrow{} \Pic(R), \end{equation} whose composition with $\Pic(R) \to \Pic(\cO_L) \simeq J(\cO_L, \mathfrak{f}) / \{x\cO_L\in J(\cO_L, \mathfrak{f}):x\in L\}$ 
    is surjective.
    Then surjectivity of \eqref{eq:orderPicard} follows from the exact sequence $(\cO_L/\mathfrak{f}\cO_L)^\times/(R/\mathfrak{f}\cO_L)^\times \to \Pic(R) \to \Pic(\cO_L) \to 1$ (\cite{MR1697859}*{\MakeUppercase{\romannumeral 1}.12.9, \MakeUppercase{\romannumeral 1}.12.11}) and surjectivity of $\{x\in L^\times: x\cO_L\in J(\cO_{L}, \mathfrak{f})\} \to (\cO_L/\mathfrak{f}\cO_L)^\times$. 
\end{proof}

\begin{lemma}\label{lem:classnumber}
    Let $F$ be a totally real number field with odd narrow class number, and $\fp_1, \dotsc, \fp_r$ be distinct primes of $F$ above $2$. Let $\lambda$ be a large enough totally positive prime of $F$ such that $-\lambda$ is a square modulo $4$, and $R$ be an order of $L = F(\sqrt{-\lambda})$ with conductor $\mathfrak{f}$. If $\mathfrak{f}\cap \cO_F \mid \fp_1\cdots \fp_r$, then the class number $h(R)$ of $R$ is odd. 
\end{lemma}
\begin{proof}
    Given that $-\lambda \equiv m^2 \pmod4$ for some $m\in \mathcal{O}_F$, the maximal order of $L$ is $\cO_L = \cO_F[\frac{m+\sqrt{-\lambda}}{2}]$, and $\lambda$ is the only prime of $F$ that ramifies in the CM extension $L/F$. By \cite{MR963648}*{13.7}, the class number $h(L)$ is odd. 
    For an $\cO_F$-order $R$ of $L$, we have 
    \[h(R) = \frac{h(L)}{(\mathcal{O}_L^\times : R^\times)} \frac{\#(\mathcal{O}_L/\mathfrak{f})^\times}{\#(R/\mathfrak{f})^\times}\](see for instance \cite{MR1697859}*{\MakeUppercase{\romannumeral 1}.12.12}). By \cite{MR963648}*{13.3, 13.4, 13.5}, we have $\cO_L^\times = \mu_L\cO_F^\times$, where $\mu_L$ is the group of roots of unity in $L^\times$, so $(\cO_L^\times:R^\times) = 1$ when $\mu_L = \{\pm1\}$, which can be guaranteed if $\lambda$ is large enough.  
    From the injections $\cO_F/(\mathfrak{f}\cap \cO_F)\xhookrightarrow{}R/\mathfrak{f} \xhookrightarrow{}\cO_L/\mathfrak{f}$, it suffices to show $\frac{\#(\mathcal{O}_L/\mathfrak{f})^\times}{\#(\cO_F/(\mathfrak{f}\cap \cO_F))^\times}$ is odd. 
    Write $\mathfrak{f}\cap \cO_F = \prod_{i=1}^r\fp_i^{\varepsilon_i}$ where $\varepsilon_1, \dotsc, \varepsilon_r\in\{0,1\}$, then  $\#(\mathcal{O}_F / (\mathfrak{f}\cap \cO_F))^\times = \prod_{i=1}^r(2^{f(\fp_i/2)} - 1)^{\varepsilon_i}$. Each $\fp_i$ is unramified in $L$, then $\mathfrak{f} \mid \prod_{i=1}^r \prod_{\fq \mid \fp_i} \fq^{\varepsilon_i}$ and $\#(\mathcal{O}_L / \mathfrak{f})^\times \mid \prod_{i=1}^r\prod_{\fq|\fp_i}(2^{f(\fq/2)} - 1)^{\varepsilon_i}$. Since \[\frac{\prod_{\fq|\fp_i}(2^{f(\fq/2)} - 1)}{2^{f(\fp_i/2)} - 1} = \begin{cases}
        2^{f(\fp_i/2)} - 1& \text{if }\fp_i\text{ split in }L \\
        2^{f(\fp_i/2)} + 1& \text{if }\fp_i\text{ inert in }L
    \end{cases}\] is odd, we conclude that $h(R)$ is odd. 
\end{proof}
\begin{para} \label{order}
    When $F$ has a unique prime $\fp$ above $2$, and $\lambda$ is a totally positive prime of $F$ such that $-\lambda$ is a square modulo $4$, we will consider the CM cycles corresoponding to the optimal embeddings of the maximal order $\cO_{L}$ and a nonmaximal order $\cO_{F}+\fp\cO_L$. By  \Cref{lem:classnumber}, both orders have odd class numbers, and therefore, each corresponding CM cycle has a unique real point. 
\end{para}

\subsection{Special endomorphisms} \label{sec:special_end}

Given $\phi: L = F(\sqrt{-\lambda})\xhookrightarrow{} B$, we have $v = \phi(\sqrt{-\lambda} )\in V$ and denote by $V_v$ its orthogonal complement in $V$ as $F$-vector space.  
The inclusion $V\xhookrightarrow{} B$ induces an $F$-algebra isomorphism $C^+_F(V) \xrightarrow{\simeq} B$, under which $C^+_F(V_v) \xrightarrow{\simeq } \phi(L)$. 
Thus, we have $\Res_{F/\bQ} \GSpin(V_v, F) \xhookrightarrow{} \Res_{F/\bQ} \GSpin(V, F)$ that agrees with $\phi: L\xhookrightarrow{}B$. 

    The quadratic space $V_v\otimes_{F,\sigma}\bR$ has signature $(2,0)$ for $\sigma\neq \rho$ and $(0,2)$ for $\sigma = \rho$. Let $\{e_1, e_2\}$ be an $\bR$-basis of $V_v\otimes_{F,\rho}\bR$ such that the matrix of the bilinear form is $\begin{pmatrix}
        -1 & 0 \\ 0 & -1
    \end{pmatrix}$. The $\bR$-algebra homomorphism $\bC \to C^+_F(V_v)\otimes_{F, \rho}\bR$ defined by the oriented negative $2$-plane $\langle e_1, e_2\rangle$ (resp. $\langle e_2, e_1\rangle$) is $a+bi \mapsto a+be_1e_2$ (resp. $a+be_2e_1)$. Its restriction to $\bC^\times$ gives a homomorphism $\Res_{\bC/\bR}\bG_m\to  (\Res_{F/\bQ} \GSpin(V, F))_{\bR} $. Under the isomorphism $C^+_F(V) \otimes_{F, \rho} \bR \simeq B\otimes_{F, \rho} \bR$, we have \begin{equation}\label{eq:special_end}\{e_1e_2, e_2e_1\} = \{\frac{1}{\sqrt{\lambda}} v, -\frac{1}{\sqrt{\lambda}} v\}.\end{equation} Therefore, the homomorphism $\Res_{\bC/\bR}\mathbb{G}_m \to \widetilde{G}_{\bR}$ defined by the special endomorphism $v$ is $\tilde{h}$ or $\iota \tilde{h}$ defined above. 

\subsection{Definition field}\label{sec:def_field}

Let $\tilde{\mu} = \mu_{\tilde{h}}: \bG_{m,\bC} \to \widetilde{T}_{\bC}$ be the cocharacter defined by $\tilde{\mu}(z) = \tilde{h}_{\bC}(z,1)$. Let $T$ be the image of  $\widetilde{T}$ in $G$ and $\mu$ (resp. $h$) be the composition $\bG_{m, \bC} \xrightarrow{\tilde{\mu}} \widetilde{T}_{\bC} \to T_{\bC}$ (resp. $\Res_{\bC/\bR} \bG_m \xrightarrow{\tilde{h}} \widetilde{T}_{\bR} \to T_{\bR}$). Since $\tilde{\mu}$ corresponds to an embedding $\tilde{\rho}: L\xhookrightarrow{} \bC$ extending $\rho: F\xhookrightarrow{}\bR$, the cocharacters are defined over $\tilde{\rho}(L)$.  

Suppose $\lambda$ is an unramified prime of $F$ above $p$ and let $K'\subset \widetilde{T}(\bA_f)$ be a compact open subgroup such that $K'_p$ is maximal and $K'^p$ is small enough such that $\Sh_{K'}(T, \{h\}) := T(\bQ) \backslash \{h\} \times T(\bA_f) / K'$ is a fine moduli space  for abelian varieties with Hodge cycles and sufficiently high level structure away from $p$. By definition of canonical model, $\sigma \in \Gal(\overline{\bQ}/L)$ acts on $\Sh_{K'}(T, \{h\})$ by \[\sigma [h, t] = [h, r(T,\mu)(s)t] \text{ where } s\in \bA_{L}^\times \text{ with } \mathrm{art}(s) = \sigma|_{L^{ab}},\] and the homomorphism $r(T, \mu): \Res_{L/\bQ}\bG_m = \widetilde{T} \to T $ is the natural surjection under $T \simeq \widetilde{T}/T_F^1$. Then there exists an abelian extension $L'/L$, unramified at the primes above $p$, such that $\Gal(\overline{\bQ}/L')$ fixes $\Sh_{K'}(T, \{h\})$. 
In particular, the ramification index of $p$ in $L'$ is $2$. 
Since there are no nontrivial automorphisms of the abelian varieties with level structure parametrized by $\Sh_{K'}(T, \{h\})$, the abelian varieties are defined over $L'$. 
Moreover, since $\Gal(\overline{\bQ}/L')$ fixes $\mu$, it fixes the special endomorphism.   

\subsection{Supersinigular reduction}

\begin{lemma}
    Let $E$ be a CM field. Suppose that $A$ is an abelian variety of CM type $(E, \Phi)$ over a number field $K \subset \overline{\bQ}$. Let $(E^*, \Phi^*)$ be the reflex CM type. Suppose $E \subseteq \End^0(A/K)$ (so that $E^*\subseteq K$). Assume that $\mathfrak{p}$ is a prime of $K$ of with residue field $k$ of characteristic $p$ such that $A$ has good reduction $A_k$ at $\mathfrak{p}$. Let $\mathfrak{q} = \mathfrak{p} \cap E^{*+}$, where $E^{*+}$ is the totally real subfield of $E^*$. If $\mathfrak{q}$ is inert or ramified in $E^*$, then $A_k$ is supersingular. 
\end{lemma}
\begin{proof}
    After passing to a finite extension, we may assume $K$ contains all conjugates of $E$. Fix a $p$-adic valuation $v$ of a normal closure $\tilde{E}$ of $E$. Let $\mathfrak{q}' = \mathfrak{p}\cap E^*$.
    By Shimura-Taniyama, 
    there exists an element $\pi \in \mathcal{O}_E$ inducing the Frobenius endomorphism on $A_k$ 
    and the Newton slopes are  $\frac{v(\sigma(\pi))}{v(p^{f(\mathfrak{p}/p)})}$ for $\sigma\in \Hom(E, \overline{\bQ})$, where \[(\pi) = N_{K,\Phi}(\mathfrak{p}) = N_{\Phi}(\Nm_{K/E^*} \mathfrak{p}) = N_{\Phi}( \mathfrak{q'}^{f(\mathfrak{p}/\mathfrak{q}')}).\] 
    For a sufficiently large integer $N$, there exists $\alpha \in \mathcal{O}_{E^*}$ such that $\mathfrak{q}'^N = (\alpha)$.  Then \[\frac{v(\sigma(\pi))}{v(p^{f(\mathfrak{p}/p)})} = \frac{v(\sigma(N_{\Phi}( \alpha^{f(\mathfrak{p}/\mathfrak{q}')})))}{v(p^{Nf(\mathfrak{p}/p)})} = \frac{v(\sigma(N_{\Phi}( \alpha)))}{v(p^{Nf(\mathfrak{q}'/p)})}. \]
    With $E$ a subfield of $\overline\bQ$, we have $N_{\Phi}: E^{*\times} \to E^\times$ described by $N_{\Phi}(a) = \prod_{\psi\in \Phi^*}\psi(a)$. 
    Let $\tilde{\sigma}\in \Aut(\tilde{E})$ such that $\tilde{\sigma}|_E = \sigma$. If $\mathfrak{q}$ is inert or ramified in $E^*$, then $\fq'^N = (\alpha) = (c\alpha)$ and $v\circ \tilde{\sigma} \circ \psi (\alpha)= v\circ \tilde{\sigma}\circ c\circ \psi(\alpha)$ for any $\psi:E^*\xhookrightarrow{}\overline{\bQ}$, 
    so that $v(\sigma(N_{\Phi}(\alpha))) = v(\sigma(\prod_{\psi\in \Phi^*}\psi(a))) = v(\sigma(\prod_{\psi\in \Phi^*}c\psi(a))) = v(\sigma(cN_{\Phi}(\alpha)))$, where $c$ denotes the unique complex conjugation on the CM fields. On the other hand, $N_{\Phi}(\alpha) \cdot cN_{\Phi}(\alpha) = \Nm_{E^*/\bQ}(\alpha)$ and $v(\Nm_{E^*/\bQ}(\alpha)) = v(\Nm_{E^*/\bQ}\mathfrak{q}'^N) = v(p^{f(\mathfrak{q}'/p)N})$. Thus, the Newton slopes are all $\frac{1}{2}$. 
\end{proof}

\begin{corollary}
    Let $E = E_1\times \cdots \times E_m$ be a CM algebra and $E^*$ be its reflex field. Suppose that $A$ is a CM abelian variety over a number field $K \subset \overline{\bQ}$ with $E \subseteq \End^0(A/K)$.  Assume that $\mathfrak{p}$ is a prime of $K$ with residue field $k$ of characteristic $p$ such that $A$ has good reduction $A_k$ at $\mathfrak{p}$. Let $\mathfrak{q} = \mathfrak{p} \cap E^{*+}$, where $E^{*+}$ is the totally real subfield of $E^*$. Suppose $E^* \subseteq L$, where $L$ is a CM field and $\mathfrak{P}$ a prime of its totally real subfield $L^+$ above $\mathfrak{q}$. If $\mathfrak{P}$ is inert or ramified in $L$, then $A_k$ is supersingular. 
\end{corollary}
\begin{proof}
    The abelian variety is isogenous to $A_1\times \cdots \times A_m$ where each $A_i$ is CM by $E_i$, and after passing to a finite extension, we may assume each $A_i$ is defined over $K$ with $E_i \subseteq \End^0(A_i/K)$. Since $E_i^* \subseteq E^*$, we reduce to the case $E$ is a CM field. 

    Note that $L = L^+ E^*$ and if $\mathfrak{P}$ is inert or ramified in $L$, then $\mathfrak{q}$ is inert or ramified in $E^*$,  
    and the result follows from the lemma. 
\end{proof}

\begin{para} \label{supersingular}
    For a CM abelian variety $A$ constructed from $L = F(\sqrt{-\lambda})\xhookrightarrow{} B$, since the Hodge cocharacter factors through $\Res_{L/\bQ}\mathbb{G}_m$, elements of $\Gal(\overline{\bQ}/\bQ)$ that fixes the Hodge cocharacter fixes the CM type. It follows that $L\supseteq E^*$, and if a prime $\mathfrak{P}$ of $F$ is inert or ramified in $L$, then $A$ has supersingular reduction at primes above $\mathfrak{P}$.
\end{para}

\section{Non-Archimedean place} \label{sec:nonarch}

Recall that in \cite{MR903384}, the polynomial $P_{l}(x)P_{4l}(x)$ is a square modulo $l$ and our goal is to establish a similar result by pairing the roots of the polynomials. In this section, we prove \Cref{lem:pairwoaut} and \Cref{lem:pairingwithauto}, which together show that CM liftings of a closed point $x\in \mathscr{S}_{K_0}(k)$ occur in pairs when we consider liftings to cycles corresponding to all orders containing $\cO_F[\sqrt{-\lambda}]$. \Cref{lem:pairwoaut} is the easier case where $x$ has no self-automorphisms of even order. The main input is the description of the local deformation space of abelian varieties parametized by a $\GSpin$ Shimura variety as the deformation space of isotropic lines in \cite{MR3484114}*{5.16}, which we adapt to our setting in \Cref{lem:lifting} by considering the $F$-linear structure on $\bm{L}_{\mathrm{cris},x}$ defined by $\bm{\pi}_{{\epsilon_i}, \mathrm{cris}, x}$. In addition, we define an $F$-linear structure on the space of special quasi-endomorphisms $V(\cA_{x})$ in \cref{sec:specialendsupersingular}. When $\cA_x$ is supersingular, the structure of $V(\cA_{x})$ as a quadratic space over $F$ is described in \Cref{lem:specialend} and \Cref{lem:specialendp}.

Extra work is required when $x$ admits an even order automorphism, analogous to the case $j=1728$ in \cites{MR903384, MR1030140}. In this situation, \Cref{lem:auto} and \Cref{lem:pairend} give a pairing of the special endomorphsims, which in turn yields a pairing of liftings in \Cref{lem:pairingwithauto}. Moreover, \Cref{lem:pairingwithauto} provides necessary information on the intersection number of these cycles to address the complication caused by lack of cusps. 

Let $p$ be an odd prime unramified in $F$. Let $k = \mathbb{F}_q \subset \overline{\mathbb{F}}_p$ and $W=W(k)$ be its ring of Witt vectors. Denote by $\Art_W$ the category of Artinian local algebras over $W$. Fix an embedding $\iota: \overline{W[p^{-1}]}\xhookrightarrow{}\bC$ and let $\rho': F\xhookrightarrow{} W[p^{-1}]$ such that $\rho = \iota\circ \rho'$. 

\subsection{Lifting of abelian varieties}
Let $x\in \mathscr{S}_{K_pK^p}(k)$ be a closed point and $\widehat{U}_x$ be the completion of $\mathscr{S}_{K_pK^p}$ at $x$. We follow \cite{MR3484114}*{5.15, 5.16} to relate the infinitesimal deformations of $x$ to the liftings of certain isotropic lines in the quadratic spaces associated to $\bm{L}_{\mathrm{cris},x}$, and then apply the Grothendieck existence theorem to lift a compatible system of such deformations to a certain finite extension of $W$. 

For $\mathscr{O}$ in $\Art_W$, let $\bm{H}_\mathscr{O}\simeq H^1\cris(\mathcal{A}_x/W)\otimes_W \mathscr{O}$ be the $\mathscr{O}$-module obtained by evaluating $H^1\cris(\mathcal{A}_x/W)$ on $\Spec \mathscr{O}$, and $\bm{L}_{\mathscr{O}} \subset \End(\bm{H}_\mathscr{O})$ be the corresponding quadratic space over $\mathscr{O}$. 
If $\mathcal{A}_{\tilde{x}}$ is an abelian scheme over $\mathscr{O}$ lifting $\mathcal{A}_x$, then via the canonical identification $H^1\dR(\mathcal{A}_{\tilde{x}}/\mathscr{O}) \xrightarrow{\simeq}\bm{H}_{\mathscr{O}}$, the Hodge filtration on $H^1\dR(\mathcal{A}_{\tilde{x}}/\mathscr{O})$ corresponds to a direct summand of $\Fil^1\bm{H}_{\mathscr{O}} \subset \bm{H}_{\mathscr{O}}$. If $\mathscr{O}_2 \twoheadrightarrow \mathscr{O}_1$ is a surjection in $\Art_{W}$ whose kernel admits nilpotent divided powers and $\mathcal{A}_{x_1}$ is an abelian scheme over $\mathscr{O}_1$ lifting $x$, then this gives a natural bijection 
\begin{align} \label{eq:SerreTateGrothendieckMessing}
\left\{
\begin{array}{c}
\text{Isomorphism classes of abelian schemes over }\mathscr{O}_2 \\
\text{ lifting }\mathcal{A}_{x_1}
\end{array}
\right\}
&\xrightarrow{\;\simeq\;}
\left\{
\begin{array}{c}
\text{Direct summands } \Fil^1\bm{H}_{\mathscr{O}_2} \subset \bm{H}_{\mathscr{O}_2} \\
\text{lifting } \Fil^1\bm{H}_{\mathscr{O}_1}
\end{array}
\right\}
\end{align}
by  Serre-Tate theory (\cite{MR638600}*{1.2.1}) and Grothendieck-Messing theory (\cite{MR347836}*{\MakeUppercase{\romannumeral 5}. 1.6})\footnote{Serre-Tate theory and Grothendieck-Messing theory apply to surjections $\mathscr{O}\to \overline{\mathscr{O}}$ such that $p$ is nilpotent in $\mathscr{O}$ and the kernel $I$ admits nilpotent divided powers. 
If $k = \mathbb{F}_q\subset \overline{\mathbb{F}}_p$, $W = W(k)$, and a finite extension $E$ of $W[p^{-1}]$ has ramification index $e$ with uniformizaing parameter $\varpi$, then $(\mathcal{O}_E, (\varpi))$ has a P.D. structure iff $e\leq p-1$ (\cite{MR491705}, 3.2.3). For the liftings we will consider, where $p$ is odd and $e\leq 2$, the kernel of  $\mathcal{O}_E/(\varpi^n)\twoheadrightarrow \mathcal{O}_E/(\varpi)$ always admits nilpotent divided powers. }.

We recall Grothendieck existence theorem for the convenience of the reader. 
\begin{theorem}[Grothendieck  \cite{MR2083211}*{3.4}]\label{thm:Grothendieckexistence}
    Let $R$ be a noetherian ring which is separated and complete with respect to the $I$-adic topology for an ideal $I$. Let $X$ and $Y$ be proper $R$-schemes, and $R_n, X_n, Y_n$ the reductions modulo $I^{n+1}$. The natural map of sets $\Hom_R(X,Y)\to \varprojlim \Hom_{R_n} (X_n, Y_n)$ is bijective. 

    Moreover, if $\{X_n\}$ is a compatible system of proper schemes over the $R_n$'s and $\mathscr{L}_0$ is an ample line bundle on $X_0$ which lifts compatibly to a line bundle $\mathscr{L}_n$ on each $X_n$, then there exists a pair $(X, \mathscr{L})$ consisting of a proper $R$-scheme and ample line bundle which comaptibly reduces to each $(X_n, \mathscr{L}_n)$, and this data over $R$ is unique up to unique isomorphism. 
\end{theorem}

\begin{lemma}\label{lem:lifting}
    Suppose $p$ is an odd prime unramified in $F$. 
    Let $k = \mathbb{F}_q \subset \overline{\mathbb{F}}_p$ and  $x\in \mathscr{S}_{K_pK^p}(k)$. Suppose $E$ is a finite extension of $W[p^{-1}]$ of ramification index $e\leq p-1$ with uniformizer $\varpi$. Let \begin{align*}
        \Iso := \{\text{isotropic line } \langle w\rangle \subset \bm{L}_{\mathrm{cris},x}\otimes_W E: (\bm{\pi}_{{\epsilon_i}, \mathrm{cris}, x}\otimes 1)(w) = \rho'(\epsilon_i) w\}.
    \end{align*} 
    Then there is a natural bijection \begin{equation} \label{eq:lift}
        \left\{\tilde{x}\in \mathscr{S}_{K_pK^p}(E) \text{ lifting }x\right\} \xrightarrow{\simeq }\left\{\Fil^1 (\bm{L}_{\mathrm{cris},x}\otimes_{W} E )\in \Iso\text{ lifting }\Fil^1 \bm{L}_{\mathrm{dR}, x}\right\}.
    \end{equation}
    Moreover, if $v\in L(\mathcal{A}_x)$ is a special endomorphism, then there is a natural bijection  
    \begin{align}\label{eq:liftspecial}\left\{\begin{array}{c}\tilde{x}\in \mathscr{S}_{K_pK^p}(E) \text{ lifting }x, \\
\tilde{v}\in L(\mathcal{A}_{\tilde{x}}) \text{ lifting }v
\end{array}
\right\} \footnotemark
&\xrightarrow{\;\simeq\;}
\left\{
\begin{array}{c}
\Fil^1 (\bm{L}_{\mathrm{cris},x}\otimes_{W} E )\in \Iso\text{ lifting }\Fil^1 \bm{L}_{\mathrm{dR}, x} \\
\text{and orthogonal to }v
\end{array}
\right\}.
\end{align}\footnotetext{If $\tilde{v}$ exists, then it is unique, since $\tilde{v}\in L(\cA_{\tilde{x}})$ lifts $v\in L(\cA_x)$ and $L(\cA_{\tilde{x}}) \xhookrightarrow{} L(\cA_x)$.}
\end{lemma}
\begin{proof}
    For any lift $\tilde{x}$ of $x$, via  the comparison $H^1\dR(\mathcal{A}_{\tilde{x}}/E)\xrightarrow{\simeq}H^1\cris(\mathcal{A}_x/W)\otimes_W E$, the Hodge filtration $\Fil^1H^1\dR(\mathcal{A}_{\tilde{x}}/E)\subset H^1\dR(\mathcal{A}_{\tilde{x}}/E)$ gives a filtration $\Fil^1(H^1\cris(\mathcal{A}_{x}/W)\otimes_W E )
    \subset H^1\cris(\mathcal{A}_{x}/W)\otimes_W E$ that is split by a cocharacter $\mu: \bG_{m, E}\to G_W\otimes_W E$. Since $G_W\subset \GSpin(\bm{L}_{\mathrm{cris}, x})$, it induces a splitting \[\bm{L}_{\mathrm{cris},x} \otimes_{W}E = \Fil^1 (\bm{L}_{\mathrm{cris}, x}\otimes_W E) \oplus (\bm{L}_{\mathrm{cris}, x}\otimes_W E)^0\oplus \overline{\Fil}^1(\bm{L}_{\mathrm{cris}, x}\otimes_W E),\] where $\Fil^1 (\bm{L}_{\mathrm{cris}, x}\otimes_W E)$ is an isotropic line such that \begin{equation}\label{eq:filtration}
        \Fil^1(H^1\cris(\mathcal{A}_{x}/W)\otimes_W E) = \ker (\Fil^1(\bm{L}_{\mathrm{cris}, x}\otimes_W E)) = \im(\Fil^1(\bm{L}_{\mathrm{cris}, x}\otimes_W E)).
    \end{equation}
    Note that $\Fil^1 (\bm{L}_{\mathrm{cris}, x}\otimes_W E)$ lifts $\Fil^1 \bm{L}_{\mathrm{dR}, x}$ as $\Fil^1(H^1\cris(\mathcal{A}_{x}/W)\otimes_W E)$ lifts $\Fil^1H^1\dR(\mathcal{A}_{x}/k)$. 
    By construction, the filtration on $H^1\dR(\mathcal{A}_{\tilde{x}}/E) \otimes_{E, \iota} \bC \simeq  H^1_B(\mathcal{A}_{\iota(\tilde{x})}(\bC), \bQ)\otimes_{\bQ} \bC$ is induced by a cocharacter $\mu_{\mathbf{h}}$ defined by a Hodge structure $\mathbf{h}\in X$ such that $\Fil^1 H_{\bC} = \ker (\Fil^1 V_{\bC})$. The action of $\bm{\pi}_{{\epsilon_i}, B, \iota(\tilde{x})} \otimes 1$ on $\Fil^1 V_{\bC}$ is scalar multiplication by $\rho(\epsilon_i)$.  If $w$ is a generator of the isotropic line $\Fil^1 V_{\mathrm{dR},\tilde{x}}$, then since \eqref{eq:Betti_deRham} takes $\bm{\pi}_{{\epsilon_i}, \mathrm{dR}, \tilde{x}} \otimes 1$ to $\bm{\pi}_{{\epsilon_i}, B, \iota(\tilde{x})} \otimes 1$, we have $\bm{\pi}_{{\epsilon_i}, \mathrm{dR}, \tilde{x}} (w) \otimes_{E, \iota} 1= w\otimes_{E, \iota}\rho({\epsilon_i})$, i.e., $\bm{\pi}_{{\epsilon_i}, \mathrm{dR}, \tilde{x}} (w) = \rho'({\epsilon_i}) w$. 
    
    To show the map \eqref{eq:lift} is a bijection, we work successively with the thicknings $\mathcal{O}_E /(\varpi^n)\twoheadrightarrow \mathcal{O}_E/(\varpi^{n-1})$ following the proof of \cite{MR3484114}*{5.16}  and then apply Grothendieck existence theorem. 
    Let $\Iso_n$ denote the set of isotropic lines in  $\bm{L}_{\mathrm{cris},x}\otimes_W (\mathcal{O}_E/(\varpi^n))$ on which each $\bm{\pi}_{{\epsilon_i}, \mathrm{cris}, x}\otimes 1$  acts as $\rho'({\epsilon_i})\otimes 1$.
    Suppose $x_{n-1}\in \widehat{U}_x(\cO_E/(\varpi^{n-1}))$ gives rise to $l_{n-1}\in \Iso_{n-1}$ that lifts $\Fil^1\bm{L}_{\mathrm{dR}, x}$, and consider 
    \begin{equation} \label{eq:lift_infinitesimal}
        \left\{x_n \in \widehat{U}_x(\cO_E/(\varpi^{n})) \text{ lifting }x_{n-1} \right\} \xrightarrow{ }\left\{ l_n\in \Iso_n\text{ lifting } l_{n-1}\right\}.
    \end{equation}
    Injectivity of \eqref{eq:lift_infinitesimal} follows from \eqref{eq:SerreTateGrothendieckMessing} and the correspondence \begin{equation*}
        \Fil^1\bm{H}_{\cO_E/(\varpi^{n})} = \ker (\Fil^1(\bm{L}_{\mathrm{cris}, x}\otimes_W (\cO_E/(\varpi^{n})))) .
    \end{equation*} Then we can prove \eqref{eq:lift_infinitesimal} is a bijection by showing both sides are vector spaces over $\mathcal{O}_E/(\varpi)$ of the same dimension. 
    Since $\widehat{U}_x$ is formally smooth of relative dimension $1$ over $W$ (\cite{MR2669706}*{2.3.5}), 
    the left-hand side of \eqref{eq:lift_infinitesimal} is a $1$-dimensional vector space over $\mathcal{O}_E/(\varpi)$. On the other hand,  since $\bm{L}_{\mathrm{cris},x}$ is non-degenerate modulo $(\varpi)$, 
    there exists $l_n =\langle w_n\rangle \in \Iso_{n}$ lifting $l_{n-1}$ and the right-hand side of \eqref{eq:lift_infinitesimal} is $\{\langle w_n + \varpi^{n-1} u\rangle \subset \bm{L}_{\mathrm{cris}, x}\otimes_W (\cO_E/(\varpi^n)): \varpi^{n-1} u \perp w_n, (\bm{\pi}_{\epsilon_i, \mathrm{cris}, x}\otimes 1)(u) = \rho'(\epsilon_i)u \}\simeq \{u\in \bm{L}_{\mathrm{cris}, x}\otimes_W (\cO_E/(\varpi)): u \in l_1^{\perp}, (\bm{\pi}_{\epsilon_i, \mathrm{cris}, x}\otimes 1)(u) = \rho'(\epsilon_i)u\}/l_1$, 
    which is a $1$-dimensional $\mathcal{O}_E/(\varpi)$-vector space. Hence \eqref{eq:lift_infinitesimal} is a bijection. Applying \Cref{thm:Grothendieckexistence} to a compatible system of polarized abelian varieties associated to $\{l_n\}$, the map \eqref{eq:lift} is a bijection. 
    
    For \eqref{eq:liftspecial}, by \Cref{thm:Grothendieckexistence}, given a lift $\tilde{x}\in \mathscr{S}_{K_pK^p}(E)$ of $x$, there is $\tilde{v} \in L(\cA_{\tilde{x}})$ lifting $v$ if $v$ lifts to a compatible system $\{v_n\}$, 
    where $v_n$ is an endomorphism of the reduction of $\cA_{\tilde{x}}$ modulo $(\varpi^n)$.
    With $x_n \in \widehat{U}_x(\mathcal{O}_E/(\varpi^n))$ lifting $x$, we have $v\in L(\mathcal{A}_x)$ lifts to an endomorphism of $\mathcal{A}_{x_n}$ if and only if its crystalline realization $v_{\mathrm{cris},n}\in \bm{L}_{\cO_E/(\varpi^n)}$ preserves the Hodge filtration $\Fil^1 \bm{H}_{\cO_E/(\varpi^n)} $ if and only if $v_{\mathrm{cris},n}\in \Fil^0\bm{L}_{\cO_E/(\varpi^n)}$ if and only if $v_{\mathrm{cris},n}$ is orthogonal to $\Fil^1\bm{L}_{\cO_E/(\varpi^n)}$. 
\end{proof}

\subsection{Structure of the space of special endomorphisms} \label{sec:specialendsupersingular} 

Let $x\in \mathscr{S}_{K_pK^p}(k)$ be a closed point and $v\in L(\mathcal{A}_x)$ be a special endomorphism. By \Cref{lem:lifting} and its proof, there exists some lift $\tilde{x}$ of $x$ with special endomorphism $\tilde{v}$ lifting $v$. The abelian variety $\cA_{\iota(\tilde{x})}$ defines a Hodge structure $V\otimes_{\bQ}\bC = V^{-1,1}\oplus V^{0,0}\oplus V^{1,-1}$, where $L(\mathcal{A}_{\tilde{x}})\otimes_{\bZ}\bQ \simeq V\cap V^{0,0}$ has an $F$-linear structure defined by the absolute Hodge cycles $\bm{\pi}_{{\epsilon_i}}$. Therefore, there is an $\mathcal{O}_F$-linear structure on the sublattice $L(\mathcal{A}_{\tilde{x}}) \xhookrightarrow{} L(\mathcal{A}_x)$ such that $\bm\pi_{{\epsilon_i}} (\tilde{v})$ agrees with $\bm\pi_{{\epsilon_i}, \mathrm{cris}, x}(v\cris)$ in $\bm{L}_{\mathrm{cris},x}$. Applying this process to a basis for $L(\mathcal{A}_x)$ shows that it is stable under the $\mathcal{O}_F$-action on $\bm{L}_{\mathrm{cris},x}$, thereby endowing $L(\mathcal{A}_x)$ with an $\mathcal{O}_F$-linear structure compatible with that on $\bm{L}_{cris,x}$ and $\bm{L}_{\ell,x}$ for $\ell\neq p$.

Let $V(\cA_x) = L(\cA_x)\otimes_{\bZ}\bQ$ and $Q'$ denote the quadratic form on $V(\cA_x)$ over $\bQ$. Since $\Tr_{F/\bQ}$ is perfect, and for $u,w\in L(\cA_x)$, $[\epsilon_iu, w]_{Q'} =[\bm{\pi}_{\epsilon_i, \mathrm{cris},x} u\cris,w\cris]\cris = [u\cris, \bm{\pi}_{\epsilon_i, \mathrm{cris},x} w\cris]\cris = [u, \epsilon_i w]_{Q'}$, there is a quadratic form $Q_F'$ on $V(\mathcal{A}_x)$ over $F$ determined by $\Tr_{F/\bQ}(f[u,w]_{Q_F'}) = [fu,w]_{Q'}$ for all $f\in F$. In particular, $Q' = \Tr_{F/\bQ}\circ Q_F'$. The orthogonal direct sum decompositions \begin{align*}
        (V(\cA_x)\otimes_{\bQ} W[p^{-1}], Q') = \bigoplus_{\sigma':F\xhookrightarrow{} W[p^{-1}]} (V(\cA_x)\otimes_{F, \sigma'}W[p^{-1}], \sigma'\circ Q_F'), \\
        \bm{L}_{\mathrm{cris}, x} [p^{-1}]\simeq  (V\otimes_{\bQ} W[p^{-1}], Q)=\bigoplus_{\sigma':F\xhookrightarrow{}W[p^{-1}]}(V\otimes_{F, \sigma'} W[p^{-1}], \sigma'\circ Q_F)
\end{align*} and the $F$-linear isometric embedding \begin{equation}\label{eq:comparison_quadoverWQ}
        V(\cA_x) \otimes_{\bQ}W[p^{-1}]\xhookrightarrow{} \bm{L}_{\mathrm{cris},x}[p^{-1}]
    \end{equation}induce an isometric embedding \begin{equation}\label{eq:comparison_quadoverW}
        (V(\cA_x)\otimes_{F, \sigma'} W[p^{-1}], \sigma'\circ Q_F') \xhookrightarrow{} (V\otimes_{F, \sigma'}W[p^{-1}], \sigma'\circ Q_F)
    \end{equation}of quadratic spaces over $W[p^{-1}]$ for each $\sigma': F\xhookrightarrow{}W[p^{-1}]$. Moreover, $(V(\cA_{\tilde{x}}), Q_F) \xhookrightarrow{}(V(\cA_{x}), Q_F')$ is isometric for any lift $\tilde{x}$ of $x$. 
    Similarly, for $\ell\neq p$, we have orthogonal direct sum decompositions \begin{align*}
        (V(\cA_x)\otimes_{\bQ} \bQ_\ell, Q') = \bigoplus_{v|\ell} (V(\cA_x)\otimes_{F}F_v, \Tr_{F_v/\bQ_\ell}\circ Q_F'), \\
        \bm{V}_{\ell, x} \simeq  (V\otimes_{\bQ} \bQ_\ell, Q)=\bigoplus_{v|\ell}(V\otimes_{F} F_v, \Tr_{F_v/\bQ_\ell}\circ Q_F)
    \end{align*} of quadratic spaces over $\bQ_\ell$. Since the natural embedding \begin{equation}\label{eq:comaprison_quadoverQl}
        V(\cA_x) \otimes_{\bQ} \bQ_\ell\xhookrightarrow{} \bm{V}_{\ell,x}
    \end{equation} is $(F\otimes_{\bQ}\bQ_\ell)$-linear by construction and $\Tr_{F_v/\bQ_\ell}$ is perfect, we have an isometric embedding \begin{equation}\label{eq:comparison_quadoverFv}
        (V(\cA_x)\otimes_{F} F_v,  Q_F')\xhookrightarrow{} (V\otimes_{F}F_v,  Q_F)
    \end{equation} of quadratic spaces over $F_v$ for each $v|\ell$. The next lemma shows that \eqref{eq:comparison_quadoverWQ} -- \eqref{eq:comparison_quadoverFv} are isomorphisms when $\cA_x$ is supersingular. 
    
\begin{lemma} \label{lem:specialend}
    Suppose $\cA_x$ is supersingular. 
    \begin{enumerate}
        \item As a quadratic space over $\bQ$, $V(\cA_x)$ is positive definite with the same dimension $(= 3[F:\bQ])$ and determinant as $V$, but with Hasse invariant \[\epsilon(V(\cA_x)_{\bQ_\ell}) = \begin{cases}\epsilon(V_{\bQ_\ell}) = 1 & \text{ if }\ell\neq p \\ -\epsilon(V_{\bQ_\ell}) = -1 & \text{ if }\ell= p\end{cases}\] for all finite rational primes $\ell$. 
        \item As a quadratic space over $F$, $V(\cA_x)$ is positive definite at all real places with the same dimension $(=3)$ as $V$, and with Hasse invariant $\epsilon(V(\cA_x)_{F_v}) = \epsilon(V_{F_v}) = 1 $ for all finite primes $v\nmid p$.
    \end{enumerate}
\end{lemma}
\begin{proof}
    Positive-definiteness follows from \cite{MR3484114}*{5.12}. 
    The Tate-conjecture (\cite{MR3370622}*{6.4}) implies that under the natural isometric embedding $V(\cA_x)_{\bQ_p} = L(\cA_x)\otimes_{\bZ} \bQ_p \xhookrightarrow{} \bm{L}_{\mathrm{cris}, x}\otimes_W W(\bar{k})[p^{-1}]$,   the subspace $V(\cA_x)_{\bQ_p}$ consists of Frobenius invariant vectors in the isocrystal $\bm{L}_{\mathrm{cris}, x}\otimes_W W(\bar{k})[p^{-1}]$, which is a $\bQ_p$-quadratic space with the same dimension and determinant as $V_{\bQ_p}$, but has Hasse invariant $-1$ by \cite{MR3705249}*{4.2.5}. For $\ell\neq p$, the natural isometric embedding $V(\mathcal{A}_x) \otimes_\bQ \bQ_\ell \xhookrightarrow{} \bm{V}_{\ell,x} \simeq V\otimes_{\bQ}\bQ_\ell$ is an isomorphism for the dimension reason. 

    Since $Q' = \Tr_{F/\bQ} \circ Q'_F$, we have an orthogonal direct sum decomposition \[V(\cA_x)\otimes_{\bQ} \bR = \oplus_{\sigma:F\xhookrightarrow{}\bR} V(\cA_x)\otimes_{F, \sigma} \bR\] of quadratic space spaces over $\bR$, then positive definiteness of $ Q_F'$ follows from positive-definiteness of $Q'$. Similarly, if $\ell\neq p$, then the natural map $V(\cA_x)\otimes_{\bQ}\bQ_\ell \xrightarrow{\simeq} \bm{V}_{\ell,x} \simeq V\otimes_{\bQ}\bQ_\ell$ is an $F$-linear isometry of $\bQ_\ell$-quadratic spaces; therefore, $V(\cA_x)\otimes_{F} F_v \simeq V\otimes_{F} F_v$ for all prime $v|\ell$. 
\end{proof}

\begin{lemma}\label{lem:specialendp}
    Suppose $p$ is totally split in $F$.
    If $\cA_x$ is supersingular, then $V(\cA_x)$ is a quadratic space over $F$ of dimension $3$ and determinant $1$. Morevoer, the Hasse invariant of $V(\cA_x)$ is $-1$ at $\fp$ over $p$ corresponding to $\rho': F\xhookrightarrow{}W[p^{-1}]$, and $1$ at other finite primes of $F$. 
\end{lemma}
\begin{proof}
    When $p$ is totally split in $F$, fix a $\bZ_p$ basis $x_1, \dotsc, x_n$ ($n= 3[F:\bQ]$) of $L\otimes_{\bZ}\bZ_p \simeq \oplus_{\sigma':\mathcal{O}_F\to \bZ_p} L\otimes_{\mathcal{O}_F, 
    \sigma'}\bZ_p$ for which $x_1, x_2, x_3\in L\otimes_{\mathcal{O}_F, 
    \rho'}\bZ_p$ and the matrix of inner products has the form \[
    \begin{pmatrix}
    0 & 1 &&&&\\
    1 & 0 &&&&\\
     & & * &&&\\
     &&& * &&\\
     & &&& \ddots&\\
     &&&&& *
    \end{pmatrix}.
    \] 
    As in \cite{MR3705249}*{4.2.1}, this defines a $G_{\bZ_p}$-valued cocharacter $\mu$ by $\mu(t) = t^{-1} x_1x_2 + x_2x_1$ such that \[\mu(t) \cdot x_i = \begin{cases}
        t^{-1}x_i & i = 1,\\ tx_i & i = 2, \\x_i & 3\leq i \leq n. 
    \end{cases}\] 
    Under the fixed isomorphism $\iota: \overline{W[p^{-1}]}\to \bC$, the cocharacters $\mu$ and $\mu_{h}^{-1}$ are conjugate, where $\mu_h$ is the Hodge cocharacter.  
    
    Following \cite{MR3705249}*{4.2.5, 4.2.6}, since the derived group of $G$ is $G^{\der} \simeq  \Res_{F/\bQ} \SL_{1,B}$, which is simply connected, and the embedding $G\xhookrightarrow{} \GSpin(V, Q)$ induces $G/G^{\der} \simeq \GSpin(V, Q)/\GSpin(V, Q)^{\der}\simeq \mathbb{G}_m$, if we set $b = x_3(p^{-1}x_1+x_2)\in G(\bQ_p)\subset \GSpin(V, Q)(\bQ_p)$, then 
    the isocrystal $\bm{L}_{\mathrm{cris}, x}\otimes_W K$ is isomorphic to the isocrystal structure on $V_K$ defined by $\Phi = b\circ \sigma$, where $\sigma$ is the automorphism of $K=W(\overline{\bF}_p)[p^{-1}]$ induced by the absolute Frobenius on $ \overline{\bF}_p$.  
    If we define $M = V\otimes_{F, \rho'} \bQ_p = \bQ_p x_1+\bQ_p x_2+\bQ_p x_3$, then $V_K^{\Phi} = M_K^{\Phi}\oplus M^{\perp}$, where $M^{\perp} = \oplus_{\sigma'\neq \rho'} V\otimes_{F, \sigma'} \bQ_p$. Moreover, $M_K^{\Phi}$ and $M$ have the same dimension and determinant, but different Hasse invariants. In particular, as a quadratic space over $F_{\lambda}\simeq \bQ_p$, $V(\mathcal{A}_x)\otimes_F F_{\lambda}$ has dimension $3$, determinant $1$ and Hasse invariant $-1$. 
\end{proof}

By definition, we have \begin{equation}
    L(\cA_x) = \{v\in V(\cA_x): v\cris \in \bm{L}_{\mathrm{cris}, x}, \, v_\ell\in \bm{L}_{\ell,x} \text{ for all }\ell\neq p\}. 
\end{equation} Then we have $L(\cA_x) \otimes_\bZ W \simeq \bm{L}_{\mathrm{cris},x}$ and $L(\cA_x) \otimes_{\bZ} \bZ_\ell \simeq \bm{L}_{\ell,x}$ for $\ell\neq p$ when $\cA_x$ is supersingular.

\begin{lemma} \label{lem:auto}
    Let $p$ be a rational prime that is totally split in $F$.
    Suppose $\cA_x$ is supersingular and there is an order $2$ automorphism $\alpha\in \SO (V(\cA_x), Q_F')$ such that $\alpha(\bm{L}_{\mathrm{cris},x}) = \bm{L}_{\mathrm{cris}, x}$ under the identification $L(\cA_x) \otimes_{\bZ}W\simeq \bm{L}_{\mathrm{cris},x}$ and $\alpha (\bm{L}_{\ell,x}) = \bm{L}_{\ell,x}$ under the identification $L(\cA_x)\otimes_{\bZ}\bZ_\ell \simeq \bm{L}_{\ell,x}$ for all $\ell\neq p$, then there is $u\in L(\cA_x)$ such that $Q_F'(u) = 1$ and $\alpha(u) = u$.
\end{lemma}
\begin{proof}
    By \Cref{lem:specialend} and \Cref{lem:specialendp}, there exists a totally definite quaternion algebra $B'$ over $F$ that is ramified at $\fp$ over $p$ corresponding to $\rho': F\xhookrightarrow{} W[p^{-1}]$ and an isometry of quadratic spaces \begin{equation*} 
        f: (B'^{0}, \nrd|_{B'^{0}})\xrightarrow{\sim} (V(\mathcal{A}_x), Q_F').
    \end{equation*} Since $\SO(B'^0, \nrd|_{B'^{0}}) = B'^\times/F^\times$ with $B'^\times$ acts on $B'^0$ by conjugation, there exists $\beta \in B'$ such that $\alpha(x) = f(\beta f^{-1}(x) \beta^{-1}) $  for all $x\in V(\cA_x)$. The automorphism $\alpha$ has order $2$, so $\beta^2 \in F$ and $\beta\notin F$, which implies $\beta \in B'^0$. Let $u' = f(\beta) \in V(\cA_x)$. Note that $\alpha(cu') = cu'$ for all $c\in F$. 

    For each $\sigma': F\xhookrightarrow{}W[p^{-1}]$, let \[g_{\sigma'}: (V(\cA_x)\otimes_{F, \sigma'} W[p^{-1}], \sigma'\circ Q_F') \xrightarrow{\sim} (V\otimes_{F, \sigma'}W[p^{-1}], \sigma'\circ Q_F)\] be the isometry  \eqref{eq:comparison_quadoverW} over $W[p^{-1}]$, 
    and \[u_{\sigma'} = g_{\sigma'}(u'\otimes 1) \in V\otimes_{F,\sigma'} W[p^{-1}] = (B\otimes_{F, \sigma'}{W[p^{-1}}])^0,\] then $g_{\sigma'}\circ \alpha\circ g_{\sigma'}^{-1}$ acts as $1$ on $\langle u_{\sigma'}\rangle$ and $-1$ on $\langle u_{\sigma'}\rangle^{\perp}$, so $g_{\sigma'}\circ \alpha\circ g_{\sigma'}^{-1}(x) = u_{\sigma'} xu_{\sigma'}^{-1}$ for all $x\in V\otimes_{F, \sigma'} W[p^{-1}]$. Since the automorphism preserves $\bm{L}_{\mathrm{cris}, x} \simeq  L\otimes_{\bZ} W=\bigoplus_{\sigma':F\xhookrightarrow{}W[p^{-1}]}L\otimes_{\cO_F, \sigma'} W $, where $L = B^0\cap \cO$, 
    and $\cO\otimes_{\cO_F, \sigma'}W$ is a maximal order in $B\otimes_{F, \sigma'}{W[p^{-1}}]\simeq M_2(W[p^{-1}])$, we have
    \[u_{\sigma'} \in N_{(B\otimes_{F, \sigma'}{W[p^{-1}}])^\times}(\cO\otimes_{\cO_F, \sigma'}W) = W[p^{-1}]^\times (\cO\otimes_{\cO_F, \sigma'}W)^\times.\]  In particular, $\sigma'(Q_F'(u')) = \nrd(u_{\sigma'})\in W[p^{-1}]^\times$ has even valuation, and $u_{\sigma'} \in L\otimes_{\cO_F, \sigma'} W$ if and only $\nrd(u_{\sigma'})$ has zero valuation. Here since $p$ splits in $F$, each $\sigma'$ corresponds to a unique place of $F$ above $p$.

    For $v\nmid p$, let \[ g_v:  (V(\cA_x)\otimes_{F} F_v,  Q_F')\xrightarrow{\sim} (V\otimes_{F}F_v,  Q_F) \] be the isometry \eqref{eq:comparison_quadoverFv} over $F_v$, and \[u_{v}  = g_v(u'\otimes1)\in V\otimes_{F}F_v = B_v^0,\] then $g_v\circ \alpha\circ g_v^{-1}(x) = u_vxu_v^{-1}$ for all $x\in V\otimes_F F_v$. Similarly as before, since the automorphism preserves $\bm{L}_{\ell,x} \simeq L\otimes_{\bZ}\bZ_\ell = \oplus_{v|\ell} L\otimes_{\cO_F}\cO_{F_v}$, we have \[u_{v}\in N_{B_v^\times}(\cO_v) = F_v^\times \cO_v^\times.\] In particular, $Q_F'(u') = \nrd(u_v) \in F_v^\times$ has even valuation, and $u_v\in L\otimes_{\cO_F}\cO_{F_v}$ if and only if $\nrd(u_v)$ has zero valuation. 

    Since $h(F) = 1$ and $Q'_F(u')$ has even valuations at all finite places of $F$, there exists $c_1\in F^\times$ such that $Q_F'(c_1u') = c_1^2 Q_F'(u)\in \cO_F^\times$. Moreover, since $Q_F'$ is positively definite, and every totally positive unit in $F$ is a square ($F$ has units of independent signs), there exists $c_2\in \cO_F^\times$ such that $Q_F'(c_1c_2u') = 1$. Let $u = c_1c_2 u'$, then $u\in L(\cA_x)$ since $\sum_{\sigma':F\xhookrightarrow{}W[p^{-1}]}c_1c_2 u_{\sigma'} \in \bm{L}_{\mathrm{cris},x}$ and $\sum_{v|\ell} c_1c_2 u_v \in \bm{L}_{\ell,x}$ for all $\ell\neq p$. 
\end{proof}

\subsection{Pairing at points without even order automorphisms}

Let $T$ be an $\mathscr{S}_{K_pK^p}$-scheme, then $x\in \mathscr{S}_{K_pK^p}(T)$ gives a triple $(\cA_x, \lambda_x, \bar{\eta}_x)$, where $(\cA_x, \lambda_x: \cA_x\to \cA_x^\vee)$ is a polarized abelian scheme over $T$ up to isomorphism 
and the level structure $\eta\in \Gamma(T, \Isom_G(H_{\widehat{\bZ}^p} , \varprojlim_{p\nmid n}\cA_x[n])/K^p)$. 
An integral model $\mathscr{S}_{K_0}$ of $\Gamma_0 \backslash X^+$ is constructed from  $\mathscr{S}_{K_pK^p}$ by action of a finite group.
Two points $x_1, x_2 \in \mathscr{S}_{K_pK^p}(\bC)$, corresponding to $(\cA_{x_1}, \lambda_1, \bar{\eta}_1), (\cA_{x_2}, \lambda_2, \bar{\eta}_2)$, are identified in $\mathscr{S}_{K_0}$ if there is a prime-to-$p$ quasi-isogeny $\cA_{x_1} \to \cA_{x_2}$ that preserves polarization up to $\bZ_{(p)}^\times$, sends $s_{\alpha, B, x_1}$ to $s_{\alpha, B, x_2}$ and  takes $\bm{L}_{\ell,x_1}$ to $\bm{L}_{\ell,x_2}$ for all $\ell\neq p$. 

\begin{proposition} \label{lem:pairwoaut}
    Let $\lambda \in \cO_F$ be a totally positive prime of $F$ above $p$ and $\phi: F(\sqrt{-\lambda}) \xhookrightarrow{}B$ defining a CM cycle \eqref{eq:CMcycle}. For $x\in \mathscr{S}_{K_0}(k)$ without even order automorphisms, \footnote{The automorphisms are defined over $\bar{k}$. } points in the CM cycle lifting $x$ are paired, i.e., there is a map $$\{\text{lifts of }x\text{ in } \eqref{eq:CMcycle}  \}\longrightarrow \{\langle v\rangle : v\in L(\cA_x), Q_F'(v) = \lambda\}/\Aut(x)$$ such that the fiber of every element in the image has size $2$. 
\end{proposition}
\begin{proof}
    Suppose $\tilde{x}$ is a point in the CM cycle lifting $x$, then $\tilde{x}$ corresponds to an embedding $F(\sqrt{-\lambda})\xhookrightarrow{}B$. As in \cref{sec:special_end}, the image of $\sqrt{-\lambda}$ has reduced trace zero and reduced norm $\lambda$, and it gives a special endomorphism $\tilde{v}\in L(\cA_{\tilde{x}})$ such that $Q_F(\tilde{v}) = \lambda$. Then $\cA_x$ is supersingular by \ref{supersingular}. Let $v$ be the image of $\tilde{v}$ under the $\cO_F$-linear isometric map $L(\cA_{\tilde{x}})\xhookrightarrow{}L(\cA_x)$ and $v\cris$ be the image of $v$ under the $\cO_F$-linear isometric map $L(\cA_{x})\xhookrightarrow{}\bm{L}_{\mathrm{cris},x}$, then $\Fil^1\bm{L}_{\mathrm{dR}, x}\in v\cris^{\perp}$ over $k$. By \eqref{eq:orthogonaldecomp}, we have an orthogonal direct sum decomposition $v\cris = \sum_{\sigma': F\xhookrightarrow{}W[p^{-1}]} v_{\mathrm{cris}, \sigma'}$ and $Q_{\mathrm{cris},x}(v_{\mathrm{cris},\sigma'})=\sigma' (Q_F(\tilde{v})) = \sigma'(\lambda)$, where $Q_{\mathrm{cris}, x}$ denotes the quadratic form on $\bm{L}_{\mathrm{cris},x}$. By \cref{sec:def_field}, the abelian varieties on the CM cycle are defined over a field whose ramification index over $p$ is at most $2$, so we can apply \Cref{lem:lifting}. There are exactly two lifts $\cA_{\tilde{x}_1}, \cA_{\tilde{x}_2}$ with $\tilde{v}_1\in L(\cA_{\tilde{x}_1}), \tilde{v}_2\in L(\cA_{\tilde{x}_2})$ lifting $v$, corresponding to the two isotropic lines lifting $\bm{L}_{\mathrm{dR},x}$ that are orthogonal to $(v_{\mathrm{cris},\rho'}\otimes 1)$ in the subspace of $\bm{L}_{\mathrm{cris},x}\otimes_W\cO_E$ where each $\bm{\pi}_{\epsilon_i, \mathrm{cris},x}$ acts as $\rho'(\epsilon_i)\otimes 1$, by the linear algebra computation \Cref{lem:isopair}. 
    
    If $\cA_{\tilde{x}_1}$ and $\cA_{\tilde{x}_2}$ give the same point in $\Gamma_0\backslash X^+$, then there is a quasi-isogeney $f: \cA_{\tilde{x}_1} \to \cA_{\tilde{x}_2}$ that sends $s_{\alpha, ?,\tilde{x}_1}$ to $s_{\alpha, ?,\tilde{x}_2}$, where $? = \ell, p, \mathrm{dR}$, and takes $\bm{L}_{\ell,\tilde{x}_1}$ to  $\bm{L}_{\ell,\tilde{x}_1}$ and $\bm{L}_{p,\tilde{x}_1}$ to  $\bm{L}_{p,\tilde{x}_1}$. In particular, $f(v_1) = f\circ v_1 \circ f^{-1} \in V(\cA_{\tilde{x}_2})\cap (\prod_{\ell \text{ prime}} \bm{L}_{\ell,\tilde{x}_2}) = L(\cA_{\tilde{x}_2})$. 
    Then $f$ induces a quasi-isogeny of $\mathcal{A}_{x}$ such that $f:H^1\cris(\mathcal{A}_{x}/W) \otimes_{W} E \to H^1\cris(\mathcal{A}_{x}/W)\otimes_{W} E$ preserves all $s_{\alpha, \mathrm{cris}, x}$,  and $f(\Fil^1 H^1\dR(\mathcal{A}_{\tilde{x}_1}/E)) = \Fil^1 H^1\dR(\mathcal{A}_{\tilde{x}_2}/E) $, 
    which gives a $\cO_F$-linear isometry of $\bm{L}_{\mathrm{cris}, x} \otimes_{W} \cO_E$ taking $\Fil^1 \bm{L}_{\mathrm{dR}, \tilde{x}_1}$ to $\Fil^1\bm{L}_{\mathrm{dR}, \tilde{x}_2}$.
    The isotropic line $\Fil^1\bm{L}_{dR, \tilde{x}_2} \in (v\cris \otimes 1)^{\perp} \cap f(v\cris\otimes 1)^{\perp}$ and it lies in the $3$-dimensional $\rho'$-eigenspace of $\bm{L}_{\mathrm{cris},x}$. There is an isometry from $L(\cA_x)\otimes_{\cO_F, \rho'} W$ to the $\rho'$-eigenspace of $\bm{L}_{\mathrm{cris},x}$, where $v\otimes 1$ (resp. $f(v)\otimes 1$) maps to $v_{\mathrm{cris},\rho'}$ (resp. $f(v)_{\mathrm{cris},\rho'}$). Then $v\otimes 1$ and $f(v)\otimes1$ must span the same line in $L(\cA_x)\otimes_{\cO_F, \rho'}W$ since otherwise $(v\otimes 1)^{\perp} \cap (f(v)\otimes 1)^{\perp} \subset L(\cA_x)\otimes_{\cO_F, \rho'}W$ is a one-dimensional subspace defined in $L(\cA_x)$, which cannot be isotropic.  Then $f (\Fil^1 \bm{L}_{\mathrm{dR}, \tilde{x}_2}) \in f(v_{\mathrm{cris},\rho'}\otimes1)^\perp = (v_{\mathrm{cris},\rho'}\otimes1)^{\perp}$ implies $f (F^1 \bm{L}_{dR, \tilde{x}_2}) = F^1 \bm{L}_{dR, \tilde{x}_1}$. Therefore, $f$ fixes the line spanned by $v_{\mathrm{cris},\rho'}\otimes 1$ and switches the two isotropic lines $\Fil^1 \bm{L}_{\mathrm{dR}, \tilde{x}_1}, 
    \Fil^1 \bm{L}_{\mathrm{dR}, \tilde{x}_2}$. In particular, $f$ has even order as an automorphism of $x$. 

    If $\cA_{\tilde{x}_1}, \cA_{\tilde{x}_2}$ (resp. $\cA_{\tilde{x}_1'}$, $\cA_{\tilde{x}_2'}$) correspond to lifting $v$ (resp. $v'$) in $L(\cA_x)$, and $\cA_{\tilde{x}_1}$ and $\cA_{\tilde{x}_1'}$ give the same point in $\Gamma_0\backslash X^+$, then there is a quasi-isogeney $f: \cA_{\tilde{x}_1} \to \cA_{\tilde{x}_1'}$ that sends $s_{\alpha, ?,\tilde{x}_1}$ to $s_{\alpha, ?,\tilde{x}_1'}$, where $? = \ell, p, \mathrm{dR}$, and takes $\bm{L}_{\ell, \tilde{x_1}}$ to $\bm{L}_{\ell, \tilde{x_1}'}$ and $\bm{L}_{p, \tilde{x_1}}$ to $\bm{L}_{p, \tilde{x_1}'}$. By the same argument, $v'\otimes 1$ and $f(v)\otimes 1$ must span the same line in $L(\cA_x)\otimes_{\cO_F, \rho'} W$, and then $f$ takes $\Fil\bm{L}_{\mathrm{dR}, \tilde{x}_2}$ to $\Fil\bm{L}_{\mathrm{dR}, \tilde{x}_2'}$, which implies that  $\cA_{\tilde{x}_2}$ and $\cA_{\tilde{x}_2'}$ give the same point in $\Gamma_0\backslash X^+$. 
\end{proof}

\begin{lemma} \label{lem:isopair}
    Let $k\subset \overline{\bF}_p$ be a field of odd characteristic $p$. Let $W = W(k)$ and  $(L, Q)$ be a rank $3$ quadratic space over $W$ such that $L\otimes_W k$ is nondegenerate. If $v\in L$ such that $Q(v) = t$ is a uniformizer of $W$, then there are exactly two isotropic lines in the $2$-dimensional subspace $(v\otimes 1)^\perp \subset L\otimes_W \overline{W[p^{-1}]}$. The two isotropic lines can be defined in $L\otimes_W \cO_E$ where $E$ is a finite extension of $W[p^{-1}]$ of ramification index $2$, and reduce to the same line modulo $\fp$, where $\fp$ is the maximal ideal of $\cO_E$.  
\end{lemma}
\begin{proof}
    Since $W$ is a local PID and $2\in W^\times$, the quadratic form $Q$ over $W$ is diagonalizable. 
    Let $e_1, e_2, e_3$ be an orthogonal basis of $L$ with $Q(e_i) = a_i, \, i =1,2,3$. By assumption that $L\otimes_W k$ is nondegenerate, each $a_i\in W^\times$. Suppose $xe_1+ye_2+ze_3$ is a generator of an isotropic line orthogonal to $v = \sum_{i=1}^3 c_i e_i,\, c_i \in W$, then we have $a_1x^2 + a_2y^2 + a_3 z^2 = 0$ and $a_1c_1 x+ a_2 c_2 y+ a_3c_3z =0 $ with $\sum_{i=1}^3 a_ic_i^2 = t$. Note that at least two $c_i\in W^\times$ since $t$ is a uniformizer and $a_i \in W^\times$. Without loss of generality, assume $c_1, c_2\in W^\times$. Solving the equation gives two lines generated by $(-a_1a_2a_3c_1c_2\pm a_3c_3 \sqrt{-a_1a_2a_3t} )e_1 + a_1a_3(t-a_2c_2^2) e_2 + (\mp a_1c_1\sqrt{-a_1a_2a_3 t} - a_1a_2a_3 c_2 c_3) e_3$ defined over $E := W[p^{-1}](\sqrt{-a_1a_2a_3t})$ where $-a_1a_2a_3t$ is a uniformizer of $W$. Note that $a_1a_3(t-a_2c_2^2) \in W^\times$, and modulo the maximal ideal of $\cO_E$, both lines reduce to $\langle c_1e_1 + c_2e_2 + c_3 e_3\rangle = \langle v\rangle$. 
\end{proof}

\subsection{Pairing at points with an even order automorphism}

\begin{proposition}\label{lem:pairingwithauto}
    Let $p$ be a rational prime that is totally split in $F$ and $\lambda$ be the totally positive prime of $F$ above $p$ corresponding to $\rho': F\xhookrightarrow{}W[p^{-1}]$. \begin{enumerate} 
    \item For $x\in \mathscr{S}_{K_0}(k)$ with an even order automorphism, there is a bijection $$\{\text{lifts of }x \text{ admitting a special endomorphism }\sqrt{-\lambda}\}\longrightarrow \{\langle v\rangle: v\in L(\cA_x), \, Q_F'(v) = \lambda\}/\Aut(x) $$ where the elements on the right-hand side are paired such that $\langle v\rangle$ and $\langle w\rangle$ are paired if and only if $v\perp w$. In particular, all CM lifts of $x$ arising from all optimal embeddings $S\xhookrightarrow{} \cO$, where $S\subset F(\sqrt{-\lambda})$ is an order containing $\sqrt{-\lambda}$, are paired. 
    \item Moreover, such CM lifts are defined over an extension $E$ of $W(k)[p^{-1}]$ of ramification index $2$ and the two lifts in a pair are isomorphic modulo $\fp^2$, but they are not isomorphic to the lift with special endomorphism $\sqrt{-1}$ modulo $\fp^2$, where $\fp$ is the maximal ideal of $\cO_E$. 
    \end{enumerate}
\end{proposition}

The proof of \Cref{lem:pairingwithauto} uses \Cref{lem:pairend} and the linear algebra computation \Cref{lem:nondeg}. The proof of \Cref{lem:pairend} uses the linear algebra computation \Cref{lem:quad_global}. 

\begin{lemma} \label{lem:quad_global}
        Let $F$ be a totally real number field and $\mathcal{O}_F$ be its ring of integers. Let $\lambda$ be a large enough totally positive prime of $F$ ($\lambda > 4$ at all real places of $F$).\footnote{Given a totally positive number $\lambda\in F$ such that $\Nm_{F/\bQ}(\lambda)$ is large enough, there exists $u \in \mathcal{O}_F^\times$ such that $\rho_i(u^2) > \frac{4}{\rho_i(\lambda)}$ for each $\rho_i: F\xhookrightarrow{}\bR$; thus, after replacing $\lambda$ by $u^2\lambda$, $\rho_i(\lambda)>4$  is not an extra condition for us. The existence of such $u$ follows from Dirichlet's unit theorem. Indeed, since the image $\Gamma$ of $\mathcal{O}_F^\times \to \mathbb{R}^{r}, \,x\mapsto (-\log|\rho_i(x)|)_{\rho_i:F\xhookrightarrow{}\bR}$ is a complete lattice in the trace zero hyperplane $H =\{(x_{\rho_i}): \sum_{i=1}^r x_{\rho_i} = 0\}$, where $r=[F:\bQ]$ for $F$ totally real, we pick a basis $\gamma_1 = (\gamma_{1,1}, \dotsc, \gamma_{1, r-1}, -\sum_{i=1}^{r-1}\gamma_{1,i}), \dotsc, \gamma_{r-1} = (\gamma_{r-1,1}, \dotsc, \gamma_{r-1, r-1}, -\sum_{i=1}^{r-1}\gamma_{r-1,i})$ for the rank $(r-1)$ lattice $2\Gamma \subset H$. Write $a_i$ = $\log(\rho_i(\lambda))-\log(4)$ for each $i=1, \dotsc, r$ and $a=\sum_{i=1}^r a_i = \log \Nm_{F/\bQ}(\lambda) - r\log(4)$. Consider $\sum_{i=1}^{r-1}c_i\gamma_i = (a_1 -\frac{a}{2(r-1)}, \dotsc, a_{r-1} -\frac{a}{2(r-1)}, a_r -\frac{a}{2}) \in H$ with $c_1, \dotsc, c_{r-1} \in \bR$, and $\gamma = \sum_{i=1}^{r-1}[c_i]\gamma_i$ where each $[c_i]\in \bZ$ with $|[c_i] - c_i|\leq \frac{1}{2}$. 
        If $a > (r-1) \max_{1\leq j \leq r-1} \sum_{i=1}^{r-1}|\gamma_{i,j}|$, then the cube $C = \{(x_{\rho_i}): \sum_{i=1}^r x_{\rho_i} = 0, \, a_j -\frac{a}{r-1} < x_{\rho_j} < a_j\text{ for each }j=1, \dotsc, r-1\}\subset H$ contains $\gamma\in 2\Gamma$.} Suppose $Q: L \to \mathcal{O}_F$ is a quadratic module over $\mathcal{O}_F$ such that the extension $Q: L \otimes_{\mathcal{O}_F} F \to F$ is a quadratic form anisotropic at $\lambda$ and at all real places of $F$. 
        \begin{enumerate}
            \item If $Q(e_1) = 1$ and $ Q(e_2) = \lambda$ where $e_1, e_2\in L$, then $e_1\perp e_2$.
            \item Suppose $L$ has rank $3$. If $Q(e_1) = 1, Q(e_2) =Q(e_3) =\lambda$, where $e_1, e_2, e_3\in L$ and $e_2, e_3$ are linearly independent, and $\alpha\in \SO(L, Q)$, then $\alpha(e_1)  = \pm e_1$. 
        \end{enumerate}
\end{lemma}
\begin{proof}
    \begin{enumerate}
        \item Let $n = Q(e_1+e_2)-Q(e_1)-Q(e_2)$. Consider $f(x) = Q(e_2-xe_1) = x^2-nx+\lambda \in \mathcal{O}_F[x]$. Note that $e_1, e_2$ are linearly independent. By Hensel's lemma, since $f(0) = \lambda$ and $ f'(0) = -n$, $f$ has no roots in $F_{\lambda}$ implies $n=a\lambda$ for some $a\in\mathcal{O}_F$. On the other hand, $f$ has no roots at all real places of $F$ implies that $n^2-4\lambda$ is totally negative. Then $a^2 < \frac{4}{\lambda}<1$ at all real places of $F$, which gives $a = 0$. 
        \item Since $Q(\alpha(e_1)) = 1$, we have $\alpha(e_1) \perp e_2, \alpha(e_1)\perp e_3$, which implies $\alpha(e_1) \in \Span(e_1)$.
    \end{enumerate}
\end{proof}

\begin{lemma} \label{lem:pairend}
    Let $p$ be a large enough rational prime that is totally split in $F$ and $\lambda$ be a totally positive prime of $F$ above $p$ corresponding to $\rho': F\xhookrightarrow{}W[p^{-1}]$. Let $x\in \mathscr{S}_{K_0}(k)$. Suppose $\cA_x$ is supersingular, and $x$ has an even order automorphism, then there exists $e_1 \in L(\cA_x)$ such that $Q_F'(e_1) = 1$ by \Cref{lem:auto}. If there is $e_2\in L(\cA_x)$ with $Q_F'(e_2) = \lambda$, then there is $e_3 \in L(\cA_x)$ orthogonal to both $e_1$ and $e_2$ such that $Q_F'(e_3) =\lambda$. Moreover, there is no $\alpha\in \SO(L(\cA_x), Q_F')$ such that $\alpha(e_2) = \pm e_3$. 
\end{lemma}
\begin{proof}
    Note that $e_2\perp e_1$ by \Cref{lem:quad_global}.
    Since $Q'_F$ is a ternary quadratic form over $F$ with determinant $1\in F^\times/F^{\times 2}$, there exists a quaternion algebra $B'$ over $F$ with an isometry $(B'^{0}, \nrd|_{B'^{0}})\simeq (V(\mathcal{A}_x), Q_F')$. Under this identification, let $e_3 = e_1e_2$ where multiplication is taken in the quaternion algebra $B'$. Since $e_1\perp e_2$ and $Q'_F(e_1) = 1, Q'_F(e_2) = \lambda$, we have $e_3\in V(\mathcal{A}_x)$ and $Q'_F(e_3) = \lambda$. 

    For each $\sigma': F\xhookrightarrow{}W[p^{-1}]$, let \[g_{\sigma'}: (V(\cA_x)\otimes_{F, \sigma'} W[p^{-1}], \sigma'\circ Q_F') \xrightarrow{\sim} (V\otimes_{F, \sigma'}W[p^{-1}], \sigma'\circ Q_F) = ((B\otimes_{F, \sigma'}W[p^{-1}])^0,  \nrd)\] be the isometry  \eqref{eq:comparison_quadoverW} over $W[p^{-1}]$, and for $v\nmid p$, let \[ g_v:  (V(\cA_x)\otimes_{F} F_v,  Q_F')\xrightarrow{\sim} (V\otimes_{F}F_v,  Q_F) = (B_v^0, \nrd) \] be the isometry \eqref{eq:comparison_quadoverFv} over $F_v$. 
    Each isometry $g_{\sigma'}$ extends to a $W[p^{-1}]$-algebra isomorphism or anti-isomorphism $B'\otimes_{F, \sigma'}W[p^{-1}] \xrightarrow{\sim} B\otimes_{F,\sigma'} W[p^{-1}]$. Under this isomorphism, since $e_1, e_2\in L(\cA_x)$, we have $g_{\sigma'}(e_1\otimes1), g_{\sigma'}(e_2\otimes 1) \in \mathcal{O}\otimes_{\cO_F,\sigma'} W$, and then $g_{\sigma'}(e_3 \otimes 1) = \pm g_{\sigma'}(e_1 \otimes 1)g_{\sigma'}(e_2 \otimes 1)\in\mathcal{O}\otimes_{\cO_F,\sigma'} W\cap (B\otimes_{F,\sigma'} W[p^{-1}])^0 = L\otimes_{\cO_F, \sigma'} W$. Thus, $\sum_{\sigma'} g_{\sigma'} (e_3\otimes 1)\in \bm{L}_{cris,x}$. Similarly, for every $v\nmid p$, the isometry $g_v$ extends uniquely to an $F_v$-algebra isomorphism or anti-isomorphism $B'_{v} \xrightarrow{\sim} B_{v}$. 
    Under this isomorphsim, since $e_1, e_2\in L(\mathcal{A}_x)$, we have $g_v(e_1\otimes1), g_v(e_2\otimes 1) \in \mathcal{O}_{v}$ and then $g_v(e_3\otimes 1) = \pm g_v(e_1\otimes 1) g_v(e_2\otimes 1) \in \mathcal{O}_{v}\cap B_{v}^0 = L\otimes_{\cO_F}\cO_{F_v}$. 
    Therefore, we have $e_3\in L(\mathcal{A}_x)$.

    Suppose $\alpha\in \SO(L(\cA_x))$, then $\alpha(e_1) = \pm e_1$ by \Cref{lem:quad_global}. If $\alpha(e_1) = e_1, \, \alpha(e_2) = \pm e_3$, then $\alpha(e_3) =  \mp e_2$, and $\alpha$ is conjugation by $1\pm e_1 \in B'$. For $v|2$, since under the extended isomorphism or anti-isomorphism $g_v: B'_v \to B_v$, the automorphism $g_v \circ \alpha \circ g_v^{-1}: B_v\to B_v$ corresponding to $\alpha$, which is conjugation by $g_v(1 +  e_1)$ or $g_v(1- e_1)$, preserves $\cO_v$, we have $g_v(1+e_1) \in F_v^\times\cO_v^\times$ or  $g_v(1-e_1) \in F_v^\times\cO_v^\times$, which implies particular $v(\nrd(1+e_1)) = v(\nrd(1-e_1)) = v(2)$ is even. However, $[F:\bQ]$ is odd, so there is some $v|2$ with odd ramification index, i.e., $v(2)$ is odd. 
    If $\alpha(e_1) = -e_1, \alpha(e_2) = \pm e_3$, then $\alpha(e_3) = \pm e_2$, then $\alpha^2 = 1$. By \Cref{lem:auto}, there is $u\in L(\cA_x)$ with $Q_F'(u) = 1$ and $\alpha(u) = u$. By \Cref{lem:quad_global}, we have $u\perp e_2$ and $u\perp e_3$, which implies $u\in \Span(e_1)$ and $\alpha(u) = -u$, leading to contradiction. 
\end{proof}

\begin{lemma}\label{lem:nondeg}
        Let $K$ be a local field with  valuation ring $\mathcal{O}$, maximal ideal $\mathfrak{p}$, residue field $k$, and a uniformizer $\varpi$. Suppose $Q: L \to \mathcal{O}$ is a rank $3$ quadratic module over $\mathcal{O}$ such that the reduction $ L \otimes_{\mathcal{O}} k \to k$ is nondegenerate over $k$. Suppose $e_1, e_2, e_3\in L$ satisfy $v_\mathfrak{p}(Q(e_1)) = 0$, $v_\mathfrak{p}(Q(e_2)) = v_\mathfrak{p}(Q(e_3)) = 1$, and $v_{\mathfrak{p}}([e_i, e_j]) > 0$ for $i\neq j$. Then 
        \begin{enumerate}[label=(\arabic*)]
            \item $e_1, e_2, e_3$ are nonzero modulo $\mathfrak{p}$;
            \item $e_1, e_i$ are linearly independent modulo $\mathfrak{p}$ for $i=2,3$;
            \item $e_1, e_2, e_3$ are linearly dependent modulo $\mathfrak{p}$;
            \item $e_2$ and $e_3$ span the same line in $L \otimes_{\mathcal{O}} k$. 
        \end{enumerate}
    \end{lemma}
    \begin{proof}
       \begin{enumerate}[label=(\arabic*)]
            \item $\mathfrak{p}^2\nmid Q(e_i), i=1,2,3$.
            \item For $i=2,3$, if $e_i = c e_1 + \varpi v$ for some $c\in \mathcal{O}, v\in L$, then $0 \equiv [e_1, e_i] \equiv c[e_1, e_1] \pmod{\mathfrak{p}}$ implies $c\equiv 0 \pmod{\mathfrak{p}}$, contradicting $e_i$ being nonzero modulo $\mathfrak{p}$. 
            \item $[e_3, e_i] \equiv 0 \pmod{\mathfrak{p}}, i = 1,2,3$.
            \item Write $e_3 = c_1 e_1+c_2e_2+\varpi v$ for some $c_1, c_2\in\mathcal{O}, v\in L$, then $ 0 \equiv [e_1, e_3] \equiv c_1[e_1, e_1] \pmod{\mathfrak{p}}$ implies $c_1\equiv 0\pmod{\mathfrak{p}}$.
        \end{enumerate}
    \end{proof}
    
\begin{proof}[Proof of \Cref{lem:pairingwithauto}]
    Suppose $\tilde{x}$ is a point in the CM cycle lifting $x$ with $\tilde{v}\in L(\cA_{\tilde{x}})$ such that $Q_F(\tilde{v}) = \lambda$, then $\cA_x$ is supersingular by \ref{supersingular}. Let $e_1\in L(\cA_x)$ such that $Q_F'(e_1) = 1$ by \Cref{lem:auto}. Let $e_2$ be the image of $\tilde{v}$ under the $\cO_F$-linear isometric map $L(\cA_{\tilde{x}})\xhookrightarrow{}L(\cA_x)$, then $Q'_F(e_2) = \lambda$, and let $e_3\in L(\cA_x)$ with $Q'_F(e_3) = \lambda$ be constructed as in \Cref{lem:pairend}.

     By \Cref{lem:nondeg}, the pairwise orthogonal elements $e_1, e_2, e_3\in L(\mathcal{A}_x)$ form a basis of  $V(\mathcal{A}_x)\otimes_{F, \rho'} W[p^{-1}]\simeq V\otimes_{F, \rho'} W[p^{-1}]$ such that $e_1, e_2$ (resp. $e_1, e_3$) are linearly independent modulo $p$, while $e_2, e_3$ span the same line modulo $p$.  
     Let $E =(W[p^{-1}])(\sqrt{-\rho'(\lambda)})$ and $\mathfrak{p} = (\sqrt{-\rho'(\lambda)})$ be the maximal ideal in $\mathcal{O}_E$. Over $E$, the two isotropic lines in $\langle e_2\rangle ^\perp$ (resp. $\langle e_3\rangle ^{\perp}$) are $l_2 = \langle \sqrt{-\rho'(\lambda)} e_1 +  e_3\rangle$ and $l_2' = \langle -\sqrt{-\rho'(\lambda)} e_1 +  e_3\rangle$ (resp. $l_3 = \langle \sqrt{-\rho'(\lambda)} e_1 +  e_2\rangle$ and $l_3' = \langle -\sqrt{-\rho'(\lambda)} e_1 +  e_2\rangle$). The automorphism $e_1$ acts as $1$ on $\langle e_1\rangle $ and $-1$ on $\langle e_1\rangle ^{\perp}$. Therefore, the two liftings corresponding to $l_2$ and $l_2'$ (resp. $l_3$ and $l_3'$) are isomorphic. Over $\mathcal{O}_E/\mathfrak{p}=k$, the isotropic lines  $l_2, l_2', l_3, l_3'$ are all spanned by $e_2$. Over $\mathcal{O}_E/\mathfrak{p}^2$ both $l_2$ and  $l_3$ are spanned by $\sqrt{-\rho'(\lambda)} e_1 +  e_2$, where $e_1, e_2$ are linearly independent. Therefore, by \cite{MR3484114}*{5.16} or the proof of \Cref{lem:lifting}, the liftings corresponding to $e_2$ and $e_3$ are isomorphic modulo $\mathfrak{p}^2$. If the liftings $\tilde{x}_2$ and $\tilde{x}_3$ corresponding to $l_2$ and $l_3$, respectively, are isomorphic, then by the same argument as in \Cref{lem:pairwoaut}, there exists a quasi-isogeny $\cA_{\tilde{x}_2} \to \cA_{\tilde{x}_3}$ inducing an automorphism $\alpha\in \SO(L(\cA_x), Q_F')$ with $\alpha(e_2) =\pm(e_3)$, which contradicts \Cref{lem:pairend}.

     After possibly extending $k$, we may assume $\sqrt{-1}\in W$. Over $W[p^{-1}]$, the two isotropic lines in $\langle e_1\rangle^{\perp}$ are $l_1 = \langle e_2 + \sqrt{-1} e_3\rangle$ and $l_1' = \langle  e_2 - \sqrt{-1} e_3\rangle$. Since $e_2, e_3$ are nonzero modulo $p$, at least one of $e_2 \pm \sqrt{-1} e_3$ is nonzero modulo $p$, then at least one of $l_1, \, l_1'$ reduces modulo $p$ to a line spanned by $e_2$, in which case it corresponds to a lift with special endomorphism $\sqrt{-1}$. 
     Over $\mathcal{O}_E/\mathfrak{p}^2$, $l_1$ (or $l_1'$) is spanned by $e_2$, while $l_2, l_3$ are spanned by $\sqrt{-\sigma(\lambda)} e_1 +  e_2$, with $e_1, e_2$ are linearly independent. Therefore, over $\mathcal{O}_K/ \mathfrak{p}^2$, the lifts corresponding to $e_2$ and $e_3$ are isomorphic to each other, but not to the lift(s) correspoinding to $l_1$ or $l_1'$. 
\end{proof}

\section{Archimedean place}\label{sec:arch}

In this section, we prove \Cref{prop:equidistribution} and its slightly more general form, \Cref{prop:equidistribution_odd}, which give an equidistribution result used to control the archimedean contribution, i.e., the sign of the polynomial $P_{\lambda}(x)P_{4\lambda}(x)$ evaluated at the coordinate corresponding to a given abelian variety. In \cref{sec:Shimura}, we recall Shimura's work on the real points of Shimura varieties, specializing to the one-dimensional case, where the set of real points is the image of a finite collection of geodesics in the upper half plane. 
Under the additional assumptions on $F$, we focus on the case where it suffices to consider a single geodesic. In \cref{sec:CMgeodesic}, we study the CM points on this geodesic, and in \cref{sec:equidis}, we use equidistribution of primes to show the density of the CM points on it. 

\subsection{Real points on the Shimura curve} \label{sec:Shimura}
Let $\varphi: \mathcal{H}\to \widetilde{\Gamma} \backslash \mathcal{H}$ be the natural map.
According to Shimura \cite{MR572971}*{4.2}, the action of complex conjugation on the Shimura curve $\widetilde{\Gamma} \backslash \mathcal{H}$ satisfies \[\overline{\varphi(z)} = \varphi(\alpha(\bar{z})),\] 
for any $\alpha\in F^\times\cO^\times$ with $\rho(\nrd(\alpha)) <0$. 
Recall that $\epsilon$ denotes a unit of $F$ that is negative at $\rho$ and positive at the other real places. The field $F(\sqrt{-\epsilon})$ embeds into $B$, since $B$ is split at all finite places and $F(\sqrt{-\epsilon})$ splits $B$ at all real ramified places of $B$. 
Moreover, there exists $\alpha\in \cO$ such that $\trd(\alpha) = 0$ and $\nrd(\alpha) = -\alpha^2 = \epsilon$, as any two maximal orders in $B$ are conjugate to each other. 
Let \[U(\epsilon) := \{ \alpha \in \mathcal{O}^\times :\trd(\alpha) = 0, \nrd(\alpha) = \epsilon\} = \{ \alpha \in \mathcal{O}^\times : \alpha^2 = -\epsilon\},\] and for each $\alpha \in U(\epsilon)$, put \[Z_{\alpha} := \{z\in \mathcal{H}: \alpha(\overline{z}) = z\}.\] Since every totally poistive unit in $\cO_F^\times$ is a square, by \cite{MR572971}*{7.4}, 
the real points of $\widetilde{\Gamma}\backslash\cH$ are given by \[\bigcup_{\alpha\in U(\epsilon)}\varphi(Z_{\alpha}).\]
The union might be taken over a set of representatives for $U(\epsilon)$ modulo inner automorphisms given by the elements of $\widetilde{\Gamma}$. 
This set is in bijection with the set of $\widetilde{\Gamma}$-conjugacy classes of 
embeddings $\mathcal{O}_F[\sqrt{-\epsilon}]\xhookrightarrow{}\mathcal{O}$, where $\mathcal{O}_F[\sqrt{-\epsilon}]$ is an order in $F(\sqrt{-\epsilon})$.

\subsubsection{One geodesic case}
Assume $[F:\bQ]>1$, $\mathcal{O}_F[\sqrt{-\epsilon}]$ is the maximal order of $F(\sqrt{-\epsilon})$ \footnote{Since $F$ has narrow class number $1$, if $2$ is inert in $F$, then $2\cO_F$ ramifies in $F(\sqrt{-\epsilon})$ and $\mathcal{O}_F[\sqrt{-\epsilon}]$ is the maximal order of $F(\sqrt{-\epsilon})$.} and $F(\sqrt{-\epsilon})$ has class number $h(F(\sqrt{-\epsilon})) = 1$, then the real points of $\widetilde{\Gamma}\backslash\cH$ is $\varphi(Z_{\alpha})$ for any $\alpha\in U(\epsilon)$ by \Cref{lem:onegeo}. 

\begin{lemma}\label{lem:onegeo}
    For any quadratic $\cO_F$-order $S\subset F[\sqrt{-\epsilon}]$, the number of  $\widetilde{\Gamma}$-conjugacy classes of optimal $S\xhookrightarrow{}\cO$ is equal to the class number $h(S)$. 
\end{lemma}
\begin{proof}
    For $\cO^1\subseteq \Gamma \subseteq N_{B^\times}(\cO)$, let $m(S, \cO; \Gamma)$ denote the number of $\Gamma$-conjugacy classes of optimal $S\xhookrightarrow{}\cO$. Then $m(S, \cO; \widetilde{\Gamma}) = m(S, \cO; \cO^\times)[\nrd(\cO^\times):\nrd(\widetilde{\Gamma}) \nrd(S^\times)]$ by \cite{MR4279905}*{30.3.14}, where $[\nrd(\cO^\times):\nrd(\widetilde{\Gamma}) \nrd(S^\times)] = 1$ since $\rho(\Nm_{F(\sqrt{-\epsilon})/F}(S))\cap\bR_{<0}\neq \emptyset$. Since $F$ has narrow class number $1$, the maximal order $\cO$ of $B$ has class number $1$ and then $m(S, \cO; \cO^\times) = h(S) m(\widehat{S}, \widehat{\cO}; \widehat{\cO}^\times)$ by \cite{MR4279905}*{30.4.17}. Finally, as $B$ is split at all finite places, we have $m(\widehat{S}, \widehat{\cO}; \widehat{\cO}^\times) = \prod_{v \text{ finite}} m(S_v, \cO_v; \cO_v^\times) =  1$.
\end{proof}

We assume $F$ has only one prime $\fp_2$ above $2$, then the quaternion algebra $B \simeq \left(\frac{-\epsilon, -1}{F}\right)$ \footnote{This is the quaternion algebra unramified at all odd primes and exactly one of the real places $\rho$ of $F$.} and can be realized as \[B \simeq \left\{\begin{pmatrix}
    a & b \\ -b' & a'
\end{pmatrix}: a,b \in F(\sqrt{-\epsilon})\right\},\]  where $'$ denotes the nontrivial involution of $F(\sqrt{-\epsilon})/F$. 
Let \[\mu = \begin{pmatrix}
    \sqrt{-\epsilon} & 0 \\ 0 & -\sqrt{-\epsilon}
\end{pmatrix}, \, \nu = \begin{pmatrix}
    0 & 1 \\ -1 & 0
\end{pmatrix}, \, \mu\nu = \begin{pmatrix}
    0 & \sqrt{-\epsilon} \\ \sqrt{-\epsilon} & 0
\end{pmatrix}.\] 

\begin{lemma}
    The quaternion algebra $\left(\frac{-\epsilon, -1}{F}\right)$ 
    has a maximal order of the form \begin{equation} \label{eq:maxorder}
        \cO = \cO_F[\mu, \frac{a+b\mu+\nu}{2}] = \cO_F + \cO_F\mu +\cO_F\frac{a+b\mu+\nu}{2} + \cO_F\frac{-\epsilon b+a\mu+\mu\nu}{2}
    \end{equation} for some nonzero $a, b\in \cO_F$. 
\end{lemma}
\begin{proof}
    Let $\pi$ be a generator of the prime $\fp_2$ of $F$ above $2. $
    By \cite{MR4279905}*{5.4.4}, since $\left(\frac{-\epsilon, -1}{F_{\pi}}\right)$ is split, there exists $a', b'\in F_{\pi}$ such that $a'^2+b'^2\epsilon + 1 =0$. 
    
    We must have $a', b'\in \cO_{F_{\pi}}$. Otherwise, $v_{\pi}(a') = v_{\pi}(b') = -m< 0$ and then $c'^2 + d'^2 \epsilon \equiv 0 \pmod{\pi^{2m}}$ where $c' = \pi^{m} a', \, d' =\pi^{m} b'\in \cO_{F_{\pi}}^\times$.
    If $m > e$, then $-\epsilon$ is a square in $F_\pi$ by Hensel's lemma, which contradicts $(\pi)$ being ramified in $F(\sqrt{-\epsilon})$. If $m\leq e$, then pick $c,d\in \cO_F$ with $c\equiv c'\pmod{\pi^{2m}}, d\equiv d'\pmod{\pi^{2m}}$, then $\cO_{F}[\frac{c+d\sqrt{-\epsilon}}{\pi^{m}} ]$ is an order in $F(\sqrt{-\epsilon})$, which contradicts $F(\sqrt{-\epsilon})$ having maximal order $\cO_{F}[\sqrt{-\epsilon}]$. 

    Choose $a, b\in \cO_F$ with $a\equiv a'\pmod{4}, \, b\equiv b'\pmod{4}$, then $\frac{a+b\mu+\nu}{2}$ is integral and $\cO = \cO_F[\mu, \frac{a+b\mu+\nu}{2}] = \cO_F + \cO_F\mu +\cO_F\frac{a+b\mu+\nu}{2} + \cO_F\mu\frac{a+b\mu+\nu}{2}$ is an order in $B$. Moreover, $\cO \supset \cO_F[\mu,\nu]$ and its discriminant $\disc(\cO) = (\frac{1}{4})^2\disc(\cO_F[\mu,\nu]) = (\frac{1}{4})^2(-4\epsilon)^2 =\epsilon^2$ is a unit in $F$. 
\end{proof}
We may assume $\cO$ is the maximal order as defined by \eqref{eq:maxorder} and $\alpha = \mu \in U(\epsilon)$. Then the geodesic \begin{equation} \label{eq:imagaxis}
    Z_{\alpha} = \{iy:y>0\}
\end{equation} is the positive imaginary axis and its stablizer in $\widetilde{\Gamma}$ is \[\left\{\begin{pmatrix}
    a & 0 \\ 0 & a'
\end{pmatrix}: a \in \mathcal{O}_F[\sqrt{-\epsilon}]^\times,\, \rho(aa')>0\right\} \bigcup \left\{\begin{pmatrix}
    0 & b \\ -b' & 0 
\end{pmatrix}: b \in \mathcal{O}_F[\sqrt{-\epsilon}]^\times, \, \rho(bb')>0\right\}.\]
By assumption, the Shimura curve $\widetilde{\Gamma}\backslash\cH$ is isomorphic to $\mathbb{P}^1(\bC)$. Hence, its real points form a circle. 
Fix an extension of $\rho$ to $F(\sqrt{-\epsilon})$ and let $u$ be the generator of $\ker(\mathcal{O}_F[\sqrt{-\epsilon}]^\times \xrightarrow{\Nm} \mathcal{O}_F^\times)$ with $\rho(u)>1$.\footnote{The rank of $\ker(\mathcal{O}_F[\sqrt{-\epsilon}]^\times \xrightarrow{\Nm} \mathcal{O}_F^\times)$ is $\rank(\mathcal{O}_F[\sqrt{-\epsilon}]^\times) - \rank(\mathcal{O}_F^\times) = (2+([F:\bQ]-1)-1)-([F:\bQ]-1) =1$. Thus, $\ker(\mathcal{O}_F[\sqrt{-\epsilon}]^\times \xrightarrow{\Nm} \mathcal{O}_F^\times) = \{\pm u^m:m\in\bZ\}$.} 
Then $\pm u^2 = \pm \frac{u}{u'}$ generates $\{\frac{a'}{a} : a\in \mathcal{O}_F[\sqrt{-\epsilon}]^\times \}$,\footnote{The $2$-rank of $H^1(\Gal(F(\sqrt{-\epsilon})/F), \mathcal{O}_F[\sqrt{-\epsilon}]^\times )$ is $1$ by \cite{MR963648}*{5.1, 9.1} and $-1 = \frac{\sqrt{-\epsilon}}{-\sqrt{-\epsilon}}$. Thus, the class of $u$ generates the size $2$ group $H^1(\Gal(F(\sqrt{-\epsilon})/F), \mathcal{O}_F[\sqrt{-\epsilon}]^\times )$.} and $[i, i\rho(u^2)]$ is a fundamental domain for $\varphi(Z_{\alpha})$.

\subsection{CM points on the geodesic}\label{sec:CMgeodesic}
For a CM point $x+iy\in \mathcal{H}$ corresponding to an embedding $L = F(\sqrt{-\lambda}) \xhookrightarrow{} B$, the element $\sqrt{-\lambda}$ maps to $ \begin{pmatrix}
    a\sqrt{-\epsilon} & b \\ -b' &-a\sqrt{-\epsilon}
\end{pmatrix}$ with $a\in F, b\in F(\sqrt{-\epsilon})$, and the matrix stabilizes $x+iy$. 
There exists $y>0$ such that $\rho\left(\begin{pmatrix}
    a\sqrt{-\epsilon} & b \\ -b' &-a\sqrt{-\epsilon}\end{pmatrix}\right)iy = iy$  if and only if $a = 0 $ and $
    \lambda = bb'\in \Nm_{F(\sqrt{-\epsilon})/F}(F(\sqrt{-\epsilon})^\times)$. In this case, \[y = \sqrt{\rho\left(\frac{b}{b'}\right)}.\]
    Suppose $\sqrt{-\lambda} \in \mathcal{O}$, then $2b\in \mathcal{O}_F[\sqrt{-\epsilon}]$, and since $\lambda = bb'$ and $F(\sqrt{-\epsilon})$ has only one prime above $2$ by assumption, it follows that $b \in \mathcal{O}_F[\sqrt{-\epsilon}]$ is a prime above $\lambda$. 

\begin{lemma}\label{lem:CMgeo}
    Assume $F(\sqrt{-\epsilon})$ has class number $h(F(\sqrt{-\epsilon})) = 1$. For a totally positive prime $\lambda$ of $F$ such that $-\lambda$ is a square modulo $4$, there exists $x\in \cO_{F(\sqrt{-\epsilon})}$ such that $\lambda = \Nm_{F(\sqrt{-\epsilon})/F}(x)$; moreover, if $x_1, x_2\in \cO_{F(\sqrt{-\epsilon})}$ both satisfy the norm equation, then one of  $\frac{x_1}{x_2}$, $\frac{x_1}{x_2'}$ is a unit in $\ker(\mathcal{O}_{F(\sqrt{-\epsilon})}^\times \xrightarrow{\Nm} \mathcal{O}_F^\times)$. 
\end{lemma}
\begin{proof}
    Since the totally negative $-\lambda$ is square modulo $4$, the quadratic reciprocity law gives \[\left(\frac{-\epsilon}{-\lambda}\right) = \left(\frac{-\lambda}{-\epsilon}\right)\left(\frac{-\epsilon}{-\lambda}\right) = (-1)^{\sum_{\sigma:F\xhookrightarrow{}\bR} \frac{\sign(\sigma(-\epsilon))-1}{2}} =(-1)^{[F:\bQ]-1} = 1,\] which implies $(\lambda)$ splits in $F(\sqrt{-\epsilon})$. Let $x\in \cO_{F(\sqrt{-\lambda})}$ be a prime above $\lambda$, then $xx' = z \lambda$ for some $z\in \cO_F^\times$ such that $\sigma(z) > 0$ for all embeddings $\sigma:F\xhookrightarrow{}\bR$ with $\sigma\neq \rho$. If $\rho(z) > 0$, then $z \in (\cO_F^\times)^2\subset \Nm_{F(\sqrt{-\epsilon})/F}(\cO_{F(\sqrt{-\epsilon})}^\times)$; if $\rho(z) < 0$, then $z \in \epsilon (\cO_F^\times)^2$, where $\epsilon = -\sqrt{-\epsilon}\sqrt{-\epsilon}$. Therefore, $\lambda$ is always a norm from $\cO_{F(\sqrt{-\epsilon})}$. 
\end{proof}

By \Cref{lem:CMgeo}, the norm equation $\lambda = bb'$ with $b\in \mathcal{O}_F[\sqrt{-\epsilon}]$ gives two points $\varphi\left(i\sqrt{\rho\left(\frac{b}{b'}\right)}\right)$ and $\varphi\left(i\rho(u)\sqrt{\rho\left(\frac{b}{b'}\right)}\right)$ on $\varphi(Z_{\alpha})$. Recall in  \cref{order} that each of the two CM cycles  defined by optimal embeddings of $\cO_{F(\sqrt{-\lambda})}$ and $\cO_{F}+\fp_2\cO_{F(\sqrt{-\lambda})}$ has a unique real point.
The two real CM points are precisely the two points obtained from the norm equation. 

\subsection{Equidistribution} \label{sec:equidis}

Given a CM cycle on the Shimura curve defined by $\widetilde{T}\subset \widetilde{G}$ and $\tau \in \Gal(\overline{\bQ}/\bQ)$, by \cite{MR654325}*{\MakeUppercase{\romannumeral 5}}, the conjugate of the CM cycle on the conjugate Shimura curve is defined by $\widetilde{T} = {}^{\tau,\mu}\widetilde{T}\subset{}^{\tau,\mu}\widetilde{G}$, where $\mu$ is the Hodge cocharacter and ${}^{\tau,\mu}\widetilde{G}$ is a twist of $\widetilde{G}$. Explicitly, as in \cite{MR2666906}*{1.4}, $^{\tau,\mu}\widetilde{G}$ arises from a quaternion algebra ${}^{\tau\!}B$ such that $\inv_{v}({}^{\tau\!}B) = \inv_{\tau\circ v}(B)$ for all infinite places $v$ of $F$ and $\inv_v({}^{\tau\!}B) = \inv_{v}(B)$ for all finite places $v$ of $F$. 
Moreover, the isomorphism $\widetilde{G}(\bA_f)\simeq {}^{\tau,\mu}\widetilde{G}(\bA_f)$ induced by $B_v\simeq {}^{\tau\!}B_v$ 
maps $\widetilde{K} = \prod\cO_v^\times$  
to ${}^{\tau,\mu}\widetilde{K} = \prod{\cO'_v}^\times$, 
where each $\cO'_v$ is a maximal order of ${}^{\tau\!}B_v$.
For a CM cycle corresponding to optimal embeddings $R\xhookrightarrow{}\cO_B$, where $R$ is an $\cO_F$-order in a CM extension of $F$, its conjugate on the conjugate Shimura curve corresponds to optimal embeddings $R\xhookrightarrow{}\cO_{{}^{\tau\!}B}$, where $\cO_{{}^{\tau\!}B}$ is a maximal order of ${}^{\tau\!}B$. 

Let $n = [F:\bQ]$ and $\rho_1, \dotsc, \rho_n:F\xhookrightarrow{}\bR$ be distinct embeddings of $F$ into $\bR$. 
For each $i=1, \dotsc, n$, let $F_i = F(\sqrt{-\epsilon_i})$, where $\epsilon_i$ is a unit of $F$ that is negative at $\rho_i$, and denote by $B_i$ the quaternion algebra over $F$ unramified at all finite places and exactly one of the real places $\rho_i$. Under the assumption that $F$ has only one prime above $2$, we have $B_i\simeq \left(\frac{-\epsilon_i, -1}{F}\right)$.

Assume $\mathcal{O}_F[\sqrt{-\epsilon_i}]$ is the maximal order of $F(\sqrt{-\epsilon_i})$ and $F(\sqrt{-\epsilon_i})$ has class number $h(F(\sqrt{-\epsilon_i})) = 1$ for all $i=1, \dots,n$.\footnote{If $F/\bQ$ is Galois, this is true if and only if one of the $\epsilon_i$ satisfies the assumption.} Then the real points of each conjugate Shimura curve is the image of one geodesic $Z_{\alpha_i}$ and we may assume $Z_{\alpha_i} = \{iy:y>0\}$ by choosing appropriate embeddings $B_i\xhookrightarrow{}M_{2}(\bR)$. Let $\varphi_i$ denote the natrual map from the upper half plane to the Shimura curve defined by $B_i$. 
The real CM points on the geodesic $\{iy:y>0\}$ are $i \sqrt{\frac{\rho_{i,1}(b_i)}{\rho_{i,1}(b_i')}}$ for some $b_i \in F_i$ satisfying $\lambda = \Nm_{F_i/F}(b_i)$, where $\rho_{i,1}$ is a real place of $F_i$ above $\rho_i$ and $\lambda\in F$ totally positive. 

Let $\cF =F_1\cdots F_n$ be the compositum. Assume $\cF$ has class number $h(\cF) = 1$. The image of $\Nm_{\cF/F} = \Nm_{F_i/F}\circ \Nm_{\cF/F_i}: \cF^\times \to F^\times$ is totally positive. Suppose $\lambda$ is a totally positve prime of $F$ such that $-\lambda$ is a square modulo $4$. Then by \Cref{lem:CMgeo},  $(\lambda)$ splits in each $F_i$, and thus splits completely in the compositum $\tilde{F}$. 

\begin{lemma}\label{lem:equidistribution}
    Let $P$ be the set of prime ideals of $\cF$. Assume $\cF$ has class number $h(\cF) = 1$. Then the set \[\left\{\left(\frac{1}{2}\log\frac{\rho_{i,1}(\Nm_{\cF/F_i}(\alpha))}{\rho_{i,1}(\Nm_{\cF/F_i}(\alpha)')}\right)_{i=1, \dotsc, n}: (\alpha) \in P\right\}\] is equidistributed in $\prod_{i=1}^n \bR/(\log \rho_{i,1}(u_i))\bZ$, where each $\rho_{i,1}$ is a real place of $F_i$ above $\rho_i$, and $u_i$ is the generator of $\ker(\cO_{F(\sqrt{-\epsilon_i})}^\times \to \cO_F^\times)$ with $\rho_{i,1}(u_i) > 1$. 
\end{lemma}
\begin{proof}
    Since $h(\cF) = 1$, we have $\mathbb{A}_{\cF}^\times = \cF^\times (\prod_{\mathfrak{p}\nmid \infty}\mathcal{O}_{\cF_\mathfrak{p}}^\times \, \prod_{w | \infty} \mathbb{C}^\times)$. The injection $\prod_{w | \infty} \mathbb{C}^\times \xhookrightarrow{}\mathbb{A}_{\cF}^\times$ induces an isomorphism $(\prod_{w | \infty} \mathbb{C}^\times)/\mathcal{O}_{\cF}^\times \xrightarrow{\sim}\mathbb{A}_{\cF}^\times /(\cF^\times (\prod_{\mathfrak{p}\nmid \infty}\mathcal{O}_{\cF_\mathfrak{p}}^\times))$. For each $i=1, \dotsc,n$, let $N_i: \prod_{w | \rho_i} \mathbb{C}^\times \to \bR_{>0}$ be the composition \begin{equation} \label{eq:Hecke}
        \prod_{w | \rho_i}  \mathbb{C}^\times  = \prod_{v | \rho_i} \prod_{w|v} \bC^\times\xrightarrow{(\Nm_{\cF/F_i}, \Nm_{\cF/F_i})} \prod_{v|\rho_i} \bR_{>0} \xrightarrow{(x,y)\mapsto \sqrt{\frac{x}{y}}} \bR_{>0} \xrightarrow{\log} \bR,
    \end{equation} where $v$ denotes infinite places of $F_i$, $w$ denotes infinite places of $\cF$, and $\Nm_{\cF/F_i}$ is the local norm. For $\alpha \in \cF^\times \xhookrightarrow{} \prod_{w|\rho_i}\bC^\times$, we have $N_i(\alpha) = \frac{1}{2} \log \frac{\rho_{i,1}(\Nm_{\cF/F_i}(\alpha))}{\rho_{i,1}(\Nm_{\cF/F_i}(\alpha)')}$. Note that \eqref{eq:Hecke} induces a surjection $\{(z_w)_{w|\rho_i}: \prod_{w|\rho_i}\|z_w\| = 1\}/\cO_{\cF}^\times \twoheadrightarrow{}\bR/(\log \rho_{i,1}(u_i))\bZ$. Then $\prod_{i=1}^n N_i: \prod_{w|\infty}\bC^\times = \prod_{i=1}^n\prod_{w|\rho_i}\bC^\times \to \prod_{i=1}^n \bR$ induces a Hecke character $\chi: \bA_\cF^\times \to \prod_{i=1}^n\bR/(\log \rho_{i,1}( u_i)) \bZ$ such that $\chi(\bA_\cF^1) = \prod_{i=1}^n\bR/(\log\rho_{i,1}( u_i)) \bZ$, where $\bA_\cF^1 := \{(a_v)\in\bA_\cF:\prod_v\|a_v\|=1\}$, and the result follows from Hecke's equidistribution \cite{MR1282723}*{\MakeUppercase{\romannumeral 15}, \S5}.
\end{proof}

Recall that $iy \mapsto \log y$ gives an isomorphism $\varphi_i(Z_{\alpha_i}) \to \mathbb{R}/(2\log \rho_{i,1}(u_i))\bZ$, and two real CM points $\varphi_i(z_{i,1}), \varphi_i(z_{i,2})\in\varphi_i(Z_{\alpha_i})$ defined by the two orders in $L = F(\sqrt{-\lambda})$ are related by $\varphi(z_{i,2}) = \varphi(\rho_{i,1}(u_i) z_{i,1})$.

\begin{proposition} \label{prop:equidistribution}
    Let $S = \{\fq_1, \dotsc, \fq_r\}$ be a finite set of odd prime ideals of $F$. 
    For each $i=1, \dotsc, n$, let $t_{i,0}$ and $ t_{i, \infty} $ be distinct points on the circle $ \varphi(Z_{\alpha_i}) \simeq \mathbb{R}/(2\log \rho_{i,1}(u_i))\bZ $. There exists a prime ideal $\fq$ of $F$ generated by some totally positive $\lambda\in\cO_F$ satisfying the following: \begin{enumerate}
        \item $-\lambda$ is a nonzero square modulo $8\fq_1\dotsc\fq_r$; 
        \item $\varphi_i(z_{i,1}), \varphi_{i}(z_{i,2})$ lie on different open segments defined by $t_{i,0}, t_{i, \infty}$ for each  $i=1, \dotsc, n$;
        \item $\fq$ lies above a rational prime that is totally split in $\cF$. 
    \end{enumerate}
\end{proposition}    
\begin{proof}
    Consider the modulus $\fm = \fm_0\fm_{\infty} = 8\fq_1\dotsc\fq_r\prod_{\fp\text{ real}}\fp$ and the canonical homomorphism $\pi_{\fm}:\bA_{F}^\times \to C_{\fm}$ realizing the ray class group $C_{\fm}$ as a quotient of $\bA_{F}^\times$. 
    For $(a_{\fp})_{\fp}\in \bA_F^\times$ with $v_{\fp}(a_{\fp}-1) \geq m(\fp)$ for all $\fp \mid \fm_0$ and $a_{\fp}>0$ for all real $\fp$, $\pi_{\fm}((a_{\fp})_{\fp})$ is the class of $\prod_{\fp \text{ finite}} \fp^{v_{\fp}(a_{\fp})}$.
    Note that $\pi_{\fm}(F^\times) = 1$ and $\pi_{\fm}(\bA_F^1) = C_{\fm}$. 
    Let $V_{\fm}$ be the image of $\{(x):x\in F^\times, v_{\fp}(x+1) \geq m(\fp)$ \text{ for all } $\fp \mid \fm_0$ and $x_{\fp} > 0$ \text{ for all real } $\fp\}$ in $C_{\fm}$, which is nonempty by weak approximation. By equidistribution there exists some prime ideal of $F$ that can be generated by some totally negative $-\lambda$ satisfying $-\lambda\equiv 1$ modulo $\fm_0$. Let  $\chi_0 =  \pi_{\fm}\circ \Nm_{\cF/F}$, then $\chi_0(\bA_\cF^1)\cap V_{\fm}\neq \emptyset$ since $(\lambda)$ splits in $\cF$. 
    
    Define a Hecke character $\chi = (\chi_0, \chi_{\infty}):\bA_{\cF}^\times \to C_{\fm} \times \left(\prod_{i=1}^n\bR/(\log \rho_{i,1}(u_i)) \bZ\right)$, where $\chi_{\infty}$ is defined in the proof of \Cref{lem:equidistribution} by \eqref{eq:Hecke}. Note that $\chi_0(\prod_{w|\infty} \bC^\times) = 1$, so we have $\chi_0(\bA_\cF^\times) = \chi_0(\ker(\chi_\infty))$, and  $\chi(\bA_\cF^1) = \chi_0(\bA_\cF^1)\times \chi_{\infty}(\bA_\cF^1)$. 
    For each $i=1, \dotsc, r$, let $V_i$ be the image in $\bR/(\log \rho_{i,1}(u_i)) \bZ$ of the open segment on the circle $ \varphi(Z_{\alpha_i}) \simeq \mathbb{R}/(2\log \rho_{i,1}(u_i))\bZ $ defined by $t_{i,0}, t_{i, \infty}$ of smaller measure. Note that a point in $V_i\subset \bR/(\log \rho_{i,1}(u_i)) \bZ$ has two preimages lie on different open segments of $\bR/(2\log \rho_{i,1}(u_i)) \bZ$. 
    Since $V:= V_{\fm}\times (\prod_{i=1}^n V_i)$ is a nonempty open subspace of the compact space $C_{\fm}\times \left(\prod_{i=1}^n\bR/(\log \rho_{i,1}(u_i)) \bZ\right) $, and $\chi(\bA_{\cF}^1) \cap V \neq \emptyset$, it follows from Hecke's equidistribution \cite{MR1282723}*{\MakeUppercase{\romannumeral 15}, \S5} that there exists a prime ideal $(\alpha)$ of $\cF$ such that $\chi((\alpha)) \in V$. Note that $\Nm_{\cF/F}(\alpha)$ is totally positive since \eqref{eq:Hecke} always gives positive terms in the middle. Let $\fq = \Nm_{\cF/F}((\alpha))$, where we may assume $(\alpha)$ lies above a totally split rational prime, as such primes have density one in $\cF$. 
\end{proof}

\begin{proposition} \label{prop:equidistribution_odd}
    Let $S = \{\fq_1, \dotsc, \fq_r\}$ be a finite set of odd prime ideals of $F$. 
    For each $i=1, \dotsc, n$, let $d_i$ be an odd positive integer, and $t_{i}, s_{i,1}, \dotsc, s_{i,d_i} \in \varphi(Z_{\alpha_i}) \simeq \mathbb{R}/(2\log \rho_{i,1}(u_i))\bZ $ such that their images in $ \mathbb{R}/(\log \rho_{i,1}(u_i))\bZ $ are distinct. There exists a prime ideal $\fq$ of $F$ generated by some totally positive $\lambda\in\cO_F$ satisfying the following: \begin{enumerate}
        \item $-\lambda$ is a nonzero square modulo $8\fq_1\dotsc\fq_r$; 
        \item for each  $i=1, \dotsc, n$, the two real CM points $\varphi_i(z_{i,1}), \varphi_{i}(z_{i,2})$ lie on different open segments defined by $t_{i}, s_{i, j}$ for an odd number of $j\in \{1, \dotsc, d_i\}$;
         \item $\fq$ lies above a rational prime that is totally split in $\cF$.
    \end{enumerate}
\end{proposition}    
\begin{proof}
    The proof follows that of \Cref{prop:equidistribution}, with the only modification being the definition of $V_i$. 
    For each $i=1, \dotsc, r$ and $j=1, \dotsc, d_i$, the two distinct points $t_i, s_{i,j}$ define two open segments in $\bR/(2\log \rho_{i,1}(u_i)) \bZ$, and let $V_{i,j}\subset \bR/(\log \rho_{i,1}(u_i)) \bZ$ be the image of the one of smaller measure. 
    Note that a point  in $\bR/(\log \rho_{i,1}(u_i)) \bZ$ has two preimages on different open segments of $\bR/(2\log \rho_{i,1}(u_i)) \bZ$ if and only if it is in $V_{i,j}$. 
    There exists $J_i \subset \{1, \dotsc, d_i\}$ such that $\#J_i$ is odd and $V_i :=(\bigcap_{j\in J_i} V_{i,j}) \backslash(\bigcup_{j\notin J_i} V_{i,j})\neq \emptyset$ is a non empty open subspace of the compact space $\bR/(\log \rho_{i,1}(u_i)) \bZ$.
\end{proof}

\section{Proof of main theorem}

In this section, we prove \Cref{thm:main}. We begin by defining the polynomials associated to the CM cycles constructed in \S3.  In \cref{sec:main2} we use the input from \S4 to study the polynomials modulo $\lambda$ after clearing denominators, and in \cref{sec:main3} we use the input from \S5 to analyze the sign of the polynomials evaluated at the coordinate of the given abelian variety and at the elliptic points with even order automorphisms. In \cref{sec:main4} we combine all of the ingredients and apply quadratic reciprocity to complete the proof.

Suppose the Shimura curve $\widetilde{\Gamma} \backslash \cH \simeq \Gamma_0\backslash X^+$ has genus $0$. Its canonical model is defined over the reflex field $\rho(F) \subset \bC$. Assume the canonical model is isomorphic to $\mathbb{P}^1_F$. Since the quaternion algebra $B$ is unramified at all finite primes, by \cite{MR772569}*{pp. 508-509}, the coarse moduli space of the integral canonical model is smooth, and thus isomorphic to $\bP^1_{\cO_F}$.   
Choose a coordinate $j$ compatible with the model over $\cO_F$.

Fix $\overline{F}\xhookrightarrow{} \bC$, and let $\tau_1, \dotsc, \tau_{[F:\bQ]} \in \Gal(\overline{F}/\bQ)$ correspond to the embeddings $\rho_1, \dotsc, \rho_{[F:\bQ]}: F\xhookrightarrow{}\bR$ and $c$ denote complex conjugation. 

Assume $2$ is inert in $F$ and $F$ satisfies all the additional assumptions in \S5. Since $h(F(\sqrt{-1}))$ is odd by \cite{MR963648}*{13.7}, there is an elliptic point of order $2$ with coordinate $j_{\infty}$ such that $[F(j_{\infty}):F]$ is odd. 
Let $j_0 \in F$ be the coordinate of the given point on the Shimura curve such that $F' := F(j_{\infty}, j_0)$ is an odd-degree extension of $F$. Let $\sigma_1, \dotsc, \sigma_{[F':F]} \in \Gal(\overline{F}/F)$ such that $\Hom_F(F', \overline{F}) = \{\sigma_1|_{F'}, \dotsc, \sigma_{[F':F]}|_{F'}\}$. To get the main idea of the proof, the reader may focus on the case when $F' = F$. 

Let $S'$ be a finite set of rational primes such that $S' \supset \{p: p \text{ is ramified in } F'\} \cup \{p: v_{\fp}(j_0) \neq 0 \text{ for some prime } \fp \text{ of } F' \text{ above } p\} \cup \{p: v_{\fp}(j_{\infty}) \neq 0 \text{ for some prime } \fp \text { of } F' \text{ above } p\}$. Let $S$ be the set of prime ideals of $F$ above the primes in $S'$. Then there exist $d_0, d_{\infty}\in \cO_F$ such that $d_0j_0 \in \cO_{F'}, d_{\infty}j_{\infty}\in \cO_{F'}$ and all primes dividing $d_0d_{\infty}$ lie in $S$.

\subsection{} \label{sec:main1}

For a large enough totally positive prime $\lambda\in F$ with $-\lambda$ square modulo $4$, let $\alpha_{1}, \dotsc, \alpha_{h_1}$ (resp. $\beta_{1}, \dotsc, \beta_{h_2}$) be the $j$-coordinates of the CM points in the CM cycle corresponding to optimal embeddings $\mathcal{O}_{F(\sqrt{-\lambda})} \xhookrightarrow{} \mathcal{O}$  (resp. $\mathcal{O}_F[\sqrt{-\lambda}] \xhookrightarrow{} \mathcal{O}$), where $h_1 = h(F(\sqrt{-\lambda}))$ (resp. $h_2 = h(\cO_F[\sqrt{-\lambda}])$ ) is odd, and define the polynomials \begin{align*}
    P_{\lambda}(x) := \prod_{l=1}^{h_1}(x-\alpha_l), \\
    P_{4\lambda}(x) := \prod_{l=1}^{h_2}(x-\beta_l).
\end{align*} By construction, $\Gal(\overline{F}/F)$ permutes $\alpha_{1}, \dotsc, \alpha_{h_1}$ (resp. $\beta_{1}, \dotsc, \beta_{h_2}$), 
so that $P_{\lambda}(x), \, P_{4\lambda}(x)\in F[x]$. Moreover, we have $\alpha_1, \dotsc, \alpha_{h_1} \in H_{F(\sqrt{-\lambda})}$ and  $\beta_1, \dotsc, \beta_{h_2} \in H_{\mathcal{O}_F[\sqrt{-\lambda}]}$, where  $H_{F(\sqrt{-\lambda})}$ is the  Hilbert class field of $F(\sqrt{-\lambda})$ and  $H_{\mathcal{O}_F[\sqrt{-\lambda}]}$ is  the ring class field  corresponding to the order $\mathcal{O}_F[\sqrt{-\lambda}]$. Note that $H_{F(\sqrt{-\lambda})} \subset H_{\mathcal{O}_F[\sqrt{-\lambda}]}$ and the extension $H_{\mathcal{O}_F[\sqrt{-\lambda}]} / H_{F(\sqrt{-\lambda})}$ is ramified at a subset of the primes dividing the conductor $2$. Since $(\sqrt{-\lambda})$ is a principal prime ideal of $\mathcal{O}_F[\sqrt{-\lambda}]$, it splits completely in $H_{\mathcal{O}_F[\sqrt{-\lambda}]}$. Let $H := H_{\mathcal{O}_F[\sqrt{-\lambda}]}$ and  $\mathfrak{l}$ be a prime of $H$ above $\lambda$. 

Let $b_{\lambda} \in \cO_F$ (resp. $b_{4\lambda}\in \cO_F$) be the denominator of $P_{\lambda}(x)$ (resp. $P_{4\lambda}(x)$) such that $b_{\lambda}P_{\lambda}(x)$ (resp. $b_{4\lambda}P_{4\lambda}(x)$) is primitive, then \[b_{\lambda} P_{\lambda}(x) = \left(b_{\lambda}\prod_{\substack{k=1 \\ v_{\mathfrak{l}}(\alpha_k)<0}}^{h_1}(x-\alpha_k)\right) \prod_{\substack{k=1 \\ v_{\mathfrak{l}}(\alpha_k)\geq0}}^{h_1}(x-\alpha_k) \in \cO_{H_{\mathfrak{l}}}[x],\] 
\[b_{4\lambda} P_{4\lambda}(x) = \left(b_{4\lambda}\prod_{\substack{k=1 \\ v_{\mathfrak{l}}(\beta_k)<0}}^{h_2}(x-\beta_k)\right) \prod_{\substack{k=1 \\ v_{\mathfrak{l}}(\beta_k)\geq0}}^{h_2}(x-\beta_k) \in \cO_{H_{\mathfrak{l}}}[x],\]
where \begin{align*}
    b_{\lambda}\prod_{\substack{k=1 \\ v_{\mathfrak{l}}(\alpha_k)<0}}^{h_1}(x-\alpha_k) \equiv \tilde{b}_{\lambda}\not\equiv 0 \pmod{\mathfrak{l}}, \qquad &\tilde{b}_{\lambda} := b_{\lambda} \prod_{\substack{k=1 \\ v_{\mathfrak{l}}(\alpha_k)<0}}^{h_1}(-\alpha_k), \\
    b_{4\lambda}\prod_{\substack{k=1 \\ v_{\mathfrak{l}}(\beta_k)<0}}^{h_2}(x-\beta_k) \equiv \tilde{b}_{4\lambda}\not\equiv 0 \pmod{\mathfrak{l}}, \qquad &\tilde{b}_{4\lambda} := b_{4\lambda} \prod_{\substack{k=1 \\ v_{\mathfrak{l}}(\beta_k)<0}}^{h_2}(-\beta_k).
\end{align*}

\subsection{} \label{sec:main2} By \Cref{lem:pairwoaut} and the first part of \Cref{lem:pairingwithauto}, roots of $P_{\lambda}(x)P_{4\lambda}(x)$ are paired modulo $\mathfrak{l}$. Then \Cref{lem:square} shows that $\Nm_{F'/F}(d_{\infty}^{h_1+h_2} b_{\lambda} P_{\lambda}(j_{\infty}) b_{4\lambda} P_{4\lambda}(j_{\infty})d_{0}^{h_1+h_2} b_{\lambda} P_{\lambda}(j_{0}) b_{4\lambda} P_{4\lambda}(j_0))\in \cO_F$ is a square in  $\cO_F/\lambda \cO_F = \cO_{H_{\mathfrak{l}}}/\mathfrak{l}\cO_{H_{\mathfrak{l}}}$. 
\begin{lemma}\label{lem:square}
    Modulo $\mathfrak{l}$, we have \begin{align}
        \Nm_{F'/F}(d_{0}^{h_1+h_2} b_{\lambda} P_{\lambda}(j_{0}) b_{4\lambda} P_{4\lambda}(j_0)) \equiv (\tilde{b}_\lambda \tilde{b}_{4\lambda})^{[F':F]} \cdot \text{square}, \label{eq:sqj0}\\
        \Nm_{F'/F}(d_{\infty}^{h_1+h_2} b_{\lambda} P_{\lambda}(j_{\infty}) b_{4\lambda} P_{4\lambda}(j_{\infty})) \equiv (\tilde{b}_\lambda \tilde{b}_{4\lambda})^{[F':F]} \cdot \text{square}.\label{eq:sqjinfty}
    \end{align}
\end{lemma}
\begin{proof}
We prove \eqref{eq:sqjinfty}, and \eqref{eq:sqj0} follows by the same argument. 
Let $\sigma_1, \dotsc, \sigma_{[F':F]} \in \Gal(\overline{F}/F)$ such that $\Hom_F(F', \overline{F}) = \{\sigma_1|_{F'}, \dotsc, \sigma_{[F':F]}|_{F'}\}$. Then \[\Nm_{F'/F}(P_{\lambda}(j_{\infty})) =  \prod_{i=1}^{[F':F]}\sigma_i(P_{\lambda}(j_{\infty}))  =  \prod_{i=1}^{[F':F]} P_{\lambda}(\sigma_i j_{\infty}) = \prod_{k=1}^{h_1}\prod_{i=1}^{[F':F]}(\sigma_ij_{\infty} - \alpha_k)\] where $\prod_{i=1}^{[F':F]}(\sigma_ij_{\infty} - \alpha_k) \in H$ for each $k=1, \dotsc, h_1$. Similarly, \[\Nm_{F'/F}(P_{4\lambda}(j_{\infty})) = \prod_{k=1}^{h_2}\prod_{i=1}^{[F':F]}(\sigma_i j_{\infty} - \beta_k)\] where $\prod_{i=1}^{[F':F]}(\sigma_ij_{\infty} - \beta_k) \in H$ for each $k=1, \dotsc, h_2$.
Then we have \[\Nm_{F'/F}(d_{\infty}^{h_1} b_{\lambda}P_{\lambda}(j_{\infty})) = \left(b_{\lambda}^{[F':F]}\prod_{\substack{k=1 \\ v_{\mathfrak{l}}(\alpha_k)<0}}^{h_1} \prod_{i} (d_{\infty}(\sigma_i (j_{\infty})-\alpha_k))\right) \left(\prod_{\substack{k=1 \\ v_{\mathfrak{l}}(\alpha_k)\geq0}}^{h_1}\prod_{i} (d_{\infty}(\sigma_i (j_{\infty})-\alpha_k))\right) \in \cO_{H_{\mathfrak{l}}},\]
\[\Nm_{F'/F}(d_{\infty}^{h_2} b_{4\lambda}P_{4\lambda}(j_{\infty})) = \left(b_{4\lambda}^{[F':F]}\prod_{\substack{k=1 \\ v_{\mathfrak{l}}(\beta_k)<0}}^{h_2} \prod_{i} (d_{\infty}(\sigma_i (j_{\infty})-\beta_k))\right) \left(\prod_{\substack{k=1 \\ v_{\mathfrak{l}}(\beta_k)\geq0}}^{h_2}\prod_{i} (d_{\infty}(\sigma_i (j_{\infty})-\beta_k))\right) \in \cO_{H_{\mathfrak{l}}}.
\]
From the assumption that $\lambda\notin S$, we have $v_{\mathfrak{l}}(d_\infty) = 0$ and $v_{\mathfrak{l}'}(\sigma_i j_{\infty}) = 0 $ for any prime $\mathfrak{l}'$ above $\mathfrak{l}$, then  \[b_{\lambda}^{[F':F]}\prod_{\substack{k=1 \\ v_{\mathfrak{l}}(\alpha_k)<0}}^{h_1} \prod_{i} (d_{\infty}(\sigma_i (j_{\infty})-\alpha_k)) \equiv \tilde{b}_{\lambda}^{[F':F]} \not\equiv0\pmod{\mathfrak{l}}, \]
\[b_{4\lambda}^{[F':F]}\prod_{\substack{k=1 \\ v_{\mathfrak{l}}(\beta_k)<0}}^{h_1} \prod_{i} (d_{\infty}(\sigma_i (j_{\infty})-\beta_k)) \equiv \tilde{b}_{4\lambda}^{[F':F]} \not\equiv0\pmod{\mathfrak{l}}.\]
Let $\mathfrak{l}'$ be a prime above $\mathfrak{l}$ in the Galois closure of $F'H$. Since $\lambda\notin S$, it is unramified in $F'$, and $\mathfrak{l}'$ over $\lambda$ has ramification index $2$. 
For each point on the Shimura curve with coordinate $j \in F'H$, if it reduces to a point without even order automorphisms, then the sets $\{\alpha_k: v_{\mathfrak{l}'}(\alpha_k-j) >0 \}$ and $\{\beta_k: v_{\mathfrak{l}'}(\beta_k-j) > 0 \}$ both have even cardinality from the pairing as described in \Cref{lem:pairwoaut};  if it reduces to a point with even order automorphisms, then the set $\{\alpha_k: v_{\mathfrak{l}'}(\alpha_k-j) >0 \}\cup \{\beta_k: v_{\mathfrak{l}'}(\beta_k-j) > 0 \}$ has even cardinality from the pairing as described in \Cref{lem:pairingwithauto}.\footnote{Here we need the assumption $2$ inert in $F$.}
Therefore, for any $\bar{j} \in \cO_{H_{\mathfrak{l}}} / \mathfrak{l}\cO_{H_{\mathfrak{l}}}$, \begin{equation}\label{eq:square}\left(\prod^{h_1}_{\substack{k=1\\ \overline{\alpha_k} = \bar{j}}} \prod_i(d_{\infty} (\sigma_i(j_{\infty}) - \alpha_k))\right)\left(\prod^{h_2}_{\substack{k=1\\ \overline{\beta_k} = \bar{j}}} \prod_i(d_{\infty} (\sigma_i(j_{\infty}) - \beta_k))\right)\end{equation} is a square in $\cO_{H_{\mathfrak{l}}}/\mathfrak{l}\cO_{H_{\mathfrak{l}}}$, where $\overline{\alpha_k}$ (resp. $\overline{\beta_k}$) denotes the image of $\alpha_k$ (resp. $\beta_k$) in $\cO_{H_{\mathfrak{l}}} / \mathfrak{l}\cO_{H_{\mathfrak{l}}}$. 
\end{proof}

We consider the product of $b_{\lambda}P_{\lambda}(j_0)b_{4\lambda}P_{4\lambda}(j_0)$ and $b_{\lambda}P_{\lambda}(j_\infty)b_{4\lambda}P_{4\lambda}(j_\infty)$ to cancel the effect of the leading coefficient $\tilde{b}_\lambda \tilde{b}_{4\lambda}$ of $b_{\lambda}P_{\lambda}(x) b_{4\lambda} P_{4\lambda}(x)$ modulo $\lambda$. In order to analyze $\Nm_{F'/F}(d_0^{h_1+h_2}b_{\lambda}P_{\lambda}(j_0)b_{4\lambda}P_{4\lambda}(j_0))$ modulo $\lambda$, we need to divide appropriate power of $\lambda$ from $\Nm_{F'/F}(d_{\infty}^{h_1+h_2} b_{\lambda} P_{\lambda}(j_{\infty}) b_{4\lambda} P_{4\lambda}(j_{\infty}))$ to get a nonzero square, so we use \Cref{lem:pairingwithauto} to prove \Cref{lem:nonzerosq}.

Each $\sigma_i(j_{\infty})$ is a point with an even order automorphism, and let \[A_i := \{\alpha_k: v_{\mathfrak{l}'}(\sigma_i(j_{\infty}) - \alpha_k) > 0\} \cup \{\beta_k: v_{\mathfrak{l}'}(\sigma_i(j_{\infty}) - \beta_k) > 0\}.\] For $\lambda$ large enough, we may assume $A_{i} \cap A_{i'} = \emptyset$ whenever $\sigma_i(j_{\infty}) \neq \sigma_{i'}(j_{\infty})$, and $\#\Aut(\sigma_i(j_{\infty})_{\overline{\bF}_{\mathfrak{l}'}}) =\#\Aut(\sigma_i(j_{\infty})) = 2$.

\begin{lemma} \label{lem:nonzerosq}
    Each $\#A_i$ is even, and we have \begin{equation}\lambda^{-\sum_{i=1}^{[F':F]}\#A_i}\Nm_{F'/F}(d_{\infty}^{h_1+h_2} b_{\lambda} P_{\lambda}(j_{\infty}) b_{4\lambda} P_{4\lambda}(j_{\infty})) \equiv (\tilde{b}_\lambda \tilde{b}_{4\lambda})^{[F':F]} \cdot \text{nonzero square}\end{equation} modulo $\mathfrak{l}$. 
\end{lemma}
\begin{proof}
For $\bar{j} \in \cO_{H_{\mathfrak{l}}} / \mathfrak{l}\cO_{H_{\mathfrak{l}}}$ such that $\bar{j}\neq \overline{\sigma_i(j_\infty)}$ for any $i$, \eqref{eq:square} is a nonzero square in $\cO_{H_{\mathfrak{l}}} / \mathfrak{l} \cO_{H_{\mathfrak{l}}}$ as in the proof of \Cref{lem:square}. 
From the pairing in the neighborhood of $\sigma_i(j_{\infty})_{\overline{\bF}_{\mathfrak{l}'}}$ as described in \Cref{lem:pairingwithauto}, elements in $A_i$ are paired 
and for each pair $\{\gamma_1, \gamma_2\}$, by \Cref{lem:intersection_auto}, we have \[v_{\mathfrak{l}}(\gamma_1-\gamma_2) = v_{\mathfrak{l}'}(\gamma_1-\gamma_2) \geq \#\Aut(\sigma_i(j_{\infty})_{\overline{\bF}_{\mathfrak{l}'}}) \cdot 2 >\#\Aut(\sigma_i(j_{\infty})_{\overline{\bF}_{\mathfrak{l}'}}) \cdot 1 = v_{\mathfrak{l}'}(\sigma_{i}(j_{\infty})-\gamma_1)=v_{\mathfrak{l}'}(\sigma_{i}(j_{\infty})-\gamma_2),\] where $\#\Aut(\sigma_i(j_{\infty})_{\overline{\bF}_{\mathfrak{l}'}}) =2 = v_{\mathfrak{l}'}(\lambda)$ and $\mathfrak{l}'$ is a prime above $\mathfrak{l}$ in the Galois closure of $F'H$. Then \[v_{\mathfrak{l}} \left(\prod_{m=1}^{[F':F]}(d_{\infty}(\sigma_m (j_{\infty})-\gamma_1))\right) = v_{\mathfrak{l}} \left(\prod_{m=1}^{[F':F]}(d_{\infty}(\sigma_m (j_{\infty})-\gamma_2))\right)= 2[F':F(j_{\infty})],\]
\[\lambda^{-[F':F(j_{\infty})]} \left(\prod_{m=1}^{[F':F]}(d_{\infty}(\sigma_m (j_{\infty})-\gamma_1))-\prod_{m=1}^{[F':F]}(d_{\infty}(\sigma_m(j_{\infty})-\gamma_2))\right) \in \lambda \cO_{{H}_{\mathfrak{l}}}.\] 
Therefore,  \[\lambda^{-\sum_{i=1}^{[F':F]}\#A_i} \left(\prod_{\substack{k=1 \\ v_{\mathfrak{l}}(\alpha_k)\geq0}}^{h_1}\prod_{i=1}^{[F':F]} (d_{\infty}(\sigma_i (j_{\infty})-\alpha_k))\right)\left(  \prod_{\substack{k=1 \\ v_{\mathfrak{l}}(\beta_k)\geq0}}^{h_2}\prod_{i=1}^{[F':F]} (d_{\infty}(\sigma_i (j_{\infty})-\beta_k)) \right)\] is a square in $(\cO_{H_{\mathfrak{l}}}/\mathfrak{l}\cO_{H_{\mathfrak{l}}})^\times \simeq (\cO_{F}/\lambda \cO_F)^\times$, where each $\#A_i$ is even.  
\end{proof}

\subsection{} \label{sec:main3}
As explained in \Cref{sec:equidis}, for each $\tau_i \in \Gal(\overline{F}/\bQ)$, $\tau_i(\alpha_1), \dotsc, \tau_i(\alpha_{h_1})$ (resp. $\tau_i(\beta_{1}), \dotsc, \tau_i(\beta_{h_2})$) correspond to the points in the conjugate CM cycle on the conjugate Shimura variety $\bP^1_{\tau_i(F)}$. The conjugate CM cycle is defined over $\tau_i(F) \subset \bR$, then $\{\tau_i(\alpha_1), \dotsc, \tau_i(\alpha_{h_1})\}$ (resp. $\{\tau_i(\beta_1), \dotsc, \tau_i(\beta_{h_2})\}$) is invariant under complex conjugation,
and let $\alpha_i^* \in \{\tau_i(\alpha_1), \dotsc, \tau_i(\alpha_{h_1})\}$ (resp. $\beta_i^* \in \{\tau_i(\beta_1), \dotsc, \tau_i(\beta_{h_2})\}$) be the unique real point (\Cref{lem:uniquereal}).

\begin{lemma}
    For each $\tau_i \in \Gal(\overline{F}/\bQ)$, the sign of $\tau_i(\Nm_{F'/F}(P_{\lambda}(j_0)P_{4\lambda}(j_0)P_{\lambda}(j_\infty)P_{4\lambda}(j_\infty)))$ is determined by the sign of $$\prod_{\substack{\sigma: F(j_0)\xhookrightarrow{}\bR \\ \sigma|_{F} = \rho_i}}(\sigma(j_0) - \alpha_{i}^*)(\sigma(j_0)- \beta_i^*)(\sigma_{\infty,i}(j_\infty) - \alpha_{i}^*)(\sigma_{\infty,i}(j_\infty)- \beta_i^*),$$
    where $\sigma_{\infty,i}$ is the unique embedding $ F(j_{\infty}) \xhookrightarrow{} \bR$ satisfying $\sigma_{\infty,i}|_F = \rho_i$. 
\end{lemma}
\begin{proof}
When $\sigma\in \Gal(\overline{F}/F)$ with $\tau_i\sigma(F')\subset \bR$, we have 
\begin{align*}
   \prod_{\substack{l=1\\\tau_i(\alpha_l)\neq \alpha_i^*}}^{h_1}\left(\tau_i\sigma (j_\infty) - \tau_i(\alpha_l)\right)>0, \quad & \quad \prod_{\substack{l=1\\\tau_i(\beta_l)\neq \beta_i^*}}^{h_2}\left(\tau_i\sigma (j_\infty) - \tau_i(\beta_l)\right)>0,\\
   \prod_{\substack{l=1\\\tau_i(\alpha_l)\neq \alpha_i^*}}^{h_1}\left(\tau_i\sigma (j_0) - \tau_i(\alpha_l)\right)>0, \quad & \quad \prod_{\substack{l=1\\\tau_i(\beta_l)\neq \beta_i^*}}^{h_2}\left(\tau_i\sigma (j_0) - \tau_i(\beta_l)\right)>0,
\end{align*}
then 
\begin{align*}
    \sgn(\tau_i\sigma(P_{\lambda}(j_\infty)P_{4\lambda}(j_\infty))) = &\sgn \left((\tau_i\sigma(j_\infty) - \alpha_{i}^*)(\tau_i\sigma(j_\infty)- \beta_i^*)\right), \\
     \sgn(\tau_i\sigma(P_{\lambda}(j_0)P_{4\lambda}(j_0))) = &\sgn \left((\tau_i\sigma(j_0) - \alpha_{i}^*)(\tau_i\sigma(j_0)- \beta_i^*)\right).
\end{align*}
Since $F$ is totally real, we have $c\circ \tau_i\circ \sigma_k|_F = \tau_i|_F$, and $c\circ \tau_i\circ \sigma_k|_{F'} = \tau_i\circ \sigma_{k'}|_{F'}$ for a unique $k'\in \sigma_1, \dotsc, \sigma_{[F':F]} \in \Gal(\overline{F}/F)$, where $k' = k$ if and only if $\tau_i(\sigma_k(F'))\subset\bR$, then  
\begin{align*}
    &\sgn(\tau_i(\Nm_{F'/F}(P_{\lambda}(j_0)P_{4\lambda}(j_0)P_{\lambda}(j_\infty)P_{4\lambda}(j_\infty)))) \\=&\sgn\left(\tau_i \prod_{k=1}^{[F':F]} (\sigma_k (P_{\lambda}(j_0)P_{4\lambda}(j_0)P_{\lambda}(j_\infty)P_{4\lambda}(j_\infty)))\right) \\
    =&\sgn\left(\tau_i \prod^{[F':F]}_{\substack{k=1\\\tau_i\sigma_k(F')\subset \bR}} (\sigma_k (P_{\lambda}(j_0)P_{4\lambda}(j_0)P_{\lambda}(j_\infty)P_{4\lambda}(j_\infty))\right)\\
   = &\sgn\left(\prod^{[F':F]}_{\substack{k=1\\\tau_i\sigma_k(F')\subset \bR}}\left((\tau_i\sigma_k(j_0) - \alpha_{i}^*)(\tau_i\sigma_k(j_0)- \beta_i^*)(\tau_i\sigma_k(j_\infty) - \alpha_{i}^*)(\tau_i\sigma_k(j_\infty)- \beta_i^*)  \right)\right)\\
   =&\sgn\left(\prod_{\substack{\sigma: F(j_0)\xhookrightarrow{}\bR \\ \sigma|_{F} = \rho_i}}((\sigma(j_0) - \alpha_{i}^*)(\sigma(j_0)- \beta_i^*)) \prod_{\substack{\sigma: F(j_\infty)\xhookrightarrow{}\bR \\ \sigma|_{F} = \rho_i}}((\sigma(j_\infty) - \alpha_{i}^*)(\sigma(j_\infty)- \beta_i^*))\right),
\end{align*} where the last equality follows from $[F':F]$ being odd. 
Since the CM cycle corresponding to optimal embedding $\cO_{F(\sqrt{-1})}$ has a unique real point, there is a unique $\sigma_{\infty,i}: F(j_{\infty}) \xhookrightarrow{} \bR$ with $\sigma_{\infty,i}|_F = \rho_i$.
\end{proof}

\subsection{} \label{sec:main4}
By \Cref{prop:equidistribution_odd}, there exists a totally positive prime $\lambda \notin S$ of $F$ such that \begin{enumerate}
    \item $-\lambda$ is a square modulo $8$ and all primes in $S$;
    \item for each $i=1, \dotsc, [F:\bQ]$, $(\sigma(j_0) - \alpha_{i}^*)(\sigma(j_0)- \beta_i^*) (\sigma_{\infty,i}(j_\infty) - \alpha_{i}^*)(\sigma_{\infty,i}(j_\infty)- \beta_i^*)<0 $ for an odd number of $\sigma:F(j_0)\xhookrightarrow{}\bR$ with $\sigma|_F = \rho_i$.
\end{enumerate}
The first condition implies that 
\begin{equation} \label{eq:conditionS}
    \left(\frac{q}{-\lambda}\right) = \left(\frac{-\lambda}{q}\right)=1 \text{ for any totally positive odd prime } q\in \cO_F \text{ with } (q) \in S
\end{equation} by \Cref{thm:quadratic_reciprocity}. The second condition implies that $\Nm_{F'/F}(P_{\lambda}(j_0)P_{4\lambda}(j_0)P_{\lambda}(j_\infty)P_{4\lambda}(j_{\infty}))$ is totally negative. Let \[N:= \lambda^{-\sum_{i=1}^{[F':F]}\#A_i}\Nm_{F'/F}(d_{\infty}^{h_1+h_2} b_{\lambda} P_{\lambda}(j_{\infty}) b_{4\lambda} P_{4\lambda}(j_{\infty})d_{0}^{h_1+h_2} b_{\lambda} P_{\lambda}(j_{0}) b_{4\lambda} P_{4\lambda}(j_0))\in \cO_F.\] Then $N$ is a square in  $\cO_{H_{\mathfrak{l}}}/\mathfrak{l}\cO_{H_{\mathfrak{l}}} = \cO_F/\lambda \cO_F$ and $N$ is totally negative. If $\lambda | N$, then \[v_{\mathfrak{l}}\left(\left(\prod_{\substack{k=1 \\ v_{\mathfrak{l}(\alpha_k)\geq 0}}}^{h_1} \prod_{i=1}^{[F':F]} (\sigma_i(j_0)-\alpha_k)\right)\left(\prod_{\substack{k=1 \\ v_{\mathfrak{l}(\beta_k)\geq 0}}}^{h_2} \prod_{i=1}^{[F':F]} (\sigma_i(j_0)-\beta_k)\right)\right)>0, \] which implies $j_0$ is a root of $b_{\lambda}b_{4\lambda} P_{\lambda}(x)P_{4\lambda}(x)$ modulo some prime above $\lambda$, where $\lambda$ is ramified in $F(\sqrt{-\lambda})$. Assume $\lambda \nmid N$ and write $(N) = \mathfrak{m}\mathfrak{n}$, where  $\mathfrak{m}$ is an integral ideal without odd prime factors and $\mathfrak{n}$ is an odd integral ideal. 
By \Cref{thm:quadratic_reciprocity}, we have \[ \left(\frac{-\lambda}{\mathfrak{n}}\right) = (-1)^{[F:\bQ]} \left(\frac{N}{-\lambda}\right) = -1 \cdot 1 = -1,\]
and therefore, there is an odd prime $\fp| \mathfrak{n}$ with $v_\fp(N)$ odd such that $\left(\frac{-\lambda}{\fp}\right) = -1$. In particular, $\fp \notin S$ by \eqref{eq:conditionS}, so that $\fp$ is a good reduction prime and $\fp$ is unramified  in $F'$. 

If $v_\fp(\Nm_{F'/F}(b_{\lambda}P_{\lambda}(j_\infty)b_{4\lambda}P_{4\lambda}(j_\infty))> 0$, then $v_\fp(\Nm_{F'/F}(b_{\lambda}P_{\lambda}(j_\infty)b_{4\lambda}P_{4\lambda}(j_\infty)))$ is even by \Cref{lem:intersection_auto}.\footnote{This treatment follows the idea of \cite{LMPT}*{\S9.5}.} 
Therefore, $j_0$ is a root of $b_{\lambda}P_{\lambda}(x) b_{4\lambda}P_{4\lambda}(x)$ modulo some prime above $\fp$.

\section{Examples}

In this section, we provide sufficient conditions for verifying some of the assumptions in \Cref{thm:main}. In particular, all the assumptions in \Cref{thm:main} hold for the cases in \Cref{thm:Mumfordfourfold}, and hence we prove \Cref{thm:Mumfordfourfold}.

\begin{lemma}
    Let $C$ be a smooth projective geometrically connected curve over a field $k$. If $C$ has genus $0$ and there exists a divisor $D$ of degree $1$, then $C\simeq \bP^1_k$. In particular, if there are divisors $D_1$ and $D_2$ such that $\deg D_1$ and $\deg D_2$ are relatively prime, then $C\simeq \mathbb{P}^1_k$. 
\end{lemma}
\begin{proof}
    By Riemann-Roch, we have \[\dim_k H^0(C, \cO_C(D)) \geq \deg D + 1 - g(C) = 2, \] then $\cO_C(D)$ is very ample and induces an isomorphism $C\simeq \bP_k^1$ (\cite{MR1917232}*{7.3.24}). 
\end{proof}

By \Cref{lem:optemb} and \Cref{lem:realCMpt}, a CM cycle defined by $\phi: L\hookrightarrow B$ defines a divisor of degree $h(\phi^{-1}(\cO))$. 
If there exist  CM extensions $L_1, L_2$ of $F$ such that $h(L_1)$ and $h(L_2)$ are relatively prime, then the canonical model of the Shimura curve $\cO^1\backslash\cH$ is isomorphic to $\bP^1_F$.

\begin{lemma}
    Let $F$ be a totally real number field with narrow class number $1$, and suppose there is a unique prime $\fp_2$ of $F$ lying above $2$. Suppose $F(\sqrt{u})$ has odd class number for all unit $u\in \cO_F^\times$. Let $M$ denote the compositum of $F(\sqrt{u})$ as  $u$ ranges over $\cO_F^\times $. Then every intermediate field $K$ with $F \subseteq K \subseteq M$ has class number $\prod_{F(\sqrt{u}) \subseteq M} h(F(\sqrt{u}))$. 
\end{lemma}
\begin{proof}
    Let $r = [F:\bQ]$ and $u_1, \dotsc, u_{r-1}$ be fundamental units of $F$, then $M = F(\sqrt{-1}, \sqrt{u_1}, \dotsc, \sqrt{u_{r-1}})$ is a multiquadratic extension of $F$ with $[M:F] = 2^r$, and any quadratic subextension of $M/F$ is of the form $F(\sqrt{u})$ for some $u\in \cO_F^\times \backslash \cO_F^{\times 2}$. 
    Any intermediate field $K$ with $F\subsetneq K \subseteq M$ is a multiquadratic extension of $F$.
    As a special case of the example in \cite{MR3069610}*{pp. 329-330}, we have $h(K) = 2^{\mu(K)}\prod_{F(\sqrt{u}) \subseteq M} h(F(\sqrt{u}))$ for some nonnegative integer $\mu(K)$. 
    
    Let $L_0 := F(\sqrt{-1})$, and $L_i:= L_{i-1}(\sqrt{u_i})$.
    Inductively, for each $i=1, \dotsc, r-1$, under the assumption that $h(L_{i-1})$ is odd, there is exactly one prime (both finite and infinite), namely the unique prime lying above $2$, of $L_{i-1}$ ramifying in the quadratic extension $L_i$, then $2\nmid h(L_{i})$ by \cite{MR963648}*{9.2}. In particular, we have $2\nmid h(M)$ and $\fp_2$ is totally ramified in $M$. 

    Let $K_0 = K$ and $K_i = KL_{i-1}$ for each $i = 1, \dotsc, r$, then $K_i = K_{i-1}$ or $K_{i}/K_{i-1}$ is a ramified quadratic extension, which implies $h(K_{i-1})$ divides $h(K_i)$ by \cite{MR963648}*{5.4}. Therefore, we have $h(K)$ odd. 
\end{proof}

When the totally real field $F$ has narrow class number $1$ and there is a unique prime of $F$ lying above $2$, given units $\epsilon_1, \dotsc, \epsilon_n$ that are negative at exactly one of the distinct real embeddings $F\xhookrightarrow{} \bR$, if we can verify that $F\left(\sqrt{(-\epsilon_1)^{\delta_1} \cdots (-\epsilon_n)^{\delta_n}}\right)$ has class number $1$ for all $\delta_1, \dotsc, \delta_n \in \{0,1\}$, then $F(\sqrt{-\epsilon_1}, \dotsc, \sqrt{-\epsilon_n})$ has class number $1$.

\begin{example} \label{examples}
    The following examples from \cite{MR2476577}*{\S4} satisfy all the assumptions. 
    \begin{itemize}
        \item $[F:\bQ] = 3$, $\disc(F) = 49, 81, 169, 321, 361, \allowbreak 473, 785, 993$
        \item $[F:\bQ] = 5$, $\disc(F) = 14641, 24217, 65657, 70601, 124817, 149169, 157457, 160801, 161121, \allowbreak 173513, \allowbreak 176281, 202817, 240881$
        \item $[F:\bQ] = 7$, $\disc(F) = 20134393, 25367689, 28118369, 31056073, 32567681, 35269513$
    \end{itemize}
    We used Magma \cite{MR1484478} to perform the computations.
\end{example}

\bibliographystyle{amsalpha}
\bibliography{references}

@article {MR2414789,
    AUTHOR = {Baba, Srinath and Granath, H\aa kan},
     TITLE = {Primes of superspecial reduction for {QM} abelian surfaces},
   JOURNAL = {Bull. Lond. Math. Soc.},
  FJOURNAL = {Bulletin of the London Mathematical Society},
    VOLUME = {40},
      YEAR = {2008},
    NUMBER = {2},
     PAGES = {311--318},
      ISSN = {0024-6093,1469-2120},
   MRCLASS = {11G18 (11G15 11G25 14G35)},
  MRNUMBER = {2414789},
MRREVIEWER = {Adrian\ Vasiu},
       DOI = {10.1112/blms/bdn008},
       URL = {https://doi.org/10.1112/blms/bdn008},
}

@book {MR491705,
    AUTHOR = {Berthelot, Pierre and Ogus, Arthur},
     TITLE = {Notes on crystalline cohomology},
 PUBLISHER = {Princeton University Press, Princeton, NJ; University of Tokyo
              Press, Tokyo},
      YEAR = {1978},
     PAGES = {vi+243},
      ISBN = {0-691-08218-9},
   MRCLASS = {14F30},
  MRNUMBER = {491705},
MRREVIEWER = {G.\ Horrocks},
}

@article {MR3843371,
    AUTHOR = {Charles, Fran\c{c}ois},
     TITLE = {Exceptional isogenies between reductions of pairs of elliptic
              curves},
   JOURNAL = {Duke Math. J.},
  FJOURNAL = {Duke Mathematical Journal},
    VOLUME = {167},
      YEAR = {2018},
    NUMBER = {11},
     PAGES = {2039--2072},
      ISSN = {0012-7094,1547-7398},
   MRCLASS = {11G05 (14G40)},
  MRNUMBER = {3843371},
MRREVIEWER = {Ariyan\ Javanpeykar},
       DOI = {10.1215/00127094-2018-0011},
       URL = {https://doi.org/10.1215/00127094-2018-0011},
}

@book {MR963648,
    AUTHOR = {Conner, P. E. and Hurrelbrink, J.},
     TITLE = {Class number parity},
    SERIES = {Series in Pure Mathematics},
    VOLUME = {8},
 PUBLISHER = {World Scientific Publishing Co., Singapore},
      YEAR = {1988},
     PAGES = {xii+234},
      ISBN = {9971-50-669-6},
   MRCLASS = {11R29 (11R11 11R27 11R34)},
  MRNUMBER = {963648},
MRREVIEWER = {H.\ Yokoi},
       DOI = {10.1142/0663},
       URL = {https://doi.org/10.1142/0663},
}

@incollection {MR2083211,
    AUTHOR = {Conrad, Brian},
     TITLE = {Gross-{Z}agier revisited},
 BOOKTITLE = {Heegner points and {R}ankin {$L$}-series},
    SERIES = {Math. Sci. Res. Inst. Publ.},
    VOLUME = {49},
     PAGES = {67--163},
      NOTE = {With an appendix by W. R. Mann},
 PUBLISHER = {Cambridge Univ. Press, Cambridge},
      YEAR = {2004},
      ISBN = {0-521-83659-X},
   MRCLASS = {11G18 (11F66 11F67 11G40)},
  MRNUMBER = {2083211},
MRREVIEWER = {Alexey\ A.\ Panchishkin},
       DOI = {10.1017/CBO9780511756375.006},
       URL = {https://doi.org/10.1017/CBO9780511756375.006},
}

@article {MR2492400,
    AUTHOR = {van Geemen, Bert},
     TITLE = {Real multiplication on {$K3$} surfaces and {K}uga-{S}atake
              varieties},
   JOURNAL = {Michigan Math. J.},
  FJOURNAL = {Michigan Mathematical Journal},
    VOLUME = {56},
      YEAR = {2008},
    NUMBER = {2},
     PAGES = {375--399},
      ISSN = {0026-2285,1945-2365},
   MRCLASS = {14J28 (11G15 14C30)},
  MRNUMBER = {2492400},
MRREVIEWER = {Trygve\ Johnsen},
       DOI = {10.1307/mmj/1224783519},
       URL = {https://doi.org/10.1307/mmj/1224783519},
}

@article {MR903384,
    AUTHOR = {Elkies, Noam D.},
     TITLE = {The existence of infinitely many supersingular primes for
              every elliptic curve over {${\bf Q}$}},
   JOURNAL = {Invent. Math.},
  FJOURNAL = {Inventiones Mathematicae},
    VOLUME = {89},
      YEAR = {1987},
    NUMBER = {3},
     PAGES = {561--567},
      ISSN = {0020-9910,1432-1297},
   MRCLASS = {11G05 (14G25)},
  MRNUMBER = {903384},
MRREVIEWER = {David\ Grant},
       DOI = {10.1007/BF01388985},
       URL = {https://doi.org/10.1007/BF01388985},
}

@article {MR1030140,
    AUTHOR = {Elkies, Noam D.},
     TITLE = {Supersingular primes for elliptic curves over real number
              fields},
   JOURNAL = {Compositio Math.},
  FJOURNAL = {Compositio Mathematica},
    VOLUME = {72},
      YEAR = {1989},
    NUMBER = {2},
     PAGES = {165--172},
      ISSN = {0010-437X,1570-5846},
   MRCLASS = {11G05},
  MRNUMBER = {1030140},
MRREVIEWER = {Joseph\ H.\ Silverman},
       URL = {http://www.numdam.org/item?id=CM_1989__72_2_165_0},
}

@article {MR1909819,
    AUTHOR = {Galluzzi, Federica},
     TITLE = {Abelian fourfold of {M}umford-type and {K}uga-{S}atake
              varieties},
   JOURNAL = {Indag. Math. (N.S.)},
  FJOURNAL = {Koninklijke Nederlandse Akademie van Wetenschappen.
              Indagationes Mathematicae. New Series},
    VOLUME = {11},
      YEAR = {2000},
    NUMBER = {4},
     PAGES = {547--560},
      ISSN = {0019-3577,1872-6100},
   MRCLASS = {14K20 (14C25)},
  MRNUMBER = {1909819},
MRREVIEWER = {V.\ Kumar\ Murty},
       DOI = {10.1016/S0019-3577(00)80024-5},
       URL = {https://doi.org/10.1016/S0019-3577(00)80024-5},
}

@book {MR638719,
    AUTHOR = {Hecke, Erich},
     TITLE = {Lectures on the theory of algebraic numbers},
    SERIES = {Graduate Texts in Mathematics},
    VOLUME = {77},
      NOTE = {Translated from the German by George U. Brauer, Jay R. Goldman
              and R. Kotzen},
 PUBLISHER = {Springer-Verlag, New York-Berlin},
      YEAR = {1981},
     PAGES = {xii+239},
      ISBN = {0-387-90595-2},
   MRCLASS = {12-01 (01A75)},
  MRNUMBER = {638719},
}

@article {MR3705249,
    AUTHOR = {Howard, Benjamin and Pappas, Georgios},
     TITLE = {Rapoport-{Z}ink spaces for spinor groups},
   JOURNAL = {Compos. Math.},
  FJOURNAL = {Compositio Mathematica},
    VOLUME = {153},
      YEAR = {2017},
    NUMBER = {5},
     PAGES = {1050--1118},
      ISSN = {0010-437X,1570-5846},
   MRCLASS = {14G35 (11G18)},
  MRNUMBER = {3705249},
MRREVIEWER = {Marc-Hubert\ Nicole},
       DOI = {10.1112/S0010437X17007011},
       URL = {https://doi.org/10.1112/S0010437X17007011},
}

@book {MR2704678,
    AUTHOR = {Jao, David Yen},
     TITLE = {Supersingular primes for rational points on modular curves},
      NOTE = {Thesis (Ph.D.)--Harvard University},
 PUBLISHER = {ProQuest LLC, Ann Arbor, MI},
      YEAR = {2003},
     PAGES = {88},
      ISBN = {978-0496-39298-8},
   MRCLASS = {99-05},
  MRNUMBER = {2704678},
       URL =
              {http://gateway.proquest.com/openurl?url_ver=Z39.88-2004&rft_val_fmt=info:ofi/fmt:kev:mtx:dissertation&res_dat=xri:pqdiss&rft_dat=xri:pqdiss:3091589},
}

@incollection {MR638600,
    AUTHOR = {Katz, N.},
     TITLE = {Serre-{T}ate local moduli},
 BOOKTITLE = {Algebraic surfaces ({O}rsay, 1976--78)},
    SERIES = {Lecture Notes in Math.},
    VOLUME = {868},
     PAGES = {138--202},
 PUBLISHER = {Springer, Berlin},
      YEAR = {1981},
      ISBN = {3-540-10842-4},
   MRCLASS = {14K10},
  MRNUMBER = {638600},
MRREVIEWER = {William\ E.\ Lang},
}

@article {MR2669706,
    AUTHOR = {Kisin, Mark},
     TITLE = {Integral models for {S}himura varieties of abelian type},
   JOURNAL = {J. Amer. Math. Soc.},
  FJOURNAL = {Journal of the American Mathematical Society},
    VOLUME = {23},
      YEAR = {2010},
    NUMBER = {4},
     PAGES = {967--1012},
      ISSN = {0894-0347,1088-6834},
   MRCLASS = {11G18 (14G35)},
  MRNUMBER = {2669706},
MRREVIEWER = {Jeffrey\ D.\ Achter},
       DOI = {10.1090/S0894-0347-10-00667-3},
       URL = {https://doi.org/10.1090/S0894-0347-10-00667-3},
}

@book {MR1282723,
    AUTHOR = {Lang, Serge},
     TITLE = {Algebraic number theory},
    SERIES = {Graduate Texts in Mathematics},
    VOLUME = {110},
   EDITION = {Second},
 PUBLISHER = {Springer-Verlag, New York},
      YEAR = {1994},
     PAGES = {xiv+357},
      ISBN = {0-387-94225-4},
   MRCLASS = {11Rxx (11-01 11-02)},
  MRNUMBER = {1282723},
MRREVIEWER = {M.\ Ram\ Murty},
       DOI = {10.1007/978-1-4612-0853-2},
       URL = {https://doi.org/10.1007/978-1-4612-0853-2},
}

@article {MR3484114,
    AUTHOR = {Madapusi Pera, Keerthi},
     TITLE = {Integral canonical models for spin {S}himura varieties},
   JOURNAL = {Compos. Math.},
  FJOURNAL = {Compositio Mathematica},
    VOLUME = {152},
      YEAR = {2016},
    NUMBER = {4},
     PAGES = {769--824},
      ISSN = {0010-437X,1570-5846},
   MRCLASS = {11G18 (14G35)},
  MRNUMBER = {3484114},
MRREVIEWER = {Shuichiro\ Takeda},
       DOI = {10.1112/S0010437X1500740X},
       URL = {https://doi.org/10.1112/S0010437X1500740X},
}

@article {MR3370622,
    AUTHOR = {Madapusi Pera, Keerthi},
     TITLE = {The {T}ate conjecture for {K}3 surfaces in odd characteristic},
   JOURNAL = {Invent. Math.},
  FJOURNAL = {Inventiones Mathematicae},
    VOLUME = {201},
      YEAR = {2015},
    NUMBER = {2},
     PAGES = {625--668},
      ISSN = {0020-9910,1432-1297},
   MRCLASS = {14J28 (11G10 14G17 14G35 14K15)},
  MRNUMBER = {3370622},
MRREVIEWER = {G.\ K.\ Sankaran},
       DOI = {10.1007/s00222-014-0557-5},
       URL = {https://doi.org/10.1007/s00222-014-0557-5},
}

@book {MR347836,
    AUTHOR = {Messing, William},
     TITLE = {The crystals associated to {B}arsotti-{T}ate groups: with
              applications to abelian schemes},
    SERIES = {Lecture Notes in Mathematics},
    VOLUME = {Vol. 264},
 PUBLISHER = {Springer-Verlag, Berlin-New York},
      YEAR = {1972},
     PAGES = {iii+190},
   MRCLASS = {14L05 (14B20 14D15 14F30 14K10)},
  MRNUMBER = {347836},
MRREVIEWER = {F.\ Oort},
}

@book {MR654325,
    AUTHOR = {Deligne, Pierre and Milne, James S. and Ogus, Arthur and Shih,
              Kuang-yen},
     TITLE = {Hodge cycles, motives, and {S}himura varieties},
    SERIES = {Lecture Notes in Mathematics},
    VOLUME = {900},
 PUBLISHER = {Springer-Verlag, Berlin-New York},
      YEAR = {1982},
     PAGES = {ii+414},
      ISBN = {3-540-11174-3},
   MRCLASS = {14Kxx (10D25 12A67 14A20 14F30 14K22)},
  MRNUMBER = {654325},
}

@book {MR772569,
    AUTHOR = {Katz, Nicholas M. and Mazur, Barry},
     TITLE = {Arithmetic moduli of elliptic curves},
    SERIES = {Annals of Mathematics Studies},
    VOLUME = {108},
 PUBLISHER = {Princeton University Press, Princeton, NJ},
      YEAR = {1985},
     PAGES = {xiv+514},
      ISBN = {0-691-08349-5; 0-691-08352-5},
   MRCLASS = {11G05 (11F11 14G25 14K15)},
  MRNUMBER = {772569},
MRREVIEWER = {Kenneth\ A.\ Ribet},
       DOI = {10.1515/9781400881710},
       URL = {https://doi.org/10.1515/9781400881710},
}

@book {MR1917232,
    AUTHOR = {Liu, Qing},
     TITLE = {Algebraic geometry and arithmetic curves},
    SERIES = {Oxford Graduate Texts in Mathematics},
    VOLUME = {6},
      NOTE = {Translated from the French by Reinie Ern\'e,
              Oxford Science Publications},
 PUBLISHER = {Oxford University Press, Oxford},
      YEAR = {2002},
     PAGES = {xvi+576},
      ISBN = {0-19-850284-2},
   MRCLASS = {14-01 (11G30 14A05 14A15 14Gxx 14Hxx)},
  MRNUMBER = {1917232},
MRREVIEWER = {C\'icero\ Carvalho},
}

@article {MR621017,
    AUTHOR = {Milne, J. S. and Shih, Kuang-yen},
     TITLE = {The action of complex conjugation on a {S}himura variety},
   JOURNAL = {Ann. of Math. (2)},
  FJOURNAL = {Annals of Mathematics. Second Series},
    VOLUME = {113},
      YEAR = {1981},
    NUMBER = {3},
     PAGES = {569--599},
      ISSN = {0003-486X},
   MRCLASS = {10D25},
  MRNUMBER = {621017},
MRREVIEWER = {Thomas\ Zink},
       DOI = {10.2307/2006998},
       URL = {https://doi.org/10.2307/2006998},
}

@article {MR2666906,
    AUTHOR = {Milne, James S. and Suh, Junecue},
     TITLE = {Nonhomeomorphic conjugates of connected {S}himura varieties},
   JOURNAL = {Amer. J. Math.},
  FJOURNAL = {American Journal of Mathematics},
    VOLUME = {132},
      YEAR = {2010},
    NUMBER = {3},
     PAGES = {731--750},
      ISSN = {0002-9327,1080-6377},
   MRCLASS = {14G35 (11G18)},
  MRNUMBER = {2666906},
       DOI = {10.1353/ajm.0.0112},
       URL = {https://doi.org/10.1353/ajm.0.0112},
}

@article {MR248146,
    AUTHOR = {Mumford, D.},
     TITLE = {A note of {S}himura's paper ``{D}iscontinuous groups and
              abelian varieties''},
   JOURNAL = {Math. Ann.},
  FJOURNAL = {Mathematische Annalen},
    VOLUME = {181},
      YEAR = {1969},
     PAGES = {345--351},
      ISSN = {0025-5831,1432-1807},
   MRCLASS = {14.55},
  MRNUMBER = {248146},
MRREVIEWER = {T.\ Kambayashi},
       DOI = {10.1007/BF01350672},
       URL = {https://doi.org/10.1007/BF01350672},
}

@article {MR3069610,
    AUTHOR = {Nehrkorn, Harald},
     TITLE = {\"Uber absolute idealklassengruppen und einheiten in
              algebraischen zahlk\"orpern},
   JOURNAL = {Abh. Math. Sem. Univ. Hamburg},
  FJOURNAL = {Abhandlungen aus dem Mathematischen Seminar der Universit\"at
              Hamburg},
    VOLUME = {9},
      YEAR = {1933},
    NUMBER = {1},
     PAGES = {318--334},
      ISSN = {0025-5858,1865-8784},
   MRCLASS = {99-04},
  MRNUMBER = {3069610},
       DOI = {10.1007/BF02940658},
       URL = {https://doi.org/10.1007/BF02940658},
}

@book {MR1697859,
    AUTHOR = {Neukirch, J\"urgen},
     TITLE = {Algebraic number theory},
    SERIES = {Grundlehren der mathematischen Wissenschaften [Fundamental
              Principles of Mathematical Sciences]},
    VOLUME = {322},
      NOTE = {Translated from the 1992 German original and with a note by
              Norbert Schappacher,
              With a foreword by G. Harder},
 PUBLISHER = {Springer-Verlag, Berlin},
      YEAR = {1999},
     PAGES = {xviii+571},
      ISBN = {3-540-65399-6},
   MRCLASS = {11Rxx (11-02 11S15 11S31 14C40)},
  MRNUMBER = {1697859},
MRREVIEWER = {Cornelius\ Greither},
       DOI = {10.1007/978-3-662-03983-0},
       URL = {https://doi.org/10.1007/978-3-662-03983-0},
}

@book {MR2705896,
    AUTHOR = {Sadykov, Marat},
     TITLE = {Two results in the arithmetic of {S}himura curves},
      NOTE = {Thesis (Ph.D.)--Columbia University},
 PUBLISHER = {ProQuest LLC, Ann Arbor, MI},
      YEAR = {2004},
     PAGES = {61},
      ISBN = {978-0496-76283-5},
   MRCLASS = {99-05},
  MRNUMBER = {2705896},
       URL =
              {http://gateway.proquest.com/openurl?url_ver=Z39.88-2004&rft_val_fmt=info:ofi/fmt:kev:mtx:dissertation&res_dat=xri:pqdiss&rft_dat=xri:pqdiss:3129023},
}

@article {MR644559,
    AUTHOR = {Serre, Jean-Pierre},
     TITLE = {Quelques applications du th\'eor\`eme de densit\'e{} de
              {C}hebotarev},
   JOURNAL = {Inst. Hautes \'Etudes Sci. Publ. Math.},
  FJOURNAL = {Institut des Hautes \'Etudes Scientifiques. Publications
              Math\'ematiques},
    NUMBER = {54},
      YEAR = {1981},
     PAGES = {323--401},
      ISSN = {0073-8301,1618-1913},
   MRCLASS = {12A75 (10D99 10H25 14G25)},
  MRNUMBER = {644559},
MRREVIEWER = {J.\ Tunnell},
       URL = {http://archive.numdam.org/article/PMIHES_1981__54__123_0.pdf},
}

@article {MR4065146,
    AUTHOR = {Shankar, Ananth N. and Tang, Yunqing},
     TITLE = {Exceptional splitting of reductions of abelian surfaces},
   JOURNAL = {Duke Math. J.},
  FJOURNAL = {Duke Mathematical Journal},
    VOLUME = {169},
      YEAR = {2020},
    NUMBER = {3},
     PAGES = {397--434},
      ISSN = {0012-7094,1547-7398},
   MRCLASS = {11G05 (11F41 11G18 14G40 14K15)},
  MRNUMBER = {4065146},
MRREVIEWER = {Ariyan\ Javanpeykar},
       DOI = {10.1215/00127094-2019-0046},
       URL = {https://doi.org/10.1215/00127094-2019-0046},
}

@article {MR4490194,
    AUTHOR = {Shankar, Ananth N. and Shankar, Arul and Tang, Yunqing and
              Tayou, Salim},
     TITLE = {Exceptional jumps of {P}icard ranks of reductions of {K}3
              surfaces over number fields},
   JOURNAL = {Forum Math. Pi},
  FJOURNAL = {Forum of Mathematics. Pi},
    VOLUME = {10},
      YEAR = {2022},
     PAGES = {Paper No. e21, 49},
      ISSN = {2050-5086},
   MRCLASS = {14J28 (11G10 11G18 14G35 14J20)},
  MRNUMBER = {4490194},
MRREVIEWER = {Sajad\ Salami},
       DOI = {10.1017/fmp.2022.14},
       URL = {https://doi.org/10.1017/fmp.2022.14},
}

@article {MR3841493,
    AUTHOR = {Shankar, Ananth N. and Tsimerman, Jacob},
     TITLE = {Unlikely intersections in finite characteristic},
   JOURNAL = {Forum Math. Sigma},
  FJOURNAL = {Forum of Mathematics. Sigma},
    VOLUME = {6},
      YEAR = {2018},
     PAGES = {Paper No. e13, 17},
      ISSN = {2050-5094},
   MRCLASS = {11G20 (14G35)},
  MRNUMBER = {3841493},
MRREVIEWER = {Takehiro\ Hasegawa},
       DOI = {10.1017/fms.2018.15},
       URL = {https://doi.org/10.1017/fms.2018.15},
}

@article {MR572971,
    AUTHOR = {Shimura, Goro},
     TITLE = {On the real points of an arithmetic quotient of a bounded
              symmetric domain},
   JOURNAL = {Math. Ann.},
  FJOURNAL = {Mathematische Annalen},
    VOLUME = {215},
      YEAR = {1975},
     PAGES = {135--164},
      ISSN = {0025-5831,1432-1807},
   MRCLASS = {14G20 (10D10 32N05)},
  MRNUMBER = {572971},
       DOI = {10.1007/BF01432692},
       URL = {https://doi.org/10.1007/BF01432692},
}

@article {MR4836458,
    AUTHOR = {Tayou, Salim},
     TITLE = {Picard rank jumps for {K}3 surfaces with bad reduction},
   JOURNAL = {Algebra Number Theory},
  FJOURNAL = {Algebra \& Number Theory},
    VOLUME = {19},
      YEAR = {2025},
    NUMBER = {1},
     PAGES = {77--112},
      ISSN = {1937-0652,1944-7833},
   MRCLASS = {11G18 (14G40 14J28)},
  MRNUMBER = {4836458},
       DOI = {10.2140/ant.2025.19.77},
       URL = {https://doi.org/10.2140/ant.2025.19.77},
}

@book {MR4279905,
    AUTHOR = {Voight, John},
     TITLE = {Quaternion algebras},
    SERIES = {Graduate Texts in Mathematics},
    VOLUME = {288},
 PUBLISHER = {Springer, Cham},
      YEAR = {2021},
     PAGES = {xxiii+885},
      ISBN = {978-3-030-56692-0; 978-3-030-56694-4},
   MRCLASS = {11R52 (11-02 11E12 11F06 16H05 16U60 20H10)},
  MRNUMBER = {4279905},
MRREVIEWER = {Juliusz\ Brzezi\'nski},
       DOI = {10.1007/978-3-030-56694-4},
       URL = {https://doi.org/10.1007/978-3-030-56694-4},
}

@article {MR2476577,
    AUTHOR = {Voight, John},
     TITLE = {Shimura curves of genus at most two},
   JOURNAL = {Math. Comp.},
  FJOURNAL = {Mathematics of Computation},
    VOLUME = {78},
      YEAR = {2009},
    NUMBER = {266},
     PAGES = {1155--1172},
      ISSN = {0025-5718,1088-6842},
   MRCLASS = {11G18 (11R52)},
  MRNUMBER = {2476577},
MRREVIEWER = {Imin\ Chen},
       DOI = {10.1090/S0025-5718-08-02163-7},
       URL = {https://doi.org/10.1090/S0025-5718-08-02163-7},
}

@book {MR4464238,
    AUTHOR = {Xu, Yujie},
     TITLE = {Normalization in the {I}ntegral {M}odels of {S}himura
              {V}arieties of {A}belian {T}ype},
      NOTE = {Thesis (Ph.D.)--Harvard University},
 PUBLISHER = {ProQuest LLC, Ann Arbor, MI},
      YEAR = {2022},
     PAGES = {97},
      ISBN = {979-8819-38318-6},
   MRCLASS = {99-05},
  MRNUMBER = {4464238},
       URL =
              {http://gateway.proquest.com/openurl?url_ver=Z39.88-2004&rft_val_fmt=info:ofi/fmt:kev:mtx:dissertation&res_dat=xri:pqm&rft_dat=xri:pqdiss:29210526},
}

@article {MR1484478,
    AUTHOR = {Bosma, Wieb and Cannon, John and Playoust, Catherine},
     TITLE = {The {M}agma algebra system. {I}. {T}he user language},
      NOTE = {Computational algebra and number theory (London, 1993)},
   JOURNAL = {J. Symbolic Comput.},
  FJOURNAL = {Journal of Symbolic Computation},
    VOLUME = {24},
      YEAR = {1997},
    NUMBER = {3-4},
     PAGES = {235--265},
      ISSN = {0747-7171},
   MRCLASS = {68Q40},
  MRNUMBER = {MR1484478},
       DOI = {10.1006/jsco.1996.0125},
       URL = {http://dx.doi.org/10.1006/jsco.1996.0125},
}

@misc{hui2025distributionsupersingularprimesabelian,
      title={On distribution of supersingular primes of abelian varieties and K3 surfaces}, 
      author={Chun-Yin Hui},
      year={2025},
      eprint={2504.08088},
      archivePrefix={arXiv},
      primaryClass={math.NT},
      url={https://arxiv.org/abs/2504.08088}, 
}

@misc{LMPT,
      title={Infinitely many primes of basic reduction for some abelian fourfolds}, 
      author={Wanlin Li and Elena Mantovan and Rachel Pries and Yunqing Tang},
      year={2025},
      eprint={2511.05322},
      archivePrefix={arXiv},
      primaryClass={math.NT},
      url={https://arxiv.org/abs/2511.05322}, 
}

\end{document}